\RequirePackage{etex}
\documentclass[11pt,reqno, dvipsnames]{amsproc}

\usepackage[sc]{mathpazo}\renewcommand{\mathbf}{\mathbold}
\normalfont
\usepackage[T1]{fontenc}
\usepackage{float}

\usepackage[margin=1in]{geometry}
\usepackage{comment}
\usepackage[small,bf]{caption}
\usepackage{scalerel,nicefrac}
\usepackage[all,cmtip]{xy}
\usepackage{tikz-cd}

\usepackage{rotating, setspace}

\usepackage{bbm}
\usepackage{listings}
\usepackage{geometry}
\usepackage{amsmath,amsthm,amssymb}
\usepackage{hyperref}
\usepackage{color}
\usepackage{mathtools}
\usepackage{tikz-cd}

\usepackage{shuffle}
\usepackage{mathrsfs}
\usepackage{multicol}
\usepackage{cleveref}
\usepackage[utf8]{inputenc}
\usepackage{shuffle}
\usepackage{cancel}
\usepackage{bigints}
\usepackage{longtable}
\usepackage{booktabs}
\usepackage{soul}
\usepackage{subcaption}
\usepackage{array}
\usepackage{multirow}

\usepackage{autonum}
\usepackage{needspace}

\hypersetup{
    colorlinks=true,
    linkcolor=blue,
    filecolor=blue,      
    }

\numberwithin{equation}{section}

\newcommand{\ps}[1]{\mkern-.25mu\mathbin{\left(\mkern-3.5mu\left({#1}\right)\mkern-3.5mu\right)}}

\newcommand{\Z}{\mathbb{Z}}
\newcommand{\N}{\mathbb{N}}

\newcommand{\R}{\mathbb{R}}

\newcommand{\E}{\mathbb{E}}
\newcommand{\F}{\mathbb{F}}

\newcommand{\bX}{\mathbf{X}}

\newcommand{\bs}{\mathbf{s}}

\newcommand{\bzero}{\mathbf{0}}
\newcommand{\bone}{\mathbf{1}}

\newcommand{\cP}{\mathcal{P}}
\newcommand{\cF}{\mathcal{F}}

\newcommand{\cI}{\mathcal{I}}

\newcommand{\cD}{\mathcal{D}}
\newcommand{\cW}{\mathcal{W}}
\newcommand{\cN}{\mathcal{N}}
\newcommand{\cM}{\mathcal{M}}
\newcommand{\cA}{\mathcal{A}}
\newcommand{\cB}{\mathcal{B}}

\newcommand{\Lip}{{\operatorname{Lip}}}
\newcommand{\var}{{\operatorname{var}}}

\newcommand{\Mat}{{\operatorname{Mat}}}
\newcommand{\LL}{{\operatorname{L}}}
\newcommand{\Sh}{{\operatorname{Sh}}}
\newcommand{\End}{{\operatorname{End}}}
\newcommand{\Id}{{\operatorname{Id}}}

\newcommand{\cube}{Q}
\newcommand{\tp}{{\operatorname{tp}}}
\newcommand{\tTheta}{\widetilde{\Theta}}

\newcommand{\andd}{\quad \text{and} \quad}

\newcommand{\p}[1]{\mathbf{#1}} 
\newcommand{\pp}[1]{\mathcal{#1}} 
\newcommand{\drive}{U} 
\newcommand{\sol}{V} 

\newcommand{\ssol}{\widetilde{\sol}} 
\newcommand{\GRP}{G\Omega} 

\newcommand{\known}[1]{{#1}^{\text{kn}}} 
\newcommand{\unknown}[1]{{#1}^{\text{unkn}}} 
\newcommand{\trueknown}{\known{\theta}_0} 

\newcommand{\samplespace}{\mathcal{S}} 
\newcommand{\sample}{\sigma} 

\newcommand*{\pmat}[1]{\begin{pmatrix}#1\end{pmatrix}}

\newcommand{\ext}[1]{\overline{#1}}

\DeclareMathOperator{\tens}{tens}

\newcommand{\param}{\theta}

\DeclareMathOperator{\J}{J}

\newtheorem{counter}{Counter}[section]
\newtheorem{lemma}[counter]{Lemma}

\newtheorem{proposition}[counter]{Proposition}
\newtheorem{theorem}[counter]{Theorem}
\newtheorem{cor}[counter]{Corollary}

\theoremstyle{definition}
\newtheorem{definition}[counter]{Definition}
\newtheorem{example}[counter]{Example}
\newtheorem{experiment}[counter]{Experiment}
\newtheorem{remark}[counter]{Remark}
\newtheorem{notation}[counter]{Notation}

\allowdisplaybreaks

\title{Path-Dependent SDEs: Solutions and Parameter Estimation}

\author{Pardis Semnani}
\email{psemnani@math.ubc.ca}
\address{Department of Mathematics, University of British Columbia, Vancouver, BC, Canada V6T 1Z2}

\author{Vincent Guan}
\email{vguan23@math.ubc.ca}
\address{Department of Mathematics, University of British Columbia, Vancouver, BC, Canada V6T 1Z2}

\author{Elina Robeva}
\email{erobeva@math.ubc.ca}
\address{Department of Mathematics, University of British Columbia, Vancouver, BC, Canada V6T 1Z2}

\author{Darrick Lee}
\email{darrick.lee@ed.ac.uk}
\address{School of Mathematics and Maxwell Institute, University of Edinburgh, Edinburgh EH9 3FD, Scotland}

\date{\today}

\begin{document}
\maketitle
\vspace{-20pt}
\begin{center}
\begin{small}
    \today
\end{small}
\end{center}

\vspace{10pt}

\begin{abstract}
We develop a consistent method for estimating the parameters of a rich class of path-dependent SDEs, called \emph{signature SDEs}, which can model general path-dependent phenomena. Path signatures are iterated integrals of a given path with the property that any sufficiently nice function of the path can be approximated by a linear functional of its signatures. This is why we model the drift and diffusion of our signature SDE as linear functions of path signatures. We provide conditions that ensure the existence and uniqueness of solutions to a general signature SDE. 
We then introduce  the Expected Signature Matching Method (ESMM) for linear signature SDEs, which enables inference of the signature-dependent drift and diffusion coefficients from observed trajectories. Furthermore, we prove that  ESMM is consistent: given sufficiently many samples and Picard iterations used by the method, the parameters estimated by the ESMM approach the true parameter with arbitrary precision. Finally, we demonstrate on a variety of empirical simulations that our ESMM accurately infers the drift and diffusion parameters from observed trajectories.
While parameter estimation is often restricted by the need for a suitable parametric model, this work makes progress toward a completely general framework for SDE parameter estimation, using signature terms to model arbitrary path-independent and path-dependent processes.\medskip
\end{abstract}

\section{Introduction}

Stochastic differential equations (SDEs) capture both random  and deterministic dynamics, making them powerful tools for modeling processes in a broad range of fields, including biology, chemistry, physics, economics, and computer vision \cite{pavliotis2014stochastic}. Due to the need for analytical and computational tractability, the SDE is commonly assumed to be Markovian, and therefore \textit{path-independent}, i.e.,  the process' evolution depends only on its current state. Formally, a \textit{path-independent} SDE obeys the form
\begin{align}
     dY_t = a(Y_t)dt + b(Y_t)dW_t,
    \label{eq: generic_path_independent_SDE}
\end{align}
where $Y_t$ is an $m$-dimensional process and $W_t$ is an $n$-dimensional Brownian motion.
Due to its importance in applications, estimating the \emph{drift} $a$ and \emph{diffusion} $b$ of an unknown SDE from sample observations has been an active area of research for several decades, and has been especially well-studied for \textit{path-independent} SDEs. 
Once a parametric drift and diffusion family is specified,  a number of approaches, including maximum likelihood estimation (MLE) \cite{pedersen1995new, beskos2006exact, pavliotis2014stochastic, kou2012multiresolution, sharrock2021parameter}, the method of moments  \cite{ogaki199317, nielsen2000parameter, papavasiliou_parameter_2011}  and kernel density estimation \cite{nickl2020nonparametric, brogat2024learning} may be used to infer the underlying parameters. \medskip

However, real-world processes may be highly non-Markovian, since the process evolution may exhibit delays, cyclic patterns, dependence on trends, and other historical dependencies. Examples include biological processes (e.g., lung function \cite{batzel2000stability}), psychological behavior \cite{gregson2014time}, population dynamics \cite{erneux2009applied}, mechanics \cite{kalmar2001subcritical}, and climate dynamics \cite{keane2017climate}. Non-Markovian dynamics are commonly known as \textit{path-dependent} processes, and can thus be modeled by \textit{path-dependent} SDEs, whose drift and diffusion coefficients may depend on the entire history of the path,
\begin{align}
    dY_t = a(Y_{[0,t]})dt + b(Y_{[0,t]})dW_t.
    \label{eq: generic_path_dependent_SDE}
\end{align}
A simple construction of a path-dependent SDE is given in \cite{manten2024signature}. First, consider the three-dimensional linear SDE,
\begin{align}
    dY_t^{(1)} &= dW_t^{(1)}\\
    dY_t^{(2)} &= Y_t^{(3)}dt + dW_t^{(2)} \label{eq:intro_example}\\
    dY_t^{(3)} &= Y_t^{(1)}dt.
\end{align}
Although this three-dimensional SDE is Markovian, we note that if the last component $Y_t^{(3)}$ is not observed, then the second component becomes dependent on the path history of the first component, yielding
\begin{align}
    dY_t^{(1)} &= dW_t^{(1)}\\
    dY_t^{(2)}&=\left(\int_0^t Y_s^{(1)} ds\right)dt + dW_t^{(2)}.
\end{align}
The above is a simple example of a distributed-delay SDE, which is a well-studied class of path-dependent SDEs. In particular, distributed-delay SDEs assume an integro-differential form, such that the drift and diffusion parameters integrate a kernel over a time interval:
\begin{align}
    dY_t = a\left(Y_t, \int_{t-\tau}^t K(t, s, Y_s)ds \right)dt + b\left(Y_t, \int_{t-\tau}^t K(t, s, Y_s)ds \right)dW_t,
    \label{eq: integro_differential_SDE}
\end{align}
where $\tau$ is the time lag, and $K$ is a bounded kernel from a parametric family. Conditions for the existence and uniqueness of solutions, as well as numerical approximations, have been studied under various assumptions \cite{buckwar2005euler,rene2017mean}. Although distributed delay SDEs are natural examples, path-dependence encompasses a much broader class of dependencies beyond those captured by delay kernels. Furthermore, existing parameter estimation methods for the drift and diffusion are typically confined to narrowly defined models and rely on strong assumptions about the delay kernel \cite{bishwal2007parameter,torkamani2013parameter} \medskip.

In this article, we develop a consistent method for estimating the parameters of a rich class of path-dependent SDEs, called \emph{signature SDEs}, which can model general path-dependent phenomena.
The \emph{path signature} of a sufficiently smooth path $Y: [0,T] \to \R^n$ is derived from an infinite sequence of iterated integrals
\begin{equation}\label{eq:intro_signature_def}
    \p{Y}_t \coloneqq (\p{Y}_t^k)_{k=0}^\infty \quad \text{where} \quad \p{Y}^k_t \coloneqq \int_{0 \leq t_1 < \ldots < t_k \leq t} dY_{t_1} \otimes \ldots \otimes dY_{t_k} \in (\R^n)^{\otimes k}.
\end{equation}
The path signature $S(Y) \coloneqq \p{Y}_T$ completely characterizes the path $Y$ up to \emph{tree-like equivalence}~\cite{chen_integration_1958,hambly_uniqueness_2010,boedihardjo_signature_2016}, i.e., it identifies the path up to reparametrizations and retracings. In fact, by appending the time parametrization and considering the signature of the augmented path $\overline{Y}_t = (t,Y_t)$, the signature becomes injective~\cite[Section 4.3]{chevyrev_signature_2022}. We leverage two fundamental properties of the signature~\cite{chevyrev_characteristic_2016, chevyrev_signature_2022,cuchiero2023global}:
\begin{itemize}
    \item \textbf{universal}: linear functionals of the path signature can approximate sufficiently nice functions on the space of paths; and
    \item \textbf{characteristic}: the expected signature $\mu \mapsto \E_{Y \sim \mu}[S(\overline{Y})]$, where $\mu$ is a probability measure on the space of paths, is injective. 
\end{itemize}

Linear signature SDEs, introduced in~\cite{cuchiero2023signature}, are Stratonovich SDEs of the form
\begin{align} \label{eq:pdsde_compact}
    dY_t = A_\theta(\p{Y}_t) dt + B_\theta(\p{Y}_t) \circ dW_t,
\end{align}
where the drift $A_\theta$ and diffusion $B_\theta$ are linear functionals of the signature, parametrized by $\theta \in \Theta$. As the signature faithfully represents the history of the path, these are path-dependent SDEs. Moreover, by the universality of the signature, general continuous path-dependent drift and diffusion terms can be approximated by $A_\theta$ and $B_\theta$. Our main contributions are twofold. \medskip

\textbf{Existence and Uniqueness of Solutions to Signature SDEs in Short Time Intervals.} 
The question of existence and uniqueness of solutions to signature SDEs is not addressed in~\cite{cuchiero2023signature}, and our first contribution is to fill this gap in the restricted setting of short time intervals. Our approach uses the theory of \emph{rough paths}~\cite{lyons_differential_2007,friz_multidimensional_2010,friz_course_2020}. While the integrals in~\eqref{eq:intro_signature_def} are well-defined if the path $Y$ is sufficiently smooth, such integrals may not exist if $Y$ is highly irregular. Instead, a \emph{rough path} is a path $Y$, together with \emph{postulated} signature terms $\p{Y}^{k}$, for $k \leq p$, where $p$ is determined by the regularity of the path. Equipped with this additional data, integrals can be defined with respect to rough paths; in fact, the \emph{Universal Limit Theorem} determines the existence and uniqueness of solutions for path-independent \emph{rough} differential equations~\cite{lyons_differential_2007}. \medskip

After presenting preliminaries in \Cref{sec:preliminaries}, in~\Cref{sec:general_existence_uniqueness} we consider a general signature controlled differential equation 
\begin{align} 
    dY_t = F(\p{Y}_t) d\p{X}_t 
\end{align}  along with its corresponding lifted equation  
\begin{align}\label{eq:intro_signature_cde} d\p{Y}_t = \p{F}(\p{Y}_t) d\p{X}_t,
\end{align}
where $\p{X}$ is a deterministic rough driving signal. 
The latter equation extends the state space of the former equation by directly incorporating the signatures $\p{Y}$ of the solution $Y$ into the differential equation. This procedure is analogous to the example in~\eqref{eq:intro_example}, where an SDE becomes path-independent by considering a larger system. Then, by restricting our attention to sufficiently bounded driving noise, we can apply the Universal Limit Theorem~\cite{lyons_differential_2007} to obtain existence and uniqueness for the solutions to these equations (see~\Cref{thm: universal limit theorem 2}), as stated informally here.

\begin{theorem}{(Informal)}
    If $\p{X}$ is an appropriately bounded rough path, then there exists a unique solution $\p{Y}$ to~\eqref{eq:intro_signature_cde}, and there exists a sequence of Picard iterations $\{\p{Y}(r)\}_{r\in \mathbb{Z}_{\geq 0}}$, where $\p{Y} = \lim_{r \to \infty} \p{Y}(r)$. 
\end{theorem}

Next, in~\Cref{sec:parametrized_signature_sde}, we specialize to the case of the linear signature SDE in~\eqref{eq:pdsde_compact} and show that for all parameters $\theta \in \Theta$, there exist uniform bounds on the driving noise such that both the solution $\p{Y}$ and the Picard estimates $\p{Y}(r)$ are well-defined.

\medskip
\textbf{Consistent Parameter Estimation.}
Our second contribution, in~\Cref{sec:esmm}, is to develop a consistent method to estimate the parameters $\theta$ of the signature SDE in~\eqref{eq:pdsde_compact}. We leverage the characteristic property of the signature to extend the \emph{Expected Signature Matching Method}, introduced in~\cite{papavasiliou_parameter_2011} for path independent SDEs with polynomial vector fields, to linear signature SDEs. We begin in~\Cref{thm: polynomial equations} by expressing the Picard iterations $\p{Y}_\theta(r)$ of the differential equation in~\eqref{eq:pdsde_compact}  as a polynomial in $\theta$, with coefficients determined by the signature $\p{X}$ of an admissible deterministic driving signal. Then, given a stochastic driving signal, the restricted expectation of the $r$th Picard iteration is a polynomial in $\theta$, given by
\begin{align}
    P_r(\theta) \coloneqq \E[\p{Y}_\theta(r) \cdot \chi_{E_\xi}],
\end{align}
where $\chi_{E_\xi}$ is the indicator function on a subset $E_\xi\subset \samplespace$ of the sample space with appropriately bounded sample driving signals. Now given a collection $\{ {Y}_{\theta_0}(\sample_i)\}_{i=1}^N$ of solutions sampled from~\eqref{eq:pdsde_compact} with respect to an unknown parameter $\theta_0 \in \Theta$, we can solve a system of polynomial equations
\begin{align} \label{eq:intro_esmm}
    P_r(\theta) = \frac{1}{N}\sum_{i=1}^N {\p{Y}}_{\theta_0}(\sample_i) \cdot \chi_{E_\xi}(\sample_i),
\end{align}
to estimate $\theta_0$. Following~\cite{papavasiliou_parameter_2011}, we call this method the \emph{Expected Signature Matching Method}. Our main result shows that this method is a consistent estimator for linear signature SDEs, stated with explicit rates, and proved in~\Cref{thm:consistency}.

\begin{theorem} {(Informal)}\label{thm:intro_consistency}
    Let $P(\theta) \coloneqq \E[\p{Y}_\theta \cdot \chi_{E_\xi}]$ denote the restricted expected signature of the solution. Suppose $P$ is differentiable at $\theta_0 \in \Theta$ with an invertible Jacobian. Then, almost surely, for all $\varepsilon>0$,  
    there exist $N_0, r_0 \in \N$ such that for $N \geq N_0$ and $r \geq r_0$, the polynomial system~\eqref{eq:intro_esmm} has a solution $\theta_{r,N} \in B_\varepsilon(\theta_0)$.
\end{theorem}
In fact, we show in~\Cref{prop:P_diff_ae} that $P$ is locally Lipschitz, and thus differentiable almost everywhere. 
Our assumptions and proof of consistency are distinct from those of~\cite{papavasiliou_parameter_2011}. In particular,~\cite{papavasiliou_parameter_2011} assumes \emph{a priori} that a unique solution to~\eqref{eq:intro_esmm} exists, while we do not make this assumption. Furthermore,~\cite{papavasiliou_parameter_2011} requires invertibility and a uniform lower bound on the Jacobian of $P_r$ for all $r \in \N$ and $\theta \in \Theta$, while our result only requires invertibility of the Jacobian of $P$ at the true parameter value $\theta_0 \in \Theta$. We also note that linear signature SDEs can model path-independent SDEs with polynomial vector fields (see~\Cref{experiment: 1 dimensional}), so~\Cref{thm:intro_consistency} can be viewed as a generalization of~\cite[Theorem 3.6]{papavasiliou_parameter_2011} in this setting; see \Cref{remark: on the other paper} and \Cref{remark: comparison of polynomials}.\medskip

We demonstrate the efficacy of our algorithm on numerical simulations in~\Cref{sec:experiments}. We then conclude in~\Cref{sec:non-identifiability} by discussing the extent of identifiability of the parameters of the signature SDE from the observed trajectories. We also provide a table of notation in~\Cref{apx:notation}.

\subsection*{Acknowledgments}
 We thank Emilio Ferrucci for discussions on the signature SDE, and James Foster for suggestions on simulations. We also thank Anastasia Papavasiliou for answering several question about her work~\cite{papavasiliou_parameter_2011}.
Pardis Semnani was supported by a Vanier Canada Graduate Scholarship. Vincent Guan was supported by an NSERC Graduate Fellowship. Elina Robeva, Pardis Semnani, and Vincent Guan were supported by a Canada CIFAR AI Chair and an NSERC Discovery Grant (DGECR-2020-00338). Part of this research was performed while Pardis Semnani was visiting the Institute for Mathematical and Statistical Innovation (IMSI), which is supported by the National Science Foundation (Grant No. DMS-1929348). Darrick Lee was supported by the Hong Kong Innovation and Technology Commission (InnoHK Project CIMDA) during part of this work.

\section{Preliminaries} \label{sec:preliminaries}

In this section, we provide some of the required background and notation on path signatures and rough paths. For further details, we refer the reader to~\cite{lyons_differential_2007,friz_multidimensional_2010}. In general, we will be working with bounded $p$-variation paths.

\begin{definition}
    Let $V$ be a Banach space with norm $\|\cdot\|$ and $p \geq 1$. Let $X:[0,T] \to V$ be continuous, i.e. $X \in C([0,T], V)$.    
    For $0 \leq s < t \leq T$, we define the the \emph{$p$-variation of $X$ on $[s,t]$} by
    \begin{align}
        |X|_{p, [s,t]} \coloneqq \left( \sup_{\cD} \sum_{i = 1} ^ {r_\cD} \left\| X_{t_i} - X_{t_{i-1}} \right\|^p\right)^{\frac 1 p},
    \end{align}
    where the supremum is taken over all partitions $\cD = \{s=t_0 < t_1 < \cdots < t_{r_\cD} = t\}$ of $[s,t]$. 
    We define the \emph{$p$-variation norm of $X$} to be
    \begin{align}
        \|X\|_{p} \coloneqq |X|_{p, [0,T]} + \|X_0\|.
    \end{align}
    The space of \emph{bounded $p$-variation paths} is
    \begin{align}
        C^{p-\var}([0,T],V) \coloneqq \left\{ X \in C([0,T], V) \, : \, \|X\|_p < \infty\right\}.
    \end{align}
\end{definition}

\subsection{Path Signatures in the Young Regime}
We begin with background on path signatures for sufficiently regular paths. 
Throughout this article, we will primarily focus on finite-dimensional Hilbert spaces $V$. Suppose we have an orthonormal basis $(e_1, \ldots, e_n)$ of $V$. This induces an orthonormal basis (and thus an inner product) on $V^{\otimes k}$, where the basis elements are
\begin{align}\label{eq: orthonormal basis}
    e_I \coloneqq e_{i_1} \otimes \ldots \otimes e_{i_k}
\end{align}
for all multi-indices $I = (i_1, \ldots, i_k)$, where $i_j \in [n]$. The \emph{tensor algebra} $T(V)$ and its completion $T\ps{V} \label{pg:tensor_alg}$ are respectively defined as the direct sum and product of all such tensor powers
\begin{align}
    T(V) \coloneqq \bigoplus_{k=0}^\infty V^{\otimes k} \andd T\ps{V} \coloneqq \prod_{k=0}^\infty V^{\otimes k}.
\end{align}
Note that the individual Hilbert space structure of $V^{\otimes k}$ does not induce a Hilbert space structure on $T\ps{V} \label{pg:hilbert_alg}$, but we may restrict to finite norm elements to obtain a Hilbert space,
\begin{align}
    H\ps{V} \coloneqq \left\{ \bs = (\bs^k)_{k=0}^\infty \in T\ps{V} \, : \, \sum_{k \geq 0} \|\bs^k\| < \infty \right\}.
\end{align}
We will also need to work with truncations of the tensor algebra, which we denote by
\begin{align}
    T^{(\leq q)}(V) \coloneqq \bigoplus_{k=0}^q V^{\otimes k}. 
\end{align}
We can now define path signatures for paths in the Young regime, with bounded $p$-variation where $p < 2$. 

\begin{definition}\label{def:signature}
    Let $p \in [1,2)$ and $X \in C^{p-\var}([0,T], V) \label{pg:sig_smooth}$. The \emph{path signature} is a map
    \begin{align}
        S: C^{p-\var}([0,T], V) \to T\ps{V},
    \end{align}
    defined by
    \begin{align}
        S(X) \coloneqq \left(S^k(X)\right)_{k=0}^\infty = \left( \int_{\Delta^k_{0,T}} dX_{t_1} \otimes \ldots \otimes dX_{t_k} \right)_{k=0}^\infty,
    \end{align}
    where $\Delta^k_{s,t} \coloneqq \{(t_1,\ldots,t_k) : s \leq t_1 < \ldots < t_k \leq t\}$ for all $s\leq t$.  
    The integral is defined as a Young integral, which is well-defined for $X \in C^{p-\var}([0,T], V)$ with $p \in [1,2)$. The component $S^k(X) \in V^{\otimes k}$ is called the \emph{level $k$ component} of the signature. The level $0$ component is $S^0(X) = 1$.
\end{definition}

In view of the rough paths setting in the following section, we will also consider the collection of signatures of a path $X \in C^{p-\var}([0,T], V)$, restricted to all subintervals. In particular, we define the map $\p{X}: \Delta_T \to T\ps{V}$, where $\Delta_T \coloneqq \Delta^2_{0,T}$ and for all $(s,t)\in \Delta_T$,
\begin{align}
    \p{X}_{s,t} \coloneqq S(X|_{[s,t]}) \coloneqq \left( \int_{\Delta^k_{s,t}} dX_{t_1} \otimes \ldots \otimes dX_{t_k} \right)_{k=0}^\infty.
\end{align}
The signature preserves the underlying concatenation structure of paths and satisfies an algebraic property called \emph{Chen's identity},
\begin{align}
    \p{X}_{s,t} \otimes \p{X}_{t,u} = \p{X}_{s,u} \quad \text{for all} \quad 0 \leq s < t < u \leq T.
\end{align} 

Following standard notation for the signature, we denote the level $k$ component by $\p{X}^k$. Given a path $X = (X^1, \ldots, X^n) \in C^{p-\var}([0,T], V)$, the path signature with respect to the multi-index $I = (i_1, \ldots, i_k) \in [n]^k$ is denoted by
\begin{align}
    \p{X}^I_{s,t} \coloneqq \int_{\Delta^k_{s,t}} dX^{i_1}_{t_1} \cdots dX^{i_k}_{t_k}. 
\end{align}
The components of the path signature satisfy the \emph{shuffle product} defined as follows. The permutation group on $k$ elements is denoted by $\Sigma_k$. For $k,\ell \in \N$, the set of \emph{$(k,\ell)$-shuffles} is defined as
\begin{align}
    \Sh(k,\ell) \coloneqq \{ \sigma \in \Sigma_{k+\ell} \, : \, \sigma^{-1}(1) < \ldots < \sigma^{-1}(k), \, \sigma^{-1}(k+1) < \ldots < \sigma^{-1}(k+\ell)\}.
\end{align}
The \emph{shuffle} of two multi-indices $I = (i_1, \ldots, i_k)$ and $J = (i_{k+1}, \ldots, i_{k+\ell})$ is defined by the multi-set
\begin{align}
    I \shuffle J \coloneqq \{ (i_{\sigma(1)}, \ldots, i_{\sigma(k+\ell)}) \, : \, \sigma \in \Sh(k,\ell)\}. 
\end{align}
The path signature satisfies the following \emph{shuffle identity},
\begin{align}
    \p{X}^I \cdot \p{X}^J = \sum_{K \in I \shuffle J} \p{X}^K. 
\end{align}
\begin{example}
    For $I=(1,2)$ and $J=(3,4)$, we have
    \begin{align}
        I \shuffle J = \{(1,2,3,4), (1,3, 2,4), (1,3,4,2), (3,1,2,4), (3,1,4,2),  (3,4,1,2)  \}.
    \end{align}
    For $I= (1)$ and $J=(1)$, we have 
    \begin{align}
        I \shuffle J = \{(1,1),(1,1)\}.
    \end{align}
\end{example}

\subsection{Rough Paths}
As we are primarily interested in differential equations driven by Brownian motion trajectories, which are almost surely bounded $p$-variation paths only for $p > 2$, we consider path signatures beyond the Young regime, where we will use the theory of rough paths~\cite{Lyons1998}. While Young integration allows us to compute signatures (and more generally, integrals) of bounded $p$-variation paths with $p<2$, we must enrich lower regularity paths with additional data in order to properly define signatures and integrals. We begin with the notion of a control, which is used to measure the regularity of paths.

\begin{definition}
    A \emph{control} is a continuous non-negative function $\omega: \Delta_T \to \R_{\geq 0}$ such that 
    \begin{align}
        \omega(t,t) = 0 \andd \omega(s,t) + \omega(t,u) \leq \omega(s,u) \quad \text{for all $s \leq t \leq u$.}
    \end{align}
\end{definition}
Now, we turn to the definition of a rough path. 
\begin{definition}
    Let $p \geq 1$. A \emph{$p$-rough path} is a function $\p{X}: \Delta_T \to T^{(\leq \lfloor p \rfloor)}(V)$ such that for all $0\leq s \leq t\leq u \leq T$, $\p{X}_{s,t}^0=1$, Chen's identity holds, i.e.,  $\p{X}_{s,t} \otimes \p{X}_{t,u} = \p{X}_{s,u}$, and it satisfies the regularity conditions
    \begin{align}
    \label{eq: regularity conditions}
        \|\p{X}^k_{s,t}\| \leq \frac{\omega(s,t)^{k/p}}{\beta (k/p)!}\quad \text{for all }  1 \le k \le \lfloor p \rfloor
    \end{align}
    for some control $\omega$ and a constant $\beta$, which only depends on $p$.\footnote{The presence or absence of constant $\beta$ and factor $(k/p)!$ in~\eqref{eq: regularity conditions} does not affect the definition of the class of $p$-rough paths. However, we choose to include them in~\eqref{eq: regularity conditions} to be consistent with the notation in \cite{lyons_differential_2007}.}
    Note that $(k/p)! \coloneqq \Gamma(k/p + 1)$ equals the Gamma function. The $p$-variation of the rough path $\p{X}$ is said to be controlled by $\omega$ if \eqref{eq: regularity conditions} holds. 
    The space of $p$-rough paths is equipped with the \emph{$p$-variation metric}
    \begin{align}
        d_p(\p{X}, \p{Y}) \coloneqq \max_{1 \leq k \leq \lfloor p \rfloor} \sup_{\cD \subset [0,T]} \left( \sum_{i=1}^{r_\cD} \|\p{X}^k_{t_{i-1}, t_i} - \p{Y}^k_{t_{i-1}, t_i}\|^{p/k}\right)^{1/p}.
    \end{align}
\end{definition}

The following extension theorem shows that path signatures are well-defined for rough paths.

\begin{theorem}{\cite[Theorem 3.7]{lyons_differential_2007}} \label{thm:extension}
    For $p \geq 1$, let $\p{X}: \Delta_T \to T^{(\leq \lfloor p \rfloor)}(V)$ be a $p$-rough path whose $p$-variation is controlled by some control $\omega$. Then, there exists a unique extension $\widetilde{\p{X}}: \Delta_T \to T\ps{V}$ of $\p{X}$ such that $\widetilde{\p{X}}^k = \p{X}^k$ for all $k \leq \lfloor p \rfloor$,  Chen's identity holds for $\widetilde{\p{X}}$, i.e. $\widetilde{\p{X}}_{s,t} \otimes \widetilde{\p{X}}_{t,u} = \widetilde{\p{X}}_{s,u}$ for all $0\leq s \leq t\leq u\leq T$, and $\widetilde{\p{X}}$ satisfies the regularity conditions
    \begin{align}
        \|\widetilde{\p{X}}^k_{s,t}\| \leq \frac{\omega(s,t)^{k/p}}{\beta (k/p)!} \quad \text{for all } k\geq 1.
    \end{align}
\end{theorem}

By uniqueness, we also denote the extension (to arbitrary levels) by $\p{X} \label{pg:sig_rough}$, and the \emph{signature} of a rough path is given by
\begin{align} \label{eq:sig_rough}
    S(\p{X})\coloneqq \p{X}_{0,T} \in T\ps{V}.
\end{align}
In this article, we will work with a class of rough paths called \emph{geometric rough paths}.

\begin{definition} \label{def:grp}
    A \emph{geometric $p$-rough path} is a $p$-rough path which is the limit of $1$-rough paths in the $p$-variation metric. The space of geometric $p$-rough paths in $V$ is denoted by $\GRP_p(V)$. 
\end{definition}

The extension of geometric rough paths still satisfies the shuffle identity.

\begin{cor}{\cite[Corollary 3.9]{cass_integration_2016}} 
    Let $\p{X}: \Delta_T \to T\ps{V}$ be the extension of a geometric $p$-rough path. Then, for multi-indices $I,J$, we have
    \begin{align}
        \p{X}^I \cdot \p{X}^J = \sum_{K \in I \shuffle J} \p{X}^K. 
    \end{align}
\end{cor}

\begin{remark}\label{remark: moving between interval and Delta}
    Throughout this article, we will interchangeably view paths as functions $X: [0,T] \to V$ and as functions $X: \Delta_T \to V$ by $X_{s,t} \coloneqq X_t - X_s$. Similarly, we can also view rough paths $\p{X} : \Delta_T \to T(V)$ as functions $\p{X} : [0,T] \to T(V)$ by $\p{X}_{t} \coloneqq \p{X}_{0,t}$. 
\end{remark}
We end this section by defining a \emph{trivial rough path}, which we will use in later sections. 
\begin{definition}\label{def: trivial rough path}
 For a Banach space $U$, the \emph{trivial rough path} in $\GRP_p(U)$ is denoted by $\p{0}_{U}$, and defined to be  $(\p{0}_U)_{s,t} \coloneqq (1,0,\ldots,0) \in T^{(\leq \lfloor p \rfloor)}(U)$ for all $(s,t)\in \Delta_T$. Note that $(1,0,\ldots,0)$ is the multiplicative identity in $T^{(\leq \lfloor p \rfloor)}(U)$. With abuse of notation, when the choice of $U$ is clear from the context, we may denote $\p{0}_U$ by $\p{0}$.
\end{definition}

\subsection{Universal Limit Theorem}
The theory of rough paths allows us to study controlled differential equations driven by highly irregular signals. In particular, for Banach spaces $U,V$, given $\p{X} \in \GRP_p(\drive)$, $F: \sol \to L(\drive, \sol)$, and $\zeta \in \sol$,  we wish to make sense of rough differential equations of the form
\begin{align} \label{eq:standard_rde}   
    dY_t = F(Y_t) d\p{X}_t, \quad Y_0 = \zeta.
\end{align} 
In order to do so, we consider the notion of $\Lip(\gamma)$ functions in the sense of~\cite[Section VI]{singular_integrals}. We state the definition for such functions on Banach spaces $U$; for the more general definition for closed sets $F \subset U$, see~\cite[Section VI.2.3]{singular_integrals} and~\cite[Definition 1.21]{lyons_differential_2007}.

\begin{definition}{\cite[Definition 10.2]{friz_multidimensional_2010}}
    Let $\gamma > 0$, and $U,V$ be Banach spaces. A function $f: U \to V$ is $\Lip(\gamma)$ if it is $\gamma_0$-times continuously differentiable, and there exists $M \geq 0$ such that
    \begin{align}
        \sup_{u \in U}\|f^{(k)}(u)\| \leq M  \text{ for all }  k = 0, \ldots, \gamma_0 \andd \sup_{u_1, u_2 \in U} \frac{\|f^{(\gamma_0)}(u_1) - f^{(\gamma_0)}(u_2)\|}{\|u_1 - u_2\|^{\gamma-\gamma_0}} \leq M,
    \end{align}
    where $f^{(k)}$ denotes the $k$th derivative of $f$, and $\gamma_0$ denotes the largest integer that is \emph{strictly} smaller than $\gamma$. The smallest such $M \geq 0$ is the $\Lip(\gamma)$ norm of $f$, denoted $\|f\|_{\Lip(\gamma)}$. 
\end{definition}

We record a basic lemma which states that the product of two $\Lip(\gamma)$ functions is still $\Lip(\gamma)$.
\begin{lemma} \label{lem:product_lipgamma}
    Suppose $U,V_1,V_2,V_3$ are Banach spaces, $U$ is finite-dimensional, and $f: U\to \LL(V_1,V_2)$ and $g: U \to \LL(V_2,V_3)$ are $\Lip(\gamma)$ functions. Then, the product $h  
    : U \to \LL(V_1,V_3)$ defined by
    \begin{align}
        h(u) \coloneqq g(u) \circ f(u) \ \text{ for all } u\in U,
    \end{align}
    is a $\Lip(\gamma)$ function. Furthermore, $\|h\|_{\Lip(\gamma)} \leq K \|f\|_{\Lip(\gamma)} \|g\|_{\Lip(\gamma)}$, where $K>0$ is a constant that only depends on $\gamma$ and the dimension of $U$.
\end{lemma}

\begin{proof}
    The proof is given in \Cref{appendix: proofs}.
\end{proof}

A key property of rough paths is that it allows us to consider integrals of $\Lip(\gamma)$ 1-forms, and we refer the reader to~\cite[Definition 4.9]{lyons_differential_2007} for details on the construction. 

\begin{theorem}{\cite[Theorem 4.12]{lyons_differential_2007}} \label{thm:rough_integral}
    Let $p \geq 1$. Suppose $f: V \to \LL(V, U)$ is a $\Lip(\gamma-1)$ function for $\gamma > p$. Then, there exists a continuous integration map $\cI_f: \GRP_p(V) \to \GRP_p(U)$ defined by
    \begin{align}
        \cI_f(\p{Z}) = \int f(\p{Z}) d\p{Z}.
    \end{align}
    If $\omega: \Delta_T \to \R_{\geq 0}$ is a control, there exists a constant $K > 0$ dependent on $\|f\|_{\Lip(\gamma)}$, $p$, $\gamma$, and $\omega(0,T)$ such that for all $\p{Z} \in \GRP_p(V)$ with $p$-variation controlled by $\omega$, we have
    \begin{align}
        \|\p{Y}^k_{s,t}\| \leq K\omega(s,t)^{k/p} \quad \text{for all } (s,t)\in \Delta_T \text{ and } k =0, \ldots, \lfloor p \rfloor,
    \end{align}
    where $\p{Y} = \cI_f(\p{Z})$. 
\end{theorem}

Now, we can use this definition of an integral to define the solution of the rough differential equation in~\eqref{eq:standard_rde}. We will denote the projection maps $\pi_\drive$ and $\pi_\sol$ from $\drive \oplus \sol$ to $\drive$ and $\sol \label{pg:projection}$ respectively, and use the same symbol for their induced maps on (truncated) tensor algebras,
\begin{align}
    \pi_\drive: T\ps{\drive \oplus \sol} \to T\ps{\drive} \andd \pi_\sol: T\ps{\drive \oplus \sol} \to T\ps{\sol}.
\end{align}

\begin{definition} \label{def:rde_sol}
    Let $1 \leq p < \gamma$, $F: \sol \to \LL(\drive, \sol)$ be a $\Lip(\gamma-1)$ function, and $\zeta\in \sol$. Define the vector field $h_F: \drive \oplus \sol \to \End(\drive \oplus \sol)$ by
    \begin{align}
        h_F(u,v) \coloneqq \pmat{\Id_\drive & 0 \\F(v+\zeta) & 0}.
    \end{align}
    We call $\p{Z} \in \GRP_p(\drive \oplus \sol)$ a \emph{coupled solution} to the rough differential equation in~\eqref{eq:standard_rde} if
    \begin{align}
        \p{Z} = \int h_F(\p{Z}) d\p{Z} \andd \pi_\drive(\p{Z}) = \p{X},
    \end{align}
    where the integral is understood in terms of~\Cref{thm:rough_integral}. In this case, we call $\p{Y} = \pi_\sol(\p{Z}) \in \GRP_p(\sol)$ the \emph{solution} to the rough differential equation in~\eqref{eq:standard_rde}.\footnote{In~\cite{lyons_differential_2007}, the rough path $\p{Z}$ is called the \emph{solution}. Here, we call $\p{Y}$ the solution in order to differentiate between various notions in the following sections.}
\end{definition}

This definition couples together the driving rough path $\p{X} = \pi_\drive(\p{Z})$ with the solution rough path $\p{Y} = \pi_\sol(\p{Z})$. In order to obtain solutions to such rough differential equations, we turn to the familiar concept of Picard iterations (though in a generalized form). We define $\p{Z}(0) \coloneqq (\p{X}, \p{0}_\sol) \in \GRP_p(\drive \oplus \sol)$. We define the \emph{Picard iterations with respect to $h_F$} recursively as
\begin{align} \label{eq:original_picard}
    \p{Z}(r+1) \coloneqq \int h_F(\p{Z}(r)) d\p{Z}(r).
\end{align}
The following \emph{Universal Limit Theorem} shows the existence and uniqueness of solutions to rough differential equations. 

\begin{theorem}{\cite[Theorem 5.3]{lyons_differential_2007}} \label{thm:universal_limit}
    Let $1 \leq p < \gamma$, $F: \sol \to \LL(\drive, \sol)$ be a $\Lip(\gamma)$ function, and $\zeta\in \sol$. For $\p{X} \in \GRP_p(\drive)$, the following hold:
    \begin{enumerate}
        \item The rough differential equation in~\eqref{eq:standard_rde} admits a unique coupled solution $\p{Z} = (\p{X}, \p{Y}) \in \GRP_p(\drive \oplus \sol)$.
        \item The map $I_{f}: \GRP_p(\drive) \times V \to \GRP_p(\sol)$ which sends $(\p{X}, \zeta)$ to $\p{Y} = \pi_{\sol}(\p{Z})$ is continuous in the $p$-variation topology.
        \item Let $\p{Z}(r)$ be the sequence of Picard iterations defined in~\eqref{eq:original_picard}, and define $\p{Y}(r) \coloneqq \pi_{\sol}(\p{Z}(r))$. The solution is given as the limit $\p{Y} = \lim_{r \to \infty} \p{Y}(r)$.
        \item Let $\omega$ be a control for the $p$-variation of $\p{X}$. For all $\rho > 1$, there exists some $T_\rho \in (0, T]$ such that
        \begin{align}
            \|\p{Y}(r)^k_{s,t} - \p{Y}(r+1)^k_{s,t}\| \leq 2^k \rho^{-r} \frac{\omega(s,t)^{k/p}}{\beta \left(\frac{k}{p}\right)!}
        \end{align}
        for all $(s,t)\in \Delta_{T_{\rho}}$ and $k = 0, \ldots \lfloor p \rfloor$. The parameter $T_\rho$ depends only on $\|f\|_{\Lip(\gamma)}$, $p$, $\gamma$, and~$\omega$. 
    \end{enumerate}
\end{theorem}

\subsection{Universal and Characteristic Properties}

While the signature of rough paths characterizes the paths up to \emph{tree-like equivalence}~\cite{boedihardjo_signature_2016}, we wish to use the signature to characterize paths without this equivalence relation. Furthermore, while rough differential equations are formulated as in~\eqref{eq:standard_rde}, we wish to study SDEs of the form~\eqref{eq:pdsde_compact}, which have both drift and diffusion terms. In order to deal with both of these issues, we will consider (rough) paths equipped with \emph{time parametrization}. We define \emph{time parametrized rough paths} to be
\begin{align}
    \GRP_p^{\tp}(\R \times V) \coloneqq \{ \overline{\p{X}} \in \GRP_p(\R \times V) \, : \, \overline{\p{X}}^1_t = (t, X_t) \}.
\end{align}

\begin{remark} \label{rem:time_parametrization}
For a rough path $\p{X} \in \GRP_p(V)$, there is a canonical way to construct a rough path $\overline{\p{X}} \in \GRP_p(\R \times V)$ such that $\overline{\p{X}}^{1}_t = (t, \p{X}^1_t)$ and $ \pi_V(\overline{\p{X}}) = \p{X}$; see~\cite[Section 9.4]{friz_multidimensional_2010}. In particular,
\begin{align}
    \iota^{\tp} : \GRP_p(V) \hookrightarrow \GRP_p^{\tp}(\R \times V), \quad \iota^{\tp}(\p{X}) = \overline{\p{X}}
\end{align}
is continuous~\cite[Theorem 9.33]{friz_multidimensional_2010}.
\end{remark}

As stated in the introduction, the signature allows us to approximate functions and characterize measures on the path space~\cite{chevyrev_characteristic_2016, chevyrev_signature_2022, cuchiero2023global}. The result from~\cite{chevyrev_signature_2022} performs a normalization on the signature such that the normalized signature is a bounded continuous map. Here, we adapt this result and remove the normalization procedure by restricting our attention to a bounded subset of the path space, which is sufficient for the purposes of this article.

\begin{theorem} \label{thm:univ_char}
    Let $\xi > 0$, and $B_\xi \coloneqq \{ \overline{\p{X}} \in \GRP_p^{\tp}(V) \, : \, d_p(\overline{\p{X}}, \bzero) < \xi\}$. 
    \begin{enumerate}
        \item The signature $S: B_{\xi} \to H\ps{\R \times V}$, defined in~\eqref{eq:sig_rough}, is a bounded continuous map. 
        \item \textbf{(Universal)} The space of linear functionals $\langle \ell, S(\cdot) \rangle : B_{\xi} \to \R$ of the signature, where $\ell \in T(\R \times V)$, is dense in continuous bounded functions $C_b(B_\xi, \R)$ equipped with the strict topology\footnote{For a topological space $A$, a function $\psi: A \to \R$ \emph{vanishes at infinity} if for all $\varepsilon > 0$, there exists a compact $K \subset A$ such that $\sup_{x \in A \setminus K} |\psi(x)| < \varepsilon$. The \emph{strict topology} on $C_b(A,\R)$ is the topology generated by the family of seminorms $p_\psi(f) = \sup_{x \in A} |f(x) \psi(x)|$ for all functions $\psi$ that vanish at infinity.} \cite{giles_generalization_1971}.
        \item \textbf{(Characteristic)} Let $\cB(B_\xi)$ denote finite regular Borel measures on $B_\xi$. The expected signature
        \begin{align}
            \E[S] : \cB(B_\xi) \to H\ps{\R \times V}, \quad \mu \mapsto \E_{\overline{\p{X}} \sim \mu} [S(\overline{\p{X}})]
        \end{align}
        is injective. 
    \end{enumerate}
\end{theorem}
\begin{proof}
    Because we consider bounded rough paths in $B_\xi$, each $\overline{\p{X}}$ has a control function $\omega$ such that $\omega(0, T) < C\xi^p$ for some constant $C>0$. Then, by~\Cref{thm:extension}, the signature satisfies the bound
    \begin{align}
        \|S(\overline{\p{X}})\| \leq \sum_{k=0}^\infty \frac{C^{k/p}\xi^k}{\beta (k/p)!} < \infty,
    \end{align}
    which does not depend on $\p{X}$. Thus $S: B_{\xi} \to H\ps{\R \times V}$ is a bounded continuous map, and linear functionals are also bounded continuous functions $\langle \ell, S(\cdot) \rangle \in C_b(B_\xi, \R)$. The remainder of the proof is identical to~\cite[Theorem 21]{chevyrev_signature_2022}.
\end{proof}

\begin{remark} \label{rem:univ_char}
    We note that more general approximation results can be found in~\cite{cuchiero2023global}, which provide universality and characteristicness results on the entire path space. This is done by considering weighted topologies, which replaces the above boundedness conditions by sufficient decay conditions on functions and measures. 
\end{remark}

\section{Path-Dependent Differential Equations} \label{sec:general_existence_uniqueness}
This section consists of the first step towards understanding the path-dependent SDEs in~\eqref{eq:pdsde_compact}, by considering a (deterministic) path-dependent rough differential equation. 
Throughout this article, assume $\drive$ and $\sol$\label{pg:U,V} are Banach spaces whose tensor powers are endowed with norms which satisfy the usual requirements of symmetry and consistency~\cite[Definition 1.25]{lyons_differential_2007}. Moreover, assume $\sol$ is finite-dimensional.
We fix $p\geq 1$, $\gamma > p$, and the natural number $q\geq \lfloor p \rfloor$ throughout\label{pg:p,gamma,q}, and assume that all rough paths are defined on $\Delta_T$ for some $T > 0$.  \medskip

We study path-dependent rough differential equations (RDEs) of the form 
\begin{align} \label{eq:original_rde}
    dY_t = F(\p{Y}_t)  d\p{X}_t,
\end{align}
where $\p{X} \in \GRP_p(\drive)$ is the driving signal as a geometric $p$-rough path, $Y: [0,T] \to \sol$ is the solution with signature denoted by $\p{Y}$, and $F$ is a $\Lip(\gamma-1)$ vector field. 
Throughout this article, we consider vector fields $F$ which depend only on the \emph{truncated} signature of $Y$ up to level $q \label{pg:ssol}$, and thus denote
\begin{align}
    \ssol \coloneqq T^{(\leq q)}(V) 
    \andd F: \ssol \to \LL(\drive, \sol).
\end{align}
In order to formally define solutions in this path-dependent context using the rough path theory discussed in the previous section, we reformulate this RDE as an ordinary path-independent RDE of the (truncated) signature of $Y$,
\begin{align} \label{eq:reformulated_rde}
    d\p{Y}_t = \p{F}(\p{Y}_t) d\p{X}_t \qquad \p{Y}_0 = \bone,
\end{align}
where we call $\p{F}:  \ssol \to \LL(\drive, \ssol)$ the \emph{lifted vector field of $F$}, and $\bone \coloneqq (1, 0, \ldots, 0) \in \ssol \label{eq:multiplicative identity}$ is the multiplicative identity in $\ssol$. \medskip

To begin, let's consider the case of a bounded $1$-variation driving signal, $X$. Our aim is to define the lifting of $F$ to $\p{F}$. In particular, we express $\p{F} = (\p{F}^0, \p{F}^1, \ldots, \p{F}^q)$, where $\p{F}^k : \ssol \to \LL(U, \sol^{\otimes k})$, and note that $\p{F}^0 = 0$ and $\p{F}^1 = F$. Then for $k=2,\ldots,q$, the level $k$ path signature of $Y$ satisfies
\begin{align}
    d\p{Y}^k_t &= \p{Y}^{k-1}_t \otimes d\p{Y}^1_t = \p{Y}^{k-1}_t \otimes F(\p{Y}_t) dX_t, \label{eq:expanded_pd}
\end{align}
and therefore suggests the definition
\begin{align} \label{eq:pF_def}
    \p{F}^k(\p{s})(x) = \p{s}^{k-1} \otimes F(\p{s})(x) \in \sol^{\otimes k},
\end{align}
where $\p{s} = (\p{s}^0, \ldots, \p{s}^q) \in \ssol$ and $x \in \drive$. The aim is to define solutions of~\eqref{eq:original_rde} as solutions of~\eqref{eq:reformulated_rde} using~\Cref{def:rde_sol}. However, an immediate issue arises: while $F$ may be a $\Lip(\gamma-1)$ function, $\p{F}$ is not $\Lip(\gamma-1)$ in general for any $\gamma > 1$; even if $F$ has bounded derivatives in $\p{s}$,  the vector field $\p{F}$ might still have unbounded derivatives. \medskip

In this section, we will consider a modification of the lifted vector field $\p{F}$ in order to maintain the $\Lip(\gamma-1)$ condition. We can then apply the Universal Limit Theorem from~\Cref{thm:universal_limit} to obtain an existence and uniqueness result for such path-dependent RDEs. 

\subsection{Reformulation with Modified Vector Field}

The main idea in reformulating the vector field is to restrict $\p{F}$ to a closed subset on which it is $\Lip(\gamma-1)$, and use the following extension theorem to extend it back to the whole space as a $\Lip(\gamma-1)$ function.

\begin{theorem}{\cite[Section VI.2, Theorem 4]{singular_integrals}} \label{thm:modify_lipgamma}
    Let $\alpha>0$. Let $U, V$ be Banach spaces, where $V$ is finite-dimensional, and $K \subset V$ be a closed subset. Let $f:K \to U$ be a $\Lip(\alpha)$ function. Then there exists an extension $\ext{f}: V \to U$ of $f$ such that $\ext{f}$ is a $\Lip(\alpha)$ function on $V$. Furthermore, there is a constant $C\in \mathbb{R}$, independent of the choice of $K$, such that for all $\Lip(\alpha)$ functions $f:K\to U$, we have $\|\ext{f}\|_{\Lip(\alpha)} \leq C \|f\|_{\Lip(\alpha)}$.
\end{theorem}

We will apply this result by first factoring the vector field $\p{F}: \ssol \to \LL(\drive, \ssol)$ as defined in~\eqref{eq:pF_def} into two components. We define the map $\tens : \ssol \to \LL(\sol, \ssol)$ to be the truncated tensor product in $\ssol$; in particular, for $\p{s} = (\p{s}^0, \ldots, \p{s}^{q})\in \ssol$ and $v \in \sol \label{pg:tens}$, we have 
\begin{align}
    \tens(\p{s})(v) \coloneqq (1, \p{s}^1, \ldots, \p{s}^{q-1}) \otimes v \in \ssol.
\end{align}
Then, we can express the lifted vector field $\p{F}$ as 
\begin{align}
    \p{F} (\p{s})\coloneqq \tens(\p{s}) \circ F(\p{s}).
\end{align}
for all $\p{s}\in \ssol$.  Note that $\tens$ is an unbounded function, but its restriction to a bounded subset is a $\Lip(\alpha)$ function for any $\alpha>0$ (i.e.~a bounded Lipschitz function). Thus, we use~\Cref{thm:modify_lipgamma} to extend this restriction to a $\Lip(\alpha)$ function on $\ssol$.

\begin{definition} \label{def:modified_tensor_mult}
    Let $M>0$. 
    Set $K_M \coloneqq \{\p{s} \in \ssol \, : \, \|\p{s}\| \leq M \} \subset \ssol$.
    We define $\tens_{M} : \ssol \to \LL(\sol, \ssol)$ to be the extension from~\Cref{thm:modify_lipgamma} of the restriction $\tens|_{K_M} : K_M \to \LL(\sol, \ssol)$ to $K_M \subset \ssol$. The map $\tens_M$ is a $\Lip(\alpha)$ function for any $\alpha>0$.
\end{definition}
Now, by applying~\Cref{lem:product_lipgamma}, we obtain a $\Lip(\gamma-1)$ modification of $\p{F}$.

\begin{cor} \label{cor:general_lifted_vf}
    Let $F: \ssol \to \LL(\drive,\sol)$ be a $\Lip(\alpha)$ function for some $\alpha>0$. For $M>0$, we define $\p{F}_M : \ssol \to \LL(\drive, \ssol)$ to be
    \begin{align} \label{eq:pFM_vector_field}
        \p{F}_M (\p{s}) \coloneqq \tens_M(\p{s}) \circ F(\p{s})
    \end{align}
    for all $\p{s}\in \ssol$. Then the vector field $\p{F}_M$ is $\Lip(\alpha)$. 
\end{cor}

We can now define one notion of a solution to the path-dependent RDE in~\eqref{eq:original_rde} in terms of the modified vector field. 

\begin{definition} \label{def:M_solution}
    Let $M > 0$, and consider $\p{F}_M: \ssol \to \LL(\drive, \ssol)$ from~\eqref{eq:pFM_vector_field}.  Define the vector field $h_M: \drive \oplus \ssol \to \End(\drive \oplus \ssol)$ by
    \begin{align} \label{eq:h_m_def}
        h_M(u,\p{s}) \coloneqq \pmat{\Id_\drive & 0 \\\p{F}_M(\p{s}+\bone) & 0}.
    \end{align}
    for all $u \in \drive$ and $\p{s}\in \ssol$. We call $\pp{Z} \in \GRP_p(\drive \oplus \ssol)$ a \emph{coupled $M$-solution} to the differential equation in~\eqref{eq:original_rde} if
    \begin{align}
        \pp{Z} = \int h_M(\pp{Z}) d\pp{Z} \andd \quad \pi_U(\pp{Z}) = \p{X}.
    \end{align}
    In this case, we call $\pp{Y} \coloneqq \pi_{\ssol}(\pp{Z}) \in \GRP_p(\ssol)$ a \emph{lifted $M$-solution}.
    Moreover, we define the \emph{$M$-solution} to be $\p{Y} \coloneqq \pp{Y}^1 + \bone$, where $\pp{Y}^1$ the level 1 component of $\pp{Y}$. Note that $\p{Y}_{0,\cdot} : [0,T] \to \ssol$ is a bounded $p$-variation path, i.e. $\p{Y}_{0,\cdot} \in  C^{p-\var}([0,T], \ssol)$. 
\end{definition}

We will further discuss the interpretation of these solutions in the following section, but we first note that the Universal Limit Theorem,~\Cref{thm:universal_limit}, can be directly applied to this setting if $F:\ssol \to \LL(\drive, \sol)$ is a $\Lip(\gamma)$ function. Therefore, we can define the function $I_{F,M}$ as follows. 
\begin{definition} \label{def:I_FM}
    Let $F: \ssol \to \LL(\drive,\sol)$ be a $\Lip(\gamma)$ function. For the path-dependent differential equation \eqref{eq:original_rde}, define $I_{F,M}:\GRP(U) \to \GRP(\ssol)$ to be the continuous function which takes the driving signal $\p{X}\in \GRP(U)$ to the corresponding lifted $M$-solution $\pp{Y}\in \GRP(\ssol)$. Note that by part (2) of \Cref{thm:universal_limit}, $I_{F,M}$ is well-defined.
\end{definition}

\begin{definition}[Picard iterations]
    Let $M>0$ and $\p{X} \in \GRP_p(\drive)$. Define $\pp{Z}(0) \coloneqq (\p{X}, \mathbf{0}_{\ssol}) \in \GRP_p(\drive \oplus \ssol)$ (see \Cref{def: trivial rough path}). For all $r\in \mathbb{N}$, define 
    \begin{align} \label{eq:pd_picard}
        \pp{Z}(r) \coloneqq \int h_M(\pp{Z}(r-1)) d\pp{Z}(r-1)
    \end{align}
    using the rough integral in~\Cref{thm:rough_integral}. The sequence $\{\pp{Z}(r)\}_{r\in \mathbb{Z}_{\geq 0}}$ is called the sequence of \textit{$M$-Picard iterations} of \eqref{eq:original_rde}.
\end{definition}

We note that the definition of the $M$-Picard iterations is exactly the definition used in the Universal Limit Theorem,~\Cref{thm:universal_limit}. Thus, the Universal Limit Theorem can be directly applied to show that unique lifted $M$-solutions exist as the limit of the projections of the $M$-Picard iterations on $\ssol$.

\subsection{The Solution as a Geometric Rough Path}
Using the Universal Limit Theorem in~\Cref{thm:universal_limit} provides a lifted $M$-solution to~\eqref{eq:original_rde} as a geometric $p$-rough path $\pp{Y} \in \GRP_p(\ssol)$. In this section, we discuss when the underlying $M$-solution is a geometric $p$-rough path in $\GRP_p(\sol)$. 
Before we continue, we will briefly discuss some notational conventions.

\begin{notation}
    We denote signatures of signatures or rough paths valued in $T(\ssol)$ using calligraphic symbols $\pp{Y}$. Integer superscripts will continue to denote the \emph{outer level} of the signature, for instance $\pp{Y}^k \in \ssol^{\otimes k}$. Given a basis of $V \cong \R^n$, we obtain a basis of $\ssol$ indexed by multi-indices $\{I = (i_1, \ldots, i_k)\}_{k=0}^q$ with $i_j \in [n]$, with
    \begin{align}
        e_I \coloneqq e_{i_1} \otimes \ldots \otimes e_{i_k}. 
    \end{align}
    We will denote the \emph{outer tensor product} of $T(\ssol)$ by $\boxtimes$, and we obtain a basis of $T(\ssol)$ by multi-indices $\{\cI = (I_1, \ldots, I_k)\}^\infty_{k=0}$ where each $I_j$ is a multi-index valued in $[n]$, defined by
    \begin{align}
        e_\cI \coloneqq e_{I_1} \boxtimes \ldots \boxtimes e_{I_k}. 
    \end{align}
    Superscripts using such multi-indices $\cI$ will denote the component $\pp{Y}^{\cI}$ of $\pp{Y}$. We denote the empty multi-index corresponding to $k=0$ with $\emptyset$.
\end{notation}

In order to show that $\p{Y}$ is a geometric $p$-rough path, we begin by relating various components of $\pp{Y}$ and $\p{Y}$. 

\begin{proposition}\label{prop: signatures of the lift}
    Assume $F: \ssol \to \LL(\drive,\sol)$ is a $\Lip(\gamma)$ function, $M>0$, and $\p{X}\in \GRP_p(\drive)$ such that $\pp{Y} = I_{F,M}(\p{X})$ is the lifted $M$-solution and $\p{Y}$ is the underlying $M$-solution to \eqref{eq:original_rde}.
    If $\|\p{Y}_{0,t}\| <M$ for all $t\in [0,T]$, then for any $(s,t)\in \Delta_T$, $\p{Y}_{s,t}^{\emptyset} = 1$, and for any $i_1,\ldots,i_k \in [\dim \sol]$ with $1\leq k\leq q$, 
    \begin{align} \label{eq: signatures of the lift}
    \pp{Y}_{s,t}^{((i_1),\ldots,(i_k))} = \pp{Y}_{s,t}^{((i_1,\ldots,i_k))} = \p{Y}_{s,t}^{(i_1,\ldots,i_k)}.
    \end{align}
    In particular, $\p{Y} = \pi_V(\pp{Y})$. 
\end{proposition}
\begin{proof} 
The proof is given in \Cref{appendix: proofs}.
\end{proof}

This result shows that under a certain condition on the  $M$-solution $\p{Y}$, it coincides with the projection of the lifted $M$-solution $\pp{Y}$  onto $\sol$. 
In particular, this justifies why $\p{Y}$ can be interpreted as a geometric $p$-rough path. 
Now, we can state a reformulation of the Universal Limit Theorem from~\Cref{thm:universal_limit} in terms of the $M$-solutions.
\begin{theorem} \label{thm: universal limit theorem 2}
    Assume $F: \ssol \to {\rm L}(\drive,\sol)$ is a $\Lip(\gamma)$ function, and $M>0$. Define
    \begin{align}
        G_M\coloneqq \{\p{X}\in \GRP_p(\drive) \, : \,  \|I_{F,M}(\p{X})_{0,t}^1 + \bone\| < M \text{ for all } t\in [0,T]\}.
    \end{align}
    Then, the following hold:
    \begin{enumerate}
        \item For any $\p{X} \in G_M$, the unique $M$-solution $\p{Y}$ to \eqref{eq:original_rde} is a geometric $p$-rough path, i.e. $\p{Y} \in \GRP_p(\sol)$. In particular, we will call $\p{Y}$ a \emph{solution}\footnote{Here, we omit the reference to $M$, as we have restricted the driving noise such that~\Cref{prop: signatures of the lift} holds, and the first level of $\pp{Y}$ is correctly interpreted as the signature of some underlying path $\p{Y}^1$. This definition of solution should be viewed as the path-dependent analogue of a solution to an ordinary RDE in~\Cref{def:rde_sol}.} to~\eqref{eq:original_rde}.
        \item The map  $J_{F,M}: G_M \to \GRP_p(\sol)\label{pg:J_FM}$, which sends $\p{X}$ to $\p{Y}$, is continuous in the $p$-variation topology. 
        \item
        For any $\p{X}\in \GRP_p(\drive)$, let $\{\pp{Z}(r)\}_{r\in \Z_{\geq 0}}$ be the sequence of $M$-Picard iterations of~\eqref{eq:original_rde}. For all $r \in \Z_{\geq 0}$, define $\pp{Y}(r) \coloneqq \pi_{\ssol}(\pp{Z}(r))$ and  $\p{Y}(r) \coloneqq \pp{Y}(r)^1 + \bone$.
        Then $\p{Y} = \lim_{r \to \infty} \p{Y}(r)$ in $p$-variation norm, where $\p{Y}(r)_{0,\cdot}$ and $\p{Y}_{0,\cdot}$ are considered as paths in $C^{p-\var}([0,T], \ssol)$. 
        \item 
        Let $\omega$ be a control of the $p$-variation of $\p{X} \in \GRP_p(\drive)$. For all $\rho>1$, there exists $T_\rho\in (0,T]$ such that for all $n\in \mathbb{Z}_{\geq 0}$,
            \begin{align}
            \left\| \p{Y}(r)_{s,t} - \p{Y}(r+1)_{s,t} \right\| \leq 2 \rho^{-r} \frac{\omega(s,t)^{\frac{1}{p}}}{\beta \left(\frac{1}{p}\right)!} \ 
            \text{ for all $(s,t)\in \Delta_{T_\rho}$},
        \end{align}
        where $T_\rho$ depends on $M$, $\|F\|_{\Lip(\gamma)}$, $\dim \sol$, $p$, $\gamma$, and $\omega$.
    \end{enumerate}
\end{theorem}
\begin{proof}
        Part (1) is straightforward since by \Cref{prop: signatures of the lift}, $\p{Y} = \pi_\sol(\pp{Y})$ is the projection of a geometric rough path $\pp{Y} \in \GRP_p(\ssol)$. Part (2) is also immediate from the original Universal Limit Theorem in~\Cref{thm:universal_limit} and the fact that on $G_M$, we have $J_{F,M} = \pi_V \circ I_{F,M}$. For part (3), we note that by~\Cref{thm:universal_limit}, we get that $I_{F,M}(\p{X}) = \lim_{r \to \infty} \pp{Y}(r)$ in the $p$-variation rough path topology. Then, restricting this to the level 1 component of $\pp{Y}$ and $\pp{Y}(r)$, we obtain the desired result. 
        Finally, part (4) also follows directly from the analogous result in part (4) of~\Cref{thm:universal_limit} and considering the level 1 component. 
\end{proof}

\begin{remark}
    An important point in this result is that $\p{Y}(r) \coloneqq \pp{Y}(r)^1 + \bone$ in parts (3) and (4) is \emph{not} a geometric $p$-rough path in general; in fact, it may not even be multiplicative. Part (1) shows that in the limit, $\p{Y}$ coincides with $\pi_\sol(\pp{Y}) = \pi_{\sol}(\pp{Z})$, but this is not true at finite Picard iterations. Thus, we treat $\p{Y}(r)$ as a bounded $p$-variation path in $C^{p-\var}([0,T], \ssol)$. 
\end{remark}

\begin{remark}\label{rem: stochastic interpretation}
    Let $(\samplespace, \mathbf{\mathcal{F}}, \mathcal{P}) \label{pg:prob_space}$ be a probability space, and $\GRP_p(\drive)$ be equipped with the Borel $\sigma$-field. When the driving noise $\p{X}:\samplespace \to \GRP_p(\drive)$ of the RDE is a stochastic process, the $M$-solution is also a stochastic process. 
    Indeed, by \Cref{remark: moving between interval and Delta}, the map $J_{F,M}$ can be extended to a continuous map $J_{F,M} : \GRP_p(U) \to C^{p-\var}([0,T], \ssol)$, where $C^{p-\var}([0,T], \ssol)$ is equipped with the $p$-variation norm.
    Then given the Borel $\sigma$-field on  $C^{p-\var}([0,T], \ssol)$, the map $J_{F,M} \circ \p{X}$ is measurable. 
\end{remark}

\section{Parametrized Signature SDEs} \label{sec:parametrized_signature_sde}

In this section, we will consider a specific model of path-dependent rough differential equations, where the underlying vector field $F: \ssol \to \LL(\drive, \sol)$ is \emph{affine} with respect to the signature of the solution. 
This choice of vector field is motivated by the universal approximation property of signatures; see~\Cref{thm:univ_char} and~\Cref{rem:univ_char}.
To leverage this property, we design our RDE so that the underlying path in the solution is $\overline{Y}_t = (t,Y_t)$ and the vector field is a linear functional of the truncated signatures of $\overline{Y}$. To simplify the notation, we will denote all parametrized paths without the overline in the remainder of this article.\medskip

Suppose that $\p{W} \in \GRP_p(\mathbb{R}^n)$ is a deterministic rough path, and $A_\theta : \ssol \to \sol$ and $B_\theta \in \LL(\ssol, \LL(\mathbb{R}^n, \sol))$ are vector fields parametrized by $\theta \in \Theta$, where $\Theta$ is the parameter space. We wish to consider the solution $\p{Y}  \in \GRP_p(\sol)$ to the rough differential equation
\begin{align} \label{eq:parametrizd_pdrde}
    dY_t = A_\theta(\p{Y}_t) dt + B_\theta(\p{Y}_t)d\p{W}_t.
\end{align}
The vector fields are affine functionals of the components of $\p{Y}$, and we assume $\sol  \cong \R^{m+1}$ and $\drive \cong \R^{n+1}$, equipped with their standard bases, for the rest of the section. We can express the vector fields as
\begin{align} \label{eq:AB_vector_fields}
    A_\theta(\p{Y}_t) = \pmat{1 \\ \langle \theta_{1,0}, \p{Y}_t \rangle\\ \vdots \\\langle \theta_{m,0}, \p{Y}_t \rangle} \andd B_\theta(\p{Y}_t) = \pmat{0 & \cdots & 0 \\\langle \theta_{1,1}, \p{Y}_t\rangle & \cdots & \langle \theta_{1,n}, \p{Y}_t\rangle \\ \vdots &\ddots& \vdots \\ \langle \theta_{m,1}, \p{Y}_t\rangle & \cdots & \langle \theta_{m,n}, \p{Y}_t \rangle },
\end{align}
where $\theta_{i,j} \in \ssol \label{pg:parametrized_vf}$ for $i = 1, \ldots , m$, and $j = 0, \ldots n$. Thus, the parameter space is $\Theta \coloneqq\{(\theta_{i,j})_{i\in[m],j\in[n]_0} : \theta_{i,j} \in \ssol\} =  \Mat_{m, n+1}(\ssol)$. \medskip

To simplify our notation such that it coincides with the notation used in the previous section, we define $\p{X} = (t, \p{W}) \in \GRP_p(\drive)$. More precisely, $\p{X}$ is a canonical $p$-rough path lift of the path $(t, W_{0,t}^1)$; see~\Cref{rem:time_parametrization}. 
Therefore, we can express the above differential equation as
\begin{align} \label{eq:F_theta_def}
    dY_t = F_{\theta}(\p{Y}_t) d\p{X}_t \quad \text{where} \quad 
    F_{\theta}(\p{Y}_t) \coloneqq \pmat{A_\theta(\p{Y}_t) & B_\theta(\p{Y}_t)}.
\end{align}
In this case, the vector field $F_\theta$ is unbounded because it is affine. We will modify this vector field in the same way as in~\Cref{def:modified_tensor_mult}. In particular, for $M > 0$, we set $K_M \coloneqq \{\p{s} \in \ssol \, : \, \|s\| \leq M\} \subset \ssol$ as before, and define $F_{\theta, M} : \ssol \to \LL(\drive, \sol)$ to be the extension from~\Cref{thm:modify_lipgamma} of the restriction $F_\theta|_{K_M} : K_M \to \LL(\drive, \sol)$ to $K_M \subset \ssol$. Thus, we in fact consider the path-dependent RDE 
\begin{align}
    dY_t = F_{\theta, M}(\p{Y}_t) d\p{X}_t,
\end{align}
where we define the corresponding lifted vector field $\p{F}_{\theta, M} : \ssol \to \LL(\drive, \ssol)$ as
\begin{align} \label{eq:bounded_lifted_vf}
    \p{F}_{\theta,M} (\p{s}) \coloneqq \tens_M(\p{s}) \circ F_{\theta, M} (\p{s}) \quad
    \text{for all $\p{s}\in \ssol$},
\end{align}
with $\tens_M$ defined in~\Cref{def:modified_tensor_mult}. 
Note that this definition ensures that $M$-solutions with respect to $\p{F}_{\theta,M}$ will not exhibit blow-ups at finite time. However, in order to ensure that the $M$-solutions are solutions in the sense of~\Cref{thm: universal limit theorem 2}, we consider sufficiently bounded driving signals. The following result shows that we can place a uniform bound on the driving signals to obtain a uniform bound on the solutions and the corresponding Picard iterations. 
For the remainder of the section, our convention is to index the coordinates of $\sol = \R^{m+1}$ and $\drive = \R^{n+1}$ with $[m]_0\coloneqq \{0,\ldots,m\}$ and $[n]_0 \coloneqq \{0,\ldots,n\}$ respectively.

\begin{proposition} \label{prop: N(theta) and M(theta)}
    For all $\theta \in \Theta$ and $M> 0$, let $\p{Y}_{\theta, M}$ be the $M$-solution, and $\{\pp{Z}_{\theta, M}(r)\}_{r\in \mathbb{Z}_{\geq 0}}$ be the sequence of $M$-Picard iterations of the path-dependent RDE
    \begin{align}
        dY_t = F_{\theta, M}(\p{Y}_t) d\p{X}_t.
    \end{align}
    Define $\pp{Y}_{\theta, M}(r) \coloneqq \pi_{\ssol}(\pp{Z}_{\theta, M}(r))$, and $\p{Y}_{\theta, M}(r)\coloneqq \pp{Y}_{\theta,M}(r)^1 + \bone$ for all $r\in \mathbb{Z}_{\geq 0}$. Let 
    \begin{align} \label{eq: rho and C}
        \rho > 1 \text{ and } C \geq \max_{1\leq i \leq \lfloor p \rfloor} \left(\beta \left(\frac{i}{p}\right)!\right)^\frac{p}{i}. 
    \end{align}
    There exist functions $\cN, \cM : \R_+ \to \R_+$ such that if $d_p(\p{X}, \bzero_\drive) \leq \cN(\|\theta\|)$, then for all $(s,t) \in \Delta_T$ and $r \in \Z_{\geq 0}$, we have
    \begin{align}
        &\left\| \p{Y}_{\theta,\cM(\|\theta\|)}(r)_{s,t} - \p{Y}_{\theta,\cM(\|\theta\|)}(r+1)_{s,t} \right\| \leq 2 \rho^{-r} \frac{C^{\frac{1}{p}} d_p(\p{X}, \mathbf{0}_\drive)}{\beta\left(\frac{1}{p}\right)!}, 
            \\
    &\left\|\p{Y}_{\theta,\cM(\|\theta\|)}(r)_{s,t}\right\| < \cM(\|\theta\|), \text{ and } \left\|(\p{Y}_{\theta, \cM(\|\theta\|)})_{s,t}\right\|< \cM(\|\theta\|).
    \end{align}
    The functions $\cM, \cN : \R_+ \to \R_+$ depend only on $p$, $\gamma$, $\rho$, $C$ and $\sol$. Furthermore, $\cN$ and $\cM$ are non-increasing.
\end{proposition}

\begin{proof}
The proof is given in \Cref{appendix: proofs}.
\end{proof}

Throughout the remainder of this paper, we fix 
\begin{align} \label{eq: C and rho}
    \rho> 1  \text{ and } C =  \max_{1\leq i \leq \lfloor p \rfloor} \left(\beta \left(\frac{i}{p}\right)!\right)^\frac{p}{i}.
\end{align}
We assume $\p{X}: \samplespace \to \ \GRP_p(\drive)$ is a stochastic process, and we have observed trajectories of the $\cM(\|\theta\|)$-solutions of the path-dependent SDE 
\begin{align}\label{eq: parameterized SDE}
    dY_t = F_{\theta,\cM(\|\theta\|)}(\p{Y}_t) \ d\p{X}_t
\end{align}
for some (partially) unknown parameter $\theta \in \Theta$.  In other words, samples of the stochastic process
$\p{Y}_{\theta} \coloneqq J_{F_{\theta,\cM(\|\theta\|)}, \cM(\|\theta\|)}\circ \p{X}$ have been observed; see \Cref{rem: stochastic interpretation}.
Our goal is to estimate the parameter $\theta$.
Moreover, for all $\theta\in \Theta$,
we let $\{ \pp{Z}_{\theta}(r) \}_{r\in \mathbb{Z}_{\geq 0}}$ be the sequence of the $\cM(\|\theta\|)$-Picard iterations of \eqref{eq: parameterized SDE}, and define $\pp{Y}_{\theta} (r) \coloneqq \pi_{\ssol}(\pp{Z}_\theta(r))$ and $\p{Y}_{\theta}(r) \coloneqq \pp{Y}_\theta(r)^1 + \bone$ for all $r \in \mathbb{Z}_{\geq 0}$.  
\section{Expected Signature Matching Method} \label{sec:esmm}

In this section, we turn to our main problem of interest: estimating the parameters of \emph{path-dependent} stochastic differential equations. Our methods are a generalization of~\cite{papavasiliou_parameter_2011}, which studies parameter estimation of \emph{path-independent} rough differential equations using a moment-matching approach, and we will briefly review their approach. \medskip

In~\cite{papavasiliou_parameter_2011}, the authors consider \emph{path-independent} stochastic differential equations of the form
\begin{align} \label{eq:path_ind_sde}
    dY_t = a_\theta(Y_t) dt + b_\theta(Y_t) \circ d\p{W}_t,
\end{align}
where 
$\p{W}$ is an $n$-dimensional Brownian motion, and $a_\theta: \sol \to \sol$ and $b_\theta: \sol \to \LL(\mathbb{R}^n, \sol)$ are path-independent polynomial vector fields parametrized by $\theta$. Suppose we observe $N$ trajectories $\{\p{Y}_{\theta_0}(\sample_i)\}_{i=1}^N \subset  \GRP_p(\sol)$, $\sample_i\in \samplespace$, sampled from the solutions to the SDE above at an unknown parameter $\theta_0 \in\Theta$. The aim is to compare the theoretical expected signature $ 
\E[\p{Y}_{\theta}]$ of the solution of~\eqref{eq:path_ind_sde} at parameter $\theta$, with the empirical expected signature
$
    \frac{1}{N} \sum_{i=1}^N \p{Y}_{\theta_0}(\sample_i)
$
in order to estimate the parameter $\theta_0\in \Theta$. While an explicit form of $\E[\p{Y}_\theta]$ is difficult to obtain in general,~\cite{papavasiliou_parameter_2011} uses the expected signature of Picard iterations of \eqref{eq:path_ind_sde},
$\E[\p{Y}_\theta(r)]$,
as an approximation, and finds that $\E[\p{Y}_\theta(r)]$ can be expressed as a polynomial in $\theta$ determined by the expected signature of the driving signal $\p{X} = (t, \p{W})$.\medskip
 
The aim of this section is to generalize the methodology from~\cite{papavasiliou_parameter_2011} to estimate the parameters $\theta \in \Theta$ of a signature SDE~\eqref{eq: parameterized SDE}, where the vector field is affine in signatures of the solution. 
We begin by showing that in this setting the theoretical expectation of the $r$th Picard iteration can also be expressed as a polynomial $P_r$ in $\theta$ determined by the expected signature $\E[\p{X}]$ of the driving signal; see \Cref{remark: comparison of polynomials}. Then, we show that our estimator is consistent in~\Cref{thm:consistency}. \medskip

We note that the expected signature of the solution of the SDE is uniquely determined by its distribution. Therefore, estimating the unknown parameters of the SDE by matching the theoretical and empirical expected signatures of the solution, identifies parameters up to their equivalence class under the distribution they induce on the solution. We show in \Cref{sec:non-identifiability} that distinct parameter sets may yield the same solution, and hence the same expected signature of the solution, for sufficiently bounded trajectories of the driving signal.

\begin{remark} \label{remark: on the other paper}
    While our vector fields are \emph{affine} with respect to the signature $\p{Y}$ of the solution, the setting in this paper is a generalization of the path-independent polynomial vector fields of~\cite{papavasiliou_parameter_2011} by using the shuffle product. For example, we have $\p{Y}_{0,t}^{(1,1)} = \frac{1}{2}(Y^{(1)}_t)^2$ for all $t\in T$; see~\Cref{experiment: 1 dimensional}.
    In addition, the polynomial vector fields considered in~\cite{papavasiliou_parameter_2011} are not  $\Lip(\gamma)$ a priori, and therefore, the authors believe a modification of these vector fields, similar to what is suggested in this article, is necessary to guarantee the existence and uniqueness of the solutions.\medskip

    Furthermore, the authors believe there is an error in~\cite[Equation 3.19]{papavasiliou_parameter_2011} in the proof of consistency of the estimator in the path-independent setting~\cite[Theorem 3.6]{papavasiliou_parameter_2011}. Their proof may be rectified by replacing their use of the Jacobian of $P_r$ with a matrix of derivatives of $P_r$, where each row is evaluated at a different point, assuming invertibility everywhere of such matrices, and an even stronger bound on the norm of their inverses. 
    However, we believe our proof of consistency in~\Cref{thm:consistency} can be adapted to their setting, which requires significantly fewer assumptions, as discussed in the introduction. 
\end{remark}

\subsection{Picard Iterations as Polynomials of Parameters} 
In order to approximate the expected signature of the solution to~\eqref{eq: parameterized SDE} with a polynomial expression in the parameters, we begin by studying the Picard iterations for a given $\p{X}\in \GRP(\drive)$. 
For $r \in \mathbb{Z}_{\geq 0}$, our aim is to express the path $\p{Y}_\theta(r)$ as a system of polynomials in $\theta$, where the coefficients are given by the signature of~$\p{X}$. 

\begin{notation}
    Let $\cA\subset \mathbb{Z}_{\geq 0}$, and $I$ be a \textit{multi-index} or \textit{word} in $\cA$, i.e. $I = (i_1,\ldots,i_\ell)\in \cA^\ell$ for some $\ell \in \mathbb{Z}_{\geq 0}$. Then we let $|I| \coloneqq \ell$ denote the length of the word. If $|I| \geq 1$, we write
    \begin{align}
        I = (I^-, I^f),
    \end{align}
    where $I^- = (i_1, \ldots, i_{\ell-1})$ is the word consisting of the first $\ell-1$ elements in $I$, and $I^f =(i_\ell)$ is a word of length $1$ consisting of the final element. For a word $I=(i_1,\ldots,i_\ell)$ in $\cA$ and a word $J = (j_1,\ldots,j_{\ell'})$ in $\{1,\ldots,\ell\}$, we define $I_J \coloneqq (i_{j_1}, \ldots, i_{j_{\ell'}})$. Note that the word in $\cA$ with length 0 is denoted by $\emptyset$. For a set $\cA$, we denote the set of words of length at most $q$ in $\cA \label{pg:cW_def}$ by
    \begin{align} \label{eq:words}
        \cW(\cA, q) \coloneqq \{ (i_1, \ldots, i_\ell) \in \cA^\ell \, : \, \ell \leq q\}.
    \end{align}
\end{notation}

\begin{example}
   $I = (4,0,1)$ is a word of length 3 in $\{0,1,2,3,4\}$. In this case $I^- = (4,0)$ and $I^f = (1)$. Moreover, if $J=(1,3)$, then $I_J = (4,1)$.
\end{example}
\begin{theorem} \label{thm: polynomial equations}
 For all $r,\ell \in \mathbb{Z}_{\geq 0}$, define 
    \begin{align} \label{eq: Q(r,l)}
        Q(r,\ell) \coloneqq\begin{cases}
            0 & \text{if } r = 0 \text{ or } \ell = 0\\
            2^{r-1} & \text{if } r\geq 1 \text{ and } \ell = 1 \\
            2^{r} - 1 & \text{if } r\geq 1 \text{ and } \ell > 1
        \end{cases}. 
    \end{align}
    If $\theta\in \Theta$ and $\p{X} \in \GRP_p(\drive)$ such that $d_p(\p{X}, \mathbf{0}) < \cN(\theta)$, 
    then for any $r\in \mathbb{Z}_{\geq 0}$, any word $I \in \cW([m]_0,q)$, and any $t\in [0,T]$,
    \begin{align} \label{eq: polynomial relation}
    \p{Y}_\theta(r)_{0,t}^{I} =  \sum_{J \in \cW([n]_0, Q(r,|I|))} \alpha_{r,J}^I(\theta) \p{X}_{0,t}^J,
    \end{align} 
    where 
    \begin{itemize}
        \item for all $r\geq 0$, $\alpha^{\emptyset}_{r,\emptyset} (\theta) = 1$;
        \item for all $I\in ([m]_0)^\ell$ with $1\leq \ell \leq q$, we have $\alpha^I_{0,\emptyset} (\theta) = 0$;
        \item for all $r\geq 1$, all $i\in [m]_0$, and all words $J \in \cW([n]_0,Q(r,1))$, 
        \begin{align} \label{eq: polynomials: level 1}
           \alpha^{(i)}_{r,J} (\theta)= \begin{cases}
            \sum_{\substack{K: K \in \{0,\ldots,m\}^\ell, \\ |J^-| \leq Q(r-1,\ell),\\ \ell \leq q'}} \theta_{i,J^f}^K \ \alpha^K_{r-1,J^-}(\theta) & \text{if $1 \leq |J| \leq Q(r-1,q)+1$}\\
            0 & \text{otherwise}
        \end{cases};
        \end{align}
        \item and for all $r\geq 1$, all words $I\in ([m]_0)^\ell$ with $2\leq \ell \leq q$, and all words $J \in \cW([n]_0, Q(r, |I|))$,
        \begin{align} \label{eq: polynomials: level 2 and higher}
        \alpha^I_{r,J}(\theta) = \begin{cases} 
        \sum_{\substack{(L,K): \\ (1,\ldots,|J|-1) \in L \shuffle K, \\ |L| \leq Q(r-1,|I|-1), \\ |K| \leq Q(r,1)-1}} \alpha^{I^-}_{r-1,J_L}(\theta)\ \alpha^{I^f} _{r,(J_K,J^f)} (\theta) & \text{if } 1\leq |J|\leq Q(r-1,|I|-1)+Q(r,1) \\
        0 & \text{otherwise} \end{cases}
        \end{align}
    \end{itemize}
    with
    \begin{align}
    \theta_{0,j}^I \coloneqq \begin{cases}
            1 &\text{if $j=0$ and $I=\emptyset$}\\
            0 & \text{otherwise}
        \end{cases}
    \end{align}
    for all $j\in [n]_0$ and all $I\in\cW([m]_0,q)$.
    Moreover, for all $r\in \mathbb{Z}_{\geq 0}$ and all words $I,J$, the function $\theta \mapsto \alpha^I_{r,J}(\theta)$ is a polynomial function of degree at most $Q(r,|I|)$.
    \end{theorem}
    \begin{proof}
    We first prove \eqref{eq: polynomial relation} for bounded $1$-variation paths, where $p=1$. Note that in this case, the first and second bullet points are immediate. These bullet points prove \eqref{eq: polynomial relation} for the cases of $r=0$ as well as $r\geq 1$  and $I=\emptyset$. We will prove the remaining cases of~\eqref{eq: polynomial relation} by induction. We assume that for some $r\in \mathbb{N}$, \eqref{eq: polynomial relation} holds for all $(r',I')$ with $r' < r$ and  $I' \in \cW([m]_0,q)$. We prove the relations in~\eqref{eq: polynomials: level 1} for Picard iteration $r$ and a length 1 word $(i)$. 
    Note that since $d_p(\p{X},\mathbf{0}) \leq \cN(\theta)$, by \Cref{prop: N(theta) and M(theta)}, we have $\left\| \p{Y}_\theta(r)_{s,t}\right\| < \cM(\|\theta\|)$ for all $r\geq 0$ and all $(s,t)\in \Delta_T$. 
    Thus, without loss of generality, we will omit $\cM(\|\theta\|)$ from the notation for the vector fields, since $F_\theta(\bs) = F_{\theta, \cM(\|\theta\|)}(\bs)$ when $\|\bs\|< \cM(\|\theta\|)$. 
    So,
    \begin{align}
         \p{Y}_\theta(r)_{0,t}^{(i)} =& \int_0^t \sum_{j=0}^n \left[F_{\theta}(\p{Y}_\theta(r-1)_{0,u} )\right]_{i,j}dX_{0,u}^{(j)} \\
        =& \int_{0}^t \sum_{j=0}^n \langle \theta_{i,j}, \p{Y}_\theta(r-1)_{0,u}\rangle dX_{0,u}^{(j)} \\ 
        =& \int_{0}^t \sum_{j=0}^n \sum_{K \in \cW([m]_0, q)} \theta_{i,j}^K \ \p{Y}_\theta(r-1)_{0,u}^{K} \ dX_{0,u}^{(j)}  \\ 
        =& \sum_{j=0}^n \sum_{K \in \cW([m]_0, q)} \theta_{i,j}^K  \sum_{\tilde{J} \in \cW([n]_0, Q(r-1, |K|))} \alpha_{r-1,\tilde{J}}^K(\theta) \int_{0}^{t}  \p{X}_{0,u}^{\tilde{J}} \ dX_{0,u}^{(j)}   \\ 
        =& \sum_{j=0}^n \sum_{K \in \cW([m]_0,q)} \sum_{\tilde{J} \in \cW([n]_0, Q(r-1, |K|))}  \theta_{i,j}^K 
        \ \alpha_{r-1,\tilde{J}}^K(\theta) \  \p{X}_{0,t}^{(\tilde{J},j)}.
    \end{align}
    This proves \eqref{eq: polynomials: level 1}, and hence, proves \eqref{eq: polynomial relation} for $(r,(i))$. 
    Now, we consider the relations in~\eqref{eq: polynomials: level 2 and higher}, which will be proved by induction on the length of the word $I$. Let $|I|=\ell$ with $2\leq \ell \leq q$. We assume that \eqref{eq: polynomial relation} additionally holds at Picard iteration $r$ and for words $I'$ with $|I'|<\ell$. Then, 
    \begin{align}
        \p{Y}_\theta(r)_{0,t}^{I}  =& \int_0^t \p{Y}_\theta(r-1)^{I^-}_{0,u} \cdot\left( F_{\theta}(\p{Y}_\theta(r-1)_{0,u} )dX_{0,u}\right)_{I^f} \\
        =&\int_{0}^t \p{Y}_\theta(r-1)_{0,u}^{I^-} \ d\p{Y}_{\theta}(r)_{0,u}^{I^f}  \\ 
        =& \int_0^t \left(\sum_{L \in \cW([n]_0, Q(r-1, |I|-1))} \alpha_{r-1, L}^{I^-}(\theta) \p{X}_{0,u}^L \right)  \left(\sum_{\substack{K \in \cW([n]_0, Q(r,1)), \\ |K|\geq 1}}\alpha_{r, K}^{I^f}(\theta) d \p{X}_{0,u}^K \right)  \\ 
        =& \int_0^t \sum_{L \in \cW([n]_0, Q(r-1, |I|-1))} \sum_{\substack{K \in \cW([n]_0, Q(r,1)), \\ |K|\geq 1}} \alpha_{r-1,L}^{I^-}(\theta) \ \alpha_{r,K}^{I^f}(\theta) \ \p{X}_{0,u}^L \ \p{X}_{0,u}^{K^-} \ d\p{X}_{0,u}^{K^f} \\ 
        =& \sum_{L \in \cW([n]_0, Q(r-1, |I|-1))} \sum_{\substack{K \in \cW([n]_0, Q(r,1)), \\ |K|\geq 1}}  \sum_{\tilde{J} \in L \shuffle K^-} \alpha_{r-1,L}^{I^-}(\theta) \ \alpha_{r,K}^{I^f}(\theta) \int_{0}^t \p{X}_{0,u}^{\tilde{J}} \ d\p{X}_{0,u}^{K^f}  \\ 
        =& \sum_{L \in \cW([n]_0, Q(r-1, |I|-1))} \sum_{\substack{K \in \cW([n]_0, Q(r,1)), \\ |K|\geq 1}}  \sum_{\tilde{J} \in L \shuffle K^-} \alpha_{r-1,L}^{I^-}(\theta) \ \alpha_{r,K}^{I^f}(\theta) \ \p{X}_{0,t}^{(\tilde{J},K^f)}.  
    \end{align}
    This proves \eqref{eq: polynomials: level 2 and higher}, and hence, proves \eqref{eq: polynomial relation} for $(r,I)$. This concludes the proof of \eqref{eq: polynomial relation} in the case of $p=1$. \medskip 
    
    Now suppose that $p>1$. Because $\p{X}\in \GRP_p(\drive)$, there exists a sequence $\{\p{X}(k)\}_{k\in \mathbb{N}}$ of extensions of bounded $1$-variation paths such that $d_p(\p{X}(k),\p{X}) \to 0$ as $k\to \infty$. For all $k\in \mathbb{N}$, let $\{\pp{Z}_{\theta}(k,r)\}_{r\in \mathbb{Z}_{\geq 0}}$ be the sequence of the $\cM(\|\theta\|)$-Picard iterations of the path-dependent differential equation 
    \begin{align}
    dY_t = F_{\theta,\cM(\|\theta\|)}(\p{Y}_t) \ d\p{X}(k)_t,
    \end{align}
    and define $\pp{Y}_\theta(k,r) \coloneqq \pi_{\ssol} (\pp{Z}_\theta(k,r))$ and $\p{Y}_{\theta} (k,r) \coloneqq \pp{Y}_\theta(k,r)^1 + \bone$ for all $r \in \mathbb{Z}_{\geq 0}$. Because $d_p(\p{X},\mathbf{0}) < \cN(\theta)$, there exists $N_0\in \mathbb{N}$ such that for all $k\geq N_0$, we have $d_p(\p{X}(k), \mathbf{0})< \cN(\theta)$. Therefore, since \eqref{eq: polynomial relation} has already been proved for the case of $p=1$, we have 
    \begin{align} \label{eq: polynomial relations for the sequence of bounded variation paths}
    \p{Y}_{\theta}(k,r)_{0,t}^{I} =  \sum_{J \in \cW([n]_0, Q(r,|I|))} \alpha_{r,J}^I(\theta) \p{X}(k)_{0,t}^J
    \end{align}
    for all $k\geq N_0$, all $r\in \mathbb{Z}_{\geq 0}$, all words $I \in \cW([m]_0,q)$, and all $t\in [0,T]$.  Because the functions $ \pp{Z} \mapsto \int h_{\cM(\|\theta\|)}(\pp{Z}) d\pp{Z}$ and $\pp{Z} \mapsto \pi_{\ssol}(\pp{Z})$ are both continuous on $\GRP_p(\drive\oplus \ssol)$, we have that $\lim_{k\to \infty} d_p\left(\pp{Y}_{\theta}(k,r), \pp{Y}_\theta(r) \right)=0$, and so, $\lim_{k\to \infty} \p{Y}_\theta(k,r)_{0,t}^I = \p{Y}_\theta(r)_{0,t}^I$ for all $r\in \mathbb{Z}_{\geq 0}$, all $I\in \cW([m]_0,q)$, and all $t\in [0,T]$. Hence, by taking the limits of the both sides of \eqref{eq: polynomial relations for the sequence of bounded variation paths} as $k\to \infty$, the proof of \eqref{eq: polynomial relation} is concluded for the case of $p>1$.\medskip

    Note that the four bullet points recursively define $\alpha_{r,J}^I$ for all $r\geq 0$, all words $I\in \cW([m]_0,q)$, and all words $J \in \cW([n]_0, Q(r, |I|))$. This recursive definition confirms that $\theta \mapsto \alpha_{r,J}^I(\theta)$ is a polynomial of degree at most $Q(r,|I|)$.
    \end{proof}

\begin{remark} \label{remark: comparison of polynomials}
Due to the fact that our paths $\p{Y}_\theta(r)$ are the level 1 components of the Picard iterations $\pp{Y}_\theta(r)$ of the lifted equation (see~\Cref{thm: universal limit theorem 2}), they differ from the Picard iterations used in~\cite{papavasiliou_parameter_2011}. As a result, for any $I\in \cW([m]_0,q)$ and any $r\in \mathbb{Z}_{\geq 0}$ in~\eqref{eq: polynomial relation}, we are able to obtain smaller polynomials both in terms of the maximum degrees of the polynomials $\alpha_{r,J}^I$ and the maximum length of words $J$ appearing on the right-hand side of~\eqref{eq: polynomial relation}. In~\cite{papavasiliou_parameter_2011},
these quantities are $|I|q^r$ and $|I| \sum_{i=0}^{r-1}q^i$  respectively, where $q$ is the maximum degree of $a_\theta$ and $b_\theta$ in \eqref{eq:path_ind_sde}. In our case, they are both equal to $Q(r,|I|) = O(2^r)$, as defined in~\eqref{eq: Q(r,l)}. Note that smaller polynomials improve the overall efficiency of the estimation method, and while $Q(r,|I|)$ grows exponentially with $r$, we prove in \Cref{thm:consistency} that the rate of convergence of our estimator is also exponential in $r$.
\end{remark}
    
\subsection{Empirical Estimator}
In the previous section, we considered the path-dependent differential equation~\eqref{eq: parameterized SDE} with a given $\theta \in \Theta$ for a fixed sample of the driving rough path $\p{X} \in \GRP_p(\drive)$ such that $d_p(\p{X}, \bzero) < \cN(\|\theta\|)$. In particular, we obtained polynomial expressions in $\theta$ for the Picard iterations, whenever the driving signal is sufficiently bounded. In this section, we consider the case where some components of the true parameter $\p{\theta_0} \in \Theta$, which determines the vector field, are unknown, and the differential equation is driven by stochastic rough paths $\p{X}$. We introduce the \textit{Expected Signature Matching Method} for estimating the unknown components from the observed trajectories of the solution. 
\medskip

First, we will fix a decomposition of the parameter space into known and unknown components, $\Theta = \known{\Theta} \oplus \unknown{\Theta}$,
where $\known{\Theta} \label{pg:known_unknown}$ is the subspace of \emph{known} parameters, while $\unknown{\Theta}$ is the subspace of \emph{unknown} parameters. We note that in general, $\dim \Theta = m(n+1)\dim \ssol > \dim \ssol$, so the system of polynomials in $\theta \in \Theta$ defined by the collection of all words $I \in \cW([m]_0, q)$ would yield an underdetermined system. Thus, we will make the assumption that for the unknown subspace\label{pg:d},
\begin{align}
    d \coloneqq \dim \unknown{\Theta} \leq \dim\ssol.
\end{align}
We will consider a vector field parametrized by the \emph{true parameter}\label{pg:decomposition of the true parameter} 
\begin{align} 
    \p{\theta_0} = (\trueknown, \theta_0) \in \known{\Theta} \oplus \unknown{\Theta}.
\end{align}
Thus, the known true parameter $\trueknown \in \known{\Theta}$ will always be fixed. We will use the notation
\begin{align}
    \p{\theta} = (\trueknown, \theta) \in \Theta
\end{align}
to denote the full parameter set for some $\theta \in \unknown{\Theta}$. 
Now pick $\xi\in \mathbb{R}_+ \label{pg:xi}$,  
and define
\begin{align}
    \Theta_\xi \coloneqq \{\theta \in \unknown{\Theta} \, : \,  \cN(\|\p{\theta}\|) \geq \xi\} \andd 
    E_\xi \coloneqq \{\sample \in \samplespace \, : \,  d_p(\p{X}(\sample) , \mathbf{0}) < \xi \}. 
\end{align}
Note that for every $\theta \in \Theta_\xi \label{pg:allowable}$ and $\sample \in E_\xi$, we have $d_p(\p{X}(\sample), \bzero) < \xi \leq \cN(\|\p\theta\|)$,
and thus, by \Cref{thm: polynomial equations}, for all $r\in \mathbb{Z}_{\geq 0}$ and all words $I\in \cW([m]_0, q)$, there exists a polynomial function $\theta \mapsto P^I_r(\theta)$ such that for all $\theta \in \Theta_\xi$, we have 
\begin{align} \label{eq:picard_expsig_polys}
    P^I_{r}(\theta) = \mathbb{E} \left[ \p{Y}_{\p\theta}(r)_{0,T}^{I} \cdot \chi_{E_\xi} \right].
\end{align}
Note that the coefficients of $P^I_{r}$ are determined by the expected signature $\mathbb{E}\left[\p{X}^J \cdot \chi_{E_\xi}\right]$ for words $J$ in a (possibly  proper) subset of $\cW([n]_0, Q(r,|I|))$. \medskip

Suppose we observe $N$ trajectories $\{\p{Y}_{\p{\theta_0}}(\sample_i)\}_{i=1}^N$, $\sample_i \in \samplespace$, of the $\cM(\|\p{\theta_0}\|)$-solution to the path-dependent SDE
\begin{align}
    dY_t = F_{\p{\theta_0}, \cM(\|\p{\theta_0}\|)}(\p{Y}_t) d\p{X}_t.
\end{align}
Since the dimension of the space of unknown parameters is $\dim\Theta_2 = d$, we choose a set of words $\{I_1, \ldots, I_d\} \subset \cW([m]_0, q)$. 
The \textit{Expected Signature Matching Method (ESMM)} finds an estimate for $\theta_0$ by solving the polynomial system of equations 
\begin{align}\label{eq: expected signature matching method}
    \left\{P^{I_k}_{r}(\theta) = \frac{1}{N}\sum_{i=1}^N \p{Y}_{\p{\theta_0}}(\sample_i)_{0,T}^{I_k} \cdot \chi_{E_\xi}(\sample_i), \ \  k = 1, \ldots, d\right\}
\end{align}
for some $r \in \N$.

\begin{remark}
   To ensure the consistency of the ESMM, we need to pick $\xi$ so that $\theta_0 \in \Theta_\xi^\circ$. To do so, we fix $\mu > 0$ such that $\|\p{\theta_0}\| < \mu$, and set $\xi \coloneqq \cN(\mu)$. Note that larger values of $\|\p{\theta_0}\|$ require larger values of $\mu$, which in turn, lead to smaller values of $\xi$, thereby shrinking the set $E_\xi$.
\end{remark}

\subsection{Consistency}
This section justifies the use of ESMM by proving its consistency under certain constraints on the function 
\begin{align} \label{eq:P_def}
    P: \Theta_\xi \to \R^d \quad \text{defined by} \quad \theta \mapsto (P^{I_1}(\theta), \ldots, P^{I_d}(\theta))
\end{align}
where for any word $I \in \cW([m]_0,q)$, the function $P^{I} : \Theta_\xi \to \mathbb{R}$ is defined with
\begin{align}
    P^{I}(\theta) \coloneqq \mathbb{E} \left[\left(\p{Y}_{\p\theta}\right)_{0,T}^I \cdot \chi_{E_\xi} \right].
\end{align}
In particular, we will prove in~\Cref{thm:consistency} that when $P$ is differentiable at $\theta_0$ with an invertible Jacobian, then ESMM is consistent: for sufficiently high Picard iterations $r \in \N$ and with suffiently many samples $N$, the system \eqref{eq: expected signature matching method} admits a solution arbitrarily close to $\theta_0$ almost surely.
We begin by showing that the expected signature of the solution is continuous in $\theta$.
\begin{lemma}\label{cor: uniform convergence}
For any $I \in \cW([m],q)$, the sequence of functions $\{P^I_r\}_{r\in \mathbb{Z}_{\geq 0}}$, defined in~\eqref{eq:picard_expsig_polys}, converges uniformly  on $\Theta_\xi$ to the function $P^I$, defined in \eqref{eq:P_def}, as $r \to \infty$. Hence, $P^I$ is continuous on $\Theta_\xi$.    
\end{lemma}
\begin{proof}
The proof is given in \Cref{appendix: proofs}.
\end{proof}

We will now move on to the main results on consistency. The primary tool we will use in our proof is a variant~\cite{miranda2002} of Miranda's theorem, which is a generalization of the classical intermediate value theorem. We begin with several required definitions. 
\begin{definition}\cite[Definition 2.1]{miranda2002}
        For all $\delta>0$,
        define 
        \begin{align}
            \cube(\delta) \coloneqq [-\delta, \delta]^d \subset \R^d \andd \cube_k(\delta)^{\pm} \coloneqq \{ (x_1, \ldots, x_d) \in \cube(\delta) \, : \, x_k = \pm \delta\}
        \end{align}
        for all $k \in [d]$.
    \begin{itemize}
        \item A set $D\subset \mathbb{R}^d$ is called a \textit{Miranda domain} if there exists a surjective continuous map $g:\cube(1) \to D$ such that  $g(\partial \cube(1) ) = \partial D$, where $\partial$ denotes the boundary of a set. In this case, $g$ is called a \textit{Miranda mapping} for $D$. 
        \item For a Miranda domain $D$, let $\mathcal{D} = \{D_1^+, D_1^-, \cdots, D_d^+, D_d^-\}$ be a set of subsets of $\partial D$. The set $\mathcal{D}$ is called a \textit{Miranda partition} of $\partial D$ if there exists a Miranda mapping $g: \cube(1) \to D$ such that for all $k\in [d]$,
        \begin{align}
            g(\cube_k(1)^+) = D_k^+ \text{ and } g(\cube_k(1)^-) = D_k^-.
        \end{align}
        To simplify terminology, we will often call the pair $(D, \cD)$ a Miranda domain. 
        \item
        For a Miranda domain $(D,\cD)$,
        a continuous mapping $f = (f_1, \ldots,f_d): D \to \mathbb{R}^d$ is said to satisfy the \textit{Miranda conditions} on $(D, \mathcal{D})$ if 
        \begin{align}
            f_k(\theta_1) f_k (\theta_2) \leq 0\ \  \text{ for all } \theta_1 \in D_k^+, \text{ all } \theta_2\in D_k^-, \text{ and all }k\in [d].
        \end{align}
    \end{itemize}
\end{definition}

We will begin by showing consistency when we assume the existence of Miranda domains and partitions such that the corresponding Miranda conditions hold for $P- P(\theta_0)$ with strict inequalities. In particular, we show that for sufficiently large $r$ and $N$, the Miranda conditions also hold for the error of the polynomial system in~\eqref{eq: expected signature matching method} on the same Miranda domains.

\begin{lemma}\label{lem:solution_given_miranda_domain}
Suppose $(D, \cD)$ is a Miranda domain such that $ D \subset \Theta_\xi$, and
\begin{align}\label{eq: Miranda condition}
    \left(P^{I_k}(\theta_1) - P^{I_k}(\theta_0)\right) \left(P^{I_k}(\theta_2) - P^{I_k}(\theta_0)\right) < 0 \ \text{ for all } \theta_1 \in D_k^+, \text{ all }\theta_2\in D_k^-, \text{ and all }k\in [d].
\end{align}
Then almost surely, there exist $N_0,r_0\in \mathbb{N}$ such that for all $N\geq N_0$ and all $r\geq r_0$, the polynomial system \eqref{eq: expected signature matching method} has a solution $\theta_{r,N}$ in $D$.
\end{lemma}
\begin{proof}
The proof is given in \Cref{appendix: proofs}.
\end{proof}

Next, we will show that if $P$ is differentiable at $\theta_0$ with an invertible Jacobian, then arbitrary small neighborhoods of $\theta_0$ contain Miranda domains with the properties described in \Cref{lem:solution_given_miranda_domain}.
\begin{lemma}\label{lem:existence_of_miranda}
Suppose $P$ is differentiable at $\theta_0$ with an invertible Jacobian. Then, there exists $\varepsilon_0> 0$ such that for all $\varepsilon \in (0, \varepsilon_0]$, there exists a Miranda domain $(D, \cD)$ where  
$D \subset B_\varepsilon(\theta_0)$ and~\eqref{eq: Miranda condition} holds. The constant $\varepsilon_0$ merely depends on the choice of words $\{I_1,\ldots,I_d\}$, the expected signature of the driving signal $\p{X}$, and the point $\theta_0$.
\end{lemma}
\begin{proof}
    The proof is given in \Cref{appendix: proofs}.
\end{proof}

Now, by putting together \Cref{lem:solution_given_miranda_domain} and \Cref{lem:existence_of_miranda}, in addition to the explicit rates from~\eqref{eq: rate of uniform convergence}, we obtain our main consistency result. 

\begin{theorem} \label{thm:consistency}
    Suppose $P$ is differentiable at $\theta_0 \in \Theta_\xi$ with an invertible Jacobian at this point. Then almost surely, for all $\varepsilon > 0$, there exist $N_0, r_0 \in \N$ such that for all $N \geq N_0$ and all $r \geq r_0$, the system \eqref{eq: expected signature matching method} has a solution $\theta_{r,N} \in B_{\varepsilon}(\theta_0)$. 
    Furthermore, there exists a constant $\varepsilon_0 > 0$ such that for all $\varepsilon \in (0, \varepsilon_0]$, we can explicitly express $r_0$ as the smallest positive integer satisfying 
    \begin{align} \label{eq:condition_on_r0}
        \rho^{-r_0} < \frac{(1-\rho^{-1}) \beta \left(\frac{1}{p}\right)!}{16\  C^{1/p} \ \xi \ \left\|(\J_{P}(\theta_0))^{-1}\right\|_\infty} \varepsilon,
    \end{align}
    where the constants $\rho$ and $C$ are those from~\eqref{eq: C and rho}. The constant $\varepsilon_0$ merely depends on the choice of words $\{I_1,\ldots,I_d\}$, the expected signature of the driving signal $\p{X}$, and the point $\theta_0$.
\end{theorem}
\begin{proof}
    By~\Cref{lem:existence_of_miranda}, for all $\varepsilon>0$, there exists a Miranda domain $(D_\varepsilon, \cD_\varepsilon)$ such that $D_\varepsilon \subset B_\varepsilon(\theta_0)$ and~\eqref{eq: Miranda condition} holds. Then by applying~\Cref{lem:solution_given_miranda_domain}, we can conclude the first part of the theorem. \medskip

    To obtain an explicit expression for $r_0$, let $\varepsilon_0$ be defined as in~\Cref{lem:existence_of_miranda} and consider $\varepsilon \in (0, \varepsilon_0]$.
    Moreover, let $\eta_k$ be defined as in \eqref{eq: def of epsilon_k} for all $k\in [d]$. By the discussion in the proof of \Cref{lem:solution_given_miranda_domain}, we require $r_0\in \mathbb{N}$ such that for all $r\geq r_0$,
    \begin{align} \label{eq: how big r,N should be}
        &\sup_{\theta\in \Theta_\xi} \left|P_{r}^{I_k}(\theta) - P^{I_k}(\theta) \right| < \eta_k \text{ for all } k\in [d].
    \end{align}
    However, by~\eqref{eq: rate of uniform convergence}, this can occur when
    \begin{align}
        \sup_{\theta \in \Theta_\xi} \left| P_{r}^{I_k}(\theta) - P^{I_k}(\theta) \right|  \leq  \frac{2 \rho^{-r}}{(1-\rho^{-1})\beta\left(\frac{1}{p}\right)!} C^{\frac{1}{p}} \xi < \eta_k \text{ for all } k\in [d].
    \end{align}
    Using the definition of $\eta_k$ in~\eqref{eq: def of epsilon_k} together with \eqref{eq: refined miranda 1} and \eqref{eq: refined miranda 2}, we obtain $\eta_k > \frac14 \delta$ for all $k\in [d]$, where $\delta$ is defined in~\eqref{eq:delta}. Therefore, a sufficient condition for $r_0$ is
    \begin{align}
        \frac{2 \rho^{-r_0}}{(1-\rho^{-1})\beta\left(\frac{1}{p}\right)!} C^{\frac{1}{p}} \xi < \frac14 \delta =  \frac{\varepsilon}{8\left\| \left(\J_{P}(\theta_0)\right)^{-1} \right\|_\infty},
    \end{align}
    and thus we obtain the desired expression in~\eqref{eq:condition_on_r0}.
\end{proof}

We now conclude this section by proving that the function $P$ is in fact, differentiable almost everywhere.

\begin{proposition} \label{prop:P_diff_ae}
    The function $P: \Theta_\xi \to \R^d$ as defined in~\eqref{eq:P_def} is locally Lipschitz. Therefore $P$ is differentiable almost everywhere on $\Theta_\xi^\circ$. 
\end{proposition}
\begin{proof}
The proof is given in \Cref{appendix: proofs}.
\end{proof}

\section{Experiments} \label{sec:experiments}
In this section, we evaluate the performance of the Expected Signature Matching Method in estimating the unknown parameters of an underlying path-dependent SDE using a number of its observed trajectories. The code, generated data, and experimental results associated to this section are available at our GitHub repository\footnote{\url{https://github.com/pardis-semnani/signature-SDE-parameter-estimation}}. 
\medskip

In all the experiments in this section, we set $N\coloneqq 2000$, $r \coloneqq 3$, and $q \coloneqq 3$. To perform each experiment, we first select the parameters $m,n,d,\delta t,T$, several $d$-subsets of the words $\cW([m]_0,q)$, as well as a linear path-dependent SDE 
\begin{align} \label{eq: simulations path dependent}
    dY_t = F_{\p{\theta_0}}(\p{Y}_{t})  d\p{X}_t,
\end{align} 
where $\p{\theta_0} = (\trueknown,\theta_0) \in \known{\Theta} \oplus \unknown {\Theta} = \Theta = \Mat_{m, n+1} (\ssol)$  with $\dim \unknown{\Theta}=d$ denotes the parameters, $\p{X}=(t,\p{W})$ is the driving noise, and $\p{W}$ is an $n$-dimensional Brownian motion.  We then repeat the following procedure 100 times:
\begin{enumerate}
    \item Generate $N$ trajectories of the solution to the SDE in~\eqref{eq: simulations path dependent}.
    \item Apply the ESMM corresponding to each selected set of words to the generated trajectories in order to estimate $\theta_0$.
\end{enumerate}
To simulate trajectories of \eqref{eq: simulations path dependent}, we consider the lifted version $\p{F}_{\p{\theta_0}}$ of the vector field $F_{\p{\theta_0}}$, and then use the \texttt{Diffrax} package~\cite{kidger2021on} in \texttt{Python} with step size $\delta t$ and the \texttt{Heun Stratonovich} solver to simulate solution trajectories of the  path-independent SDE equivalent to \eqref{eq: simulations path dependent}, i.e.
\begin{align} \label{eq: simulations path independent}
    d\p{Y}_t = \p{F}_{\p{\theta_0}}(\p{Y}_{t})  d\p{X}_t, \ \text{ with } \p{Y}_0 = \bone \text{ and } \p{F}_{\p{\theta_0}}(\p{Y}_t) = \tens(\p{Y}_t) \circ F_{\p{\theta_0}}(\p{Y}_t),
\end{align}
 over the time interval $[0,T]$. The underlying paths in the solution trajectories of the path-independent SDE \eqref{eq: simulations path independent} are $\ssol$-valued. We take the projection of each of these paths onto $\sol$ to extract the underlying paths in the solution trajectories of the path-dependent SDE \eqref{eq: simulations path dependent}, denoted by $\{Y_{\p{\theta_0}}(\sample_i)\}_{i\in [N]}$. Then, we use the \texttt{iisignature} package~\cite{iisignature} in \texttt{Python} to obtain the signatures up to level $q$ of these paths, which we denote by $\{\p{Y}_{\p{\theta_0}}(\sample_i)\}_{i\in [N]}$. Note that although the solutions to~\eqref{eq: simulations path independent} involve these signatures, we choose to approximate them from their underlying paths $\{Y_{\p{\theta_0}}(\sample_i)\}_{i\in [N]}$ as these paths are typically what a user of the ESMM would observe in practice.
\medskip
 
Now to estimate $\theta_0$, we use the \texttt{Macaulay2}~\cite{M2} package \texttt{NumericalAlgebraicGeometry}~\cite{leykin2011numerical,NumericalAlgebraicGeometrySource} to solve the polynomial system
\begin{align} \label{eq: expected signature mathcing method - modified}
    \left\{P^{I_k}_{r}(\theta) = \frac{1}{N}\sum_{i=1}^N \p{Y}_\p{\theta_0}(\sample_i)_{0,T}^{I_k} , \ \  k = 1, \ldots, d\right\}
\end{align}
for each selected set of words $\{I_1,\ldots,I_d\}\subset \cW([m]_0,q)$.
We report all the real solutions of this polynomial system. Furthermore, we let
\begin{align}\label{eq: estimate in each trial}
    \hat\theta\coloneqq \arg\min\{\|\theta - \theta_0\|_1: \theta \text{ is a real solution to the polynomial system~\eqref{eq: expected signature mathcing method - modified}.}\},
\end{align}
and report the mean and standard deviation of the obtained estimates $\hat\theta$ across all the 100 trials.

\begin{remark}
    As we note that $\lim_{T \to 0} \cP(d_p(\p{X}|_{\Delta_{T}}, \bzero) < \xi) = 1$, for small values of $T$, we consider the simplifying assumption in the polynomial system \eqref{eq: expected signature mathcing method - modified}, where we omit the term $\chi_{E_\xi}(\sample_i)$ (when compared with~\eqref{eq: expected signature matching method}). This also allows us to use $\mathbb{E}\left[\p{X}^J\right]$ in the coefficents of $P_r^{I_k}$, which is computed using the explicit formula in~\cite[Theorem 1]{ladroue_expectation_2010}. Our experiments show that even with this assumption, our method is able to effectively estimate parameters. 
\end{remark}

\begin{experiment}
\label{experiment: 1 dimensional}
In this experiment, we set $m=n=1$, and use step size $\delta t = 0.001$ over the interval $[0,T]$ with $T=0.2$. The SDE model in this experiment contains $d=3$ unknown parameters $(\param^1, \param^2, \param^3)$, and is given by
\begin{align}
    \begin{dcases}
        d {Y}_t^{(0)} = dt,\\
        d {Y}_t^{(1)} = -\param^1 \left(\p{Y}_t^{\emptyset} - \p{Y}_t^{(1)}\right) dt + \left(\param^2 \p{Y}_t^{\emptyset} + \param^3 \p{Y}_t^{(1,1)} \right) \circ d\p{W}_t^{(1)},
    \end{dcases}
\end{align}
where we recall that $\p{Y}_t^{\emptyset}=1$.  The true parameter is $\theta_0 = (-1, 0, 4)$, and \Cref{tab: exp 1} shows the component-wise means and standard deviations of the estimates $\hat\theta$, defined in~\eqref{eq: estimate in each trial}, which are obtained over the 100 trials via the ESMM with the sets of words
\begin{align}\label{eq: set of words 1}
    \cW_1 = \{ (0, 1, 0), (0, 1, 1), (1, 0, 1)\}, \quad
    \cW_2 = \{(1), (1, 1), (0, 1, 1)\}.
\end{align}
\Cref{fig: exp 1-all} in \Cref{appendix: figures for experiments} illustrates \emph{all} the real solutions to the polynomial system~\eqref{eq: expected signature mathcing method - modified} in each trial and for each set of words. Our results show that for both sets of words, the ESMM effectively estimates the correct parameters, with slightly more error for the second order signature term in the diffusion.\medskip

This experiment allows for a comparison between our version of the ESMM and the original version in~\cite{papavasiliou_parameter_2011}. In \cite[Example 5.1]{papavasiliou_parameter_2011}, parameter estimation is done for the same SDE model and the same true parameter values, where two parameters $(\param^1, \param^3)$ are considered unknown. 
Estimates from a single trial using $N=2000$, $r=3$, $T=0.25$ and the word set $\{(1),(1,1)\}$ are reported, and they are comparable to our results.
\begin{table}[h]
\centering
\renewcommand{\arraystretch}{1.3} 
\begin{tabular}{c !{\vrule width 1pt} c|c c c}
\toprule
 \multicolumn{2}{c|}{} & $\param^1$ & $\param^2$ & $\param^3$ \\
\hline
\multirow{2}{*}{$\cW_1$} & mean & -1.0135 &	-0.1117	& 4.2444 \\
\cline{2-5}
& std dev & 0.0017	& 0.0014	& 0.1820 \\
\specialrule{1.25pt}{0pt}{0pt}
\multirow{2}{*}{$\cW_2$} & mean & -0.9956& 0.0413 & 4.5703\\
\cline{2-5}
& std dev & 0.0005 & 	0.0005 &	0.6102 \\
\bottomrule
\end{tabular}
\caption{Mean and standard deviation of the estimated components of the unknown parameter $(\param^1, \param^2, \param^3)=(-1,0,4)$ in \Cref{experiment: 1 dimensional}, obtained using ESMM over 100 trials. The first two and last two rows correspond to the sets of words $\cW_1$ and $\cW_2$ in~\eqref{eq: set of words 1} respectively.}
\label{tab: exp 1}
\end{table}
\end{experiment}

\begin{experiment}
\label{experiment: 2 dimensional}
In this experiment, we set $m = 2$, $n = 1$, $\delta t=0.01$, and $T=0.2$. The SDE model under consideration involves $d=5$ unknown parameters $(\param^1, \param^2, \param^3, \param^4, \param^5)$, and is given by
\begin{align}
\begin{dcases}
dY_t^{(0)} = dt, \\
dY_t^{(1)} = \param^1 \left(\p{Y}_t^{(2,1)} - \p{Y}_t^{(1,2)}\right) dt + \left(\param^2\p{Y}_t^{\emptyset} + \param^3 \p{Y}_t^{(2)} \right) \circ d\p{W}_t^{(1)},\\
d Y_t^{(2)} = \param^4 \left(\p{Y}_t^{(2,1)} - \p{Y}_t^{(1,2)}\right) dt + \param^5\p{Y}_t^{\emptyset} \circ d\p{W}_t^{(1)}.
\end{dcases}
\end{align}
The true parameter is $\theta_0 = (-1,5,1,-2,3)$, and for each of the 100 trials, we apply the ESMM to the sets of words
\begin{align}\label{eq: set of words 2}
    \cW_3 \coloneqq \{ (1) , (1, 2) , (2, 1) , (2, 2) , (0, 1, 1)  \}, \quad \cW_4 \coloneqq \{(1, 2), (0, 1, 0), (0, 1, 1), (0, 2, 1), (1, 1, 0) \}.
\end{align}
\Cref{tab: exp 2}  presents the means and standard deviations of the components of the estimates $\hat\theta$ obtained across 100 trials, as specified in~\eqref{eq: estimate in each trial}. We observe that the ESMM accurately estimates components $\theta^2, \theta^3, \theta^5$, while the estimates for $\theta^1, \theta^4$ were better when the ESMM uses the word set $\cW_4$.
While selecting the optimal word set is beyond the scope of this paper, these empirical results suggest that the choice of word set influences the performance of the ESMM, and would be an interesting avenue for future work. \Cref{fig: exp 2-all} in \Cref{appendix: figures for experiments} shows the components of \emph{all} the real solutions to the polynomial system~\eqref{eq: expected signature mathcing method - modified} obtained in each trial and for each set of words.

\begin{table}[h]
\centering
\renewcommand{\arraystretch}{1.3} 
\begin{tabular}{c !{\vrule width 1pt} c|c c c c c}
\toprule
 \multicolumn{2}{c|}{} & $\param^1$ & $\param^2$ & $\param^3$ & $\param^4$ & $\param^5$ \\
\hline
\multirow{2}{*}{$\cW_3$} & mean &  0.3851 & 4.8549 & 1.0133 & -0.7635 & 2.9431 \\
\cline{2-7}
& std dev &  2.7267	& 0.2567 & 0.2167 & 1.5968 & 0.1668  \\
\specialrule{1.25pt}{0pt}{0pt}
\multirow{2}{*}{$\cW_4$} & mean &  -1.2528	 & 5.0165 & 0.9763 & -1.3885 & 3.0302  \\
\cline{2-7}
& std dev &  0.8674 & 0.0748 & 0.1642 & 0.2941 & 0.0795  \\
\bottomrule
\end{tabular}
\caption{Mean and standard deviation of the estimated components of the unknown parameter $(\theta^1, \theta^2, \theta^3, \theta^4, \theta^5)=(-1,5,1,-2,3)$ in \Cref{experiment: 2 dimensional}, obtained using ESMM over 100 trials. The first two and last two rows correspond to the sets of words $\cW_3$ and $\cW_4$ in~\eqref{eq: set of words 2} respectively.}
\label{tab: exp 2}
\end{table}
\end{experiment}

\begin{experiment}
\label{experiment: 3 dimensional}

In this experiment, we set $m=n=3$ and simulate solution trajectories with step size $\delta t = 0.001$ over the interval $[0,T]$ with $T=0.1$. The following system defines our SDE model, which involves
$d = 6$ unknown parameters $(\param^1, \param^2, \param^3, \param^4, \param^5, \param^6)$. 
\begin{align}
    \begin{dcases}
        d{Y}_t^{(0)} = dt,\\
        d{Y}_t^{(1)} = \param^1 \p{Y}_t^{(2)} dt + \param^2 \circ d\p{W}_t^{(1)}, \\
        d{Y}_t^{(2)} = \theta^3 \p{Y}_t^{(1)} dt + \theta^4 \circ d \p{W}_t^{(2)}, \\
        d{Y}_t^{(3)} = \param^5 \left(\p{Y}_t^{(2,1)} - \p{Y}_t^{(1,2)} \right) + \param^6 \circ d\p{W}_t^{(3)}.
    \end{dcases}
\end{align}
This model is an example of a path-dependent \emph{stochastic causal kinetic model}~\cite[Eq. (11)]{Peters2022}. The true parameter in this experiment is $\theta_0 = (1,1,-1,1,-5,1)$, and in each trial we obtain estimates of this parameter by performing the ESMM with respect to the sets of words
\begin{align}\label{eq: set of words 3}
    \cW_5 = \{(0, 3), (1, 1), (1, 2), (2, 1), (2, 2), (3, 3)\},\quad
    \cW_6 = \{(3), (1, 1), (1, 2), (2, 1), (2, 2), (3, 3)\}.
\end{align}
\Cref{tab: exp 3} indicates the component-wise means and standard deviations of the estimates $\hat\theta$ obtained over 100 trials, as specified in~\eqref{eq: estimate in each trial}. We observe that the ESMM accurately estimates the parameter set for both choices of word sets $\cW_5, \cW_6$. We note that $\theta^5$ was more difficult to estimate. \Cref{fig: exp 3-all} in \Cref{appendix: figures for experiments} shows \emph{all} the real solutions to the polynomial system~\eqref{eq: expected signature mathcing method - modified} across the 100 trials and corresponding to each set of words.
\begin{table}[h]
\centering
\renewcommand{\arraystretch}{1.3} 
\begin{tabular}{c !{\vrule width 1pt} c|c c c c c c}
\toprule
 \multicolumn{2}{c|}{} & $\param^1$ & $\param^2$ & $\param^3$ & $\param^4$ & $\param^5$ & $\param^6$\\
\hline
\multirow{2}{*}{$\cW_5$} & mean &  1.0264	& 1.0007 &	-1.0082 &	0.999 &	-6.415 & 1.0047  \\
\cline{2-8}
& std dev &  0.3177	& 0.0158	& 0.3109	& 0.0158	& 17.0444	& 0.0162 \\
\specialrule{1.25pt}{0pt}{0pt}
\multirow{2}{*}{$\cW_6$} & mean &  1.0243	& 1.001	& -1.0107	& 0.9989 &	-7.6045	& 1.0047 \\
\cline{2-8}
& std dev &  0.3186	& 0.0156	&0.3115	 & 0.0159	& 23.4885 &	0.0163 \\
\bottomrule
\end{tabular}
\caption{Mean and standard deviation of the estimated components of the unknown parameter $(\theta^1, \theta^2, \theta^3, \theta^4, \theta^5, \theta^6)=(1,1,-1,1,-5,1)$ in \Cref{experiment: 3 dimensional}, obtained using ESMM over 100 trials. The first two and last two rows correspond to the sets of words $\cW_5$ and $\cW_6$ in~\eqref{eq: set of words 3} respectively.}
\label{tab: exp 3}
\end{table}

Finally, we solve the polynomial system~\eqref{eq: expected signature mathcing method - modified} with $N = 200,000$ samples, which are obtained by aggregating the 2000 samples generated in each of the 100 trials. Among the real solutions to this polynomial system, the closest to $\theta_0$ in $L^1$ norm is $$(1.026, 1.001, -1.0079, 0.9993, -6.4232, 1.0049)$$ for the word set $\cW_5$, and $$(1.026, 1.001, -1.0079, 0.9993, -6.8913, 1.0049)$$ for the word set $\cW_6$.
\end{experiment}

\begin{remark}
For a Brownian motion trajectory $\p{W}(\sample)$, note that $-W^{(1)}(\sample)$ is also the underlying path of another Brownian motion trajectory. Therefore, the laws of the $\cM(\|\p{\theta_0}\|)$-solutions of the signatures SDEs
\begin{align}
    dY_t = A_{\p{\theta_0}}(\p{Y}_t) dt + B_{\p{\theta_0}}(\p{Y}_t) \circ d\p{W}_t 
    \ \text{ and } \ 
    dY_t = A_{\p{\theta_0}}(\p{Y}_t) dt  -B_{\p{\theta_0}}(\p{Y}_t) \circ d\p{W}_t 
\end{align}
are the same. This explains why, for each estimate $\hat\theta$, the polynomial system~\eqref{eq: expected signature mathcing method - modified} also admits a solution whose drift component is close to the drift component of $\hat\theta$, and whose diffusion component is close to the negative of the diffusion component of $\hat\theta$; see Figures~\ref{fig: exp 1-all} to~\ref{fig: exp 3-all} in \Cref{appendix: figures for experiments}.
\end{remark}

\section{Non-identifiability of Parameters from the Law} 
\label{sec:non-identifiability}
The characteristic property of the signatures in~\Cref{thm:univ_char} ensures that the distribution of the solution $\p{Y}_\theta$ of a linear signature SDE, restricted to noise terms in $E_\xi$, is uniquely characterized by the restricted expected signature of the solution $\E[\p{Y}_\theta \cdot \chi_{E_\xi}]$. Therefore, the Expected Signature Matching Method identifies the unknown parameters of the SDE to the furthest extent possible given the observed distribution of the solution. However, in this section, we show that signature RDEs with distinct parameters can admit the same solution for sufficiently bounded driving signals. Therefore, identifying a unique parameter set giving rise to the observed trajectories of the solution may not be possible.\medskip 

Here, as before, we assume that $\drive \cong \R^{n+1}$ and $\sol \cong \R^{m+1}$, where the extra coordinate for time is included in the zeroth coordinate. Recall that a signature RDE, where the vector field depends on $\p{Y}_t \in \ssol$ up to level $q$, can be written as 
\begin{align}
\label{eq: linear signature RDE}
    dY_t = F_\theta(\p{Y}_t) d\p{X}_t,
\end{align}
for some $\theta\in \Theta$, where $F_\theta : \ssol \to \LL(\drive, \sol)$ is defined in~\eqref{eq:F_theta_def}.  \medskip

Next, we begin to define two RDEs. 
Choose $q' \in \N$ such that $2q' \leq q$ and fix a linear functional $H \in \text{L}(\ssol, \mathbb{R})$ which does not depend on the final coordinate; in other words,
$
    H(\p{Y}_t) = \langle \lambda, \p{Y}_t\rangle,
$ 
where $\lambda\in \ssol$, and $\lambda^I \neq 0$ only if $I \in \cW([m-1]_0,q')$. Then fix a parameter $\nu\in\Theta$
such that for $i \in [m-1]$ and $j \in [n]_0$, $\nu_{i,j}^I = 0$ if $|I|>q'$. 
Now consider the following path-dependent~RDE,
\begin{align}
\label{eq: identifiability 1}
& \begin{dcases}
  \pmat{dY_t^{(0)} & \cdots & dY_t^{(m-1)}}^T  = F^{0, \ldots, m-1}_\nu(\p{Y}_t) d\p{X}_t,\\
  dY^{(m)}_t = \sum_{I \in \cW([m-1]_0, q')} \lambda^I \p{Y}_t^{I^{-}} d\p{Y}_t^{I^f},
\end{dcases} 
\end{align}
where $F^{0, \ldots, m-1}_\nu$ consists of the rows $0, \ldots, m-1$ of $F_\nu$.
Thus, the coordinates $0, \ldots, m-1$ evolve with respect to a linear signature RDE parametrized by $\nu$, while the final coordinate $Y^{(m)}_t$ is given in terms of the functional $H(\p{Y}_t)$.
Note that using the shuffle identity, \eqref{eq: identifiability 1} is a linear signature RDE and can be written in the form of~\eqref{eq: linear signature RDE} with respect to some $\theta_1\in \Theta$. \medskip

To define a second linear signature RDE with distinct parameters which yields the same solutions as~\eqref{eq: identifiability 1}, we introduce ``hidden dynamics'' to the system. In particular, take $\kappa \in \Theta$ such that for $i\in [m]$ and $j\in [n]_0$,  $\kappa_{i,j}^I =0$ if $|I|>q'$. We define 
\begin{align}
    \begin{dcases}\label{eq: identifiability 2}
        dY_t^{(0)} = dt, \\
        \pmat{dY_t^{(1)} & \cdots & dY_t^{(m)}}^T = \left(F^{1,\ldots, m}_{\theta_1} (\p{Y}_t) + \left(\p{Y}_t^{(m)} - H(\p{Y}_t) \right) F^{1,\ldots, m}_\kappa(\p{Y}_t) \right) d\p{X}_t,
    \end{dcases}
\end{align}
where $F^{1,\ldots, m}_\theta(\p{Y}_t)$ excludes the zeroth row of $F_\theta(\p{Y}_t)$. Using the shuffle identity, \eqref{eq: identifiability 2} is also a linear signature RDE, i.e. for some $\theta_2\in \Theta$, it can be written in the form of~\eqref{eq: linear signature RDE}. We note that if $\p{Y}$ is a solution to ~\eqref{eq: identifiability 1}, then $\p{Y}^{(m)}_t = H(\p{Y}_t)$ for all $t\in [0,T]$, so $\p{Y}$ is also a solution to~\eqref{eq: identifiability 2}.

\begin{proposition} \label{prop: non-identifiability}
For a sufficiently bounded driving signal trajectory $\p{X}$ satisfying 
\begin{align} \label{eq: identifiability bound}
    d_p(\p{X}, \p{0}) < \min\{\cN(\|\theta_1\|), \cN(\|\theta_2\|)\},
\end{align}
the $\cM(\|\theta_1\|)$-solution 
to~\eqref{eq: identifiability 1} 
coincides with the $\cM(\|\theta_2\|)$-solution to ~\eqref{eq: identifiability 2}. Therefore, the distributions of $\p{Y}_{{\theta_1}} \cdot \chi_{E_\xi}$ and $\p{Y}_{{\theta_2}}\cdot \chi_{E_\xi}$ are the same if $\min\{\cN(\|\theta_1\|), \cN(\|\theta_2\|)\}\geq \xi$.  
\end{proposition}
\begin{proof}
The proof is given in \Cref{appendix: proofs}.
\end{proof}
\begin{example}
    Define 
    $H(\p{Y}_t) \coloneqq \frac12(\p{Y}_t^{(1,2)} - \p{Y}_t^{(2,1)})$, 
    and consider the following path-dependent RDEs:
    \begin{align} \label{eq: non-identifiablility example 1}
        &\begin{dcases}
            d{Y}_t^{(0)} = dt,\\
            d{Y}_t^{(1)} = \p{Y}_t^{(1)} dt + d\p{W}_t^{(1)},\\
            d{Y}_t^{(2)} = -\p{Y}_t^{(2)} dt + d\p{W}_t^{(2)}, \\
            d{Y}_t^{(3)} = \left(-\p{Y}_t^{(1,2)} - \p{Y}_t^{(2,1)} \right) dt - \frac12 \p{Y}_t^{(2)} \circ d\p{W}_t^{(1)}+ \frac12 \p{Y}_t^{(1)}\circ d\p{W}_t^{(2)}.
        \end{dcases}\\
        \label{eq: non-identifiablility example 2}
        &\begin{dcases}
            d{Y}_t^{(0)} = dt,\\
            d{Y}_t^{(1)} =  \left(\p{Y}_t^{(1)} + H(\p{Y}_t) -\p{Y}_t^{(3)} 
            \right) dt + d\p{W}_t^{(1)},\\
            d{Y}_t^{(2)} = \left(-\p{Y}_t^{(2)}  + H(\p{Y}_t) -\p{Y}_t^{(3)}
            \right) dt + d\p{W}_t^{(2)},\\
            d{Y}_t^{(3)} = \left(-\p{Y}_t^{(1,2)} - \p{Y}_t^{(2,1)} + H(\p{Y}_t) -\p{Y}_t^{(3)} 
            \right) dt - \frac12 \p{Y}_t^{(2)}\circ d\p{W}_t^{(1)} +\frac12 \p{Y}_t^{(1)}\circ d\p{W}_t^{(2)}.
        \end{dcases}
    \end{align}
    Using the Stratonovich midpoint scheme with step size 0.001, we numerically compute the underlying paths in the solutions to these two RDEs corresponding to a sampled trajectory of the Brownian motion on the interval $[0,2]$. See \Cref{fig:non-identifiability}. Note that after a certain point of time, the bound condition in~\eqref{eq: identifiability bound} on the driving signal is no longer satisfied. As a result, we can no longer guarantee that the solutions will coincide. 
    \begin{figure}[t]
  \centering
  \begin{subfigure}[b]{0.48\textwidth}
    \centering
    \includegraphics[width=0.7\textwidth]{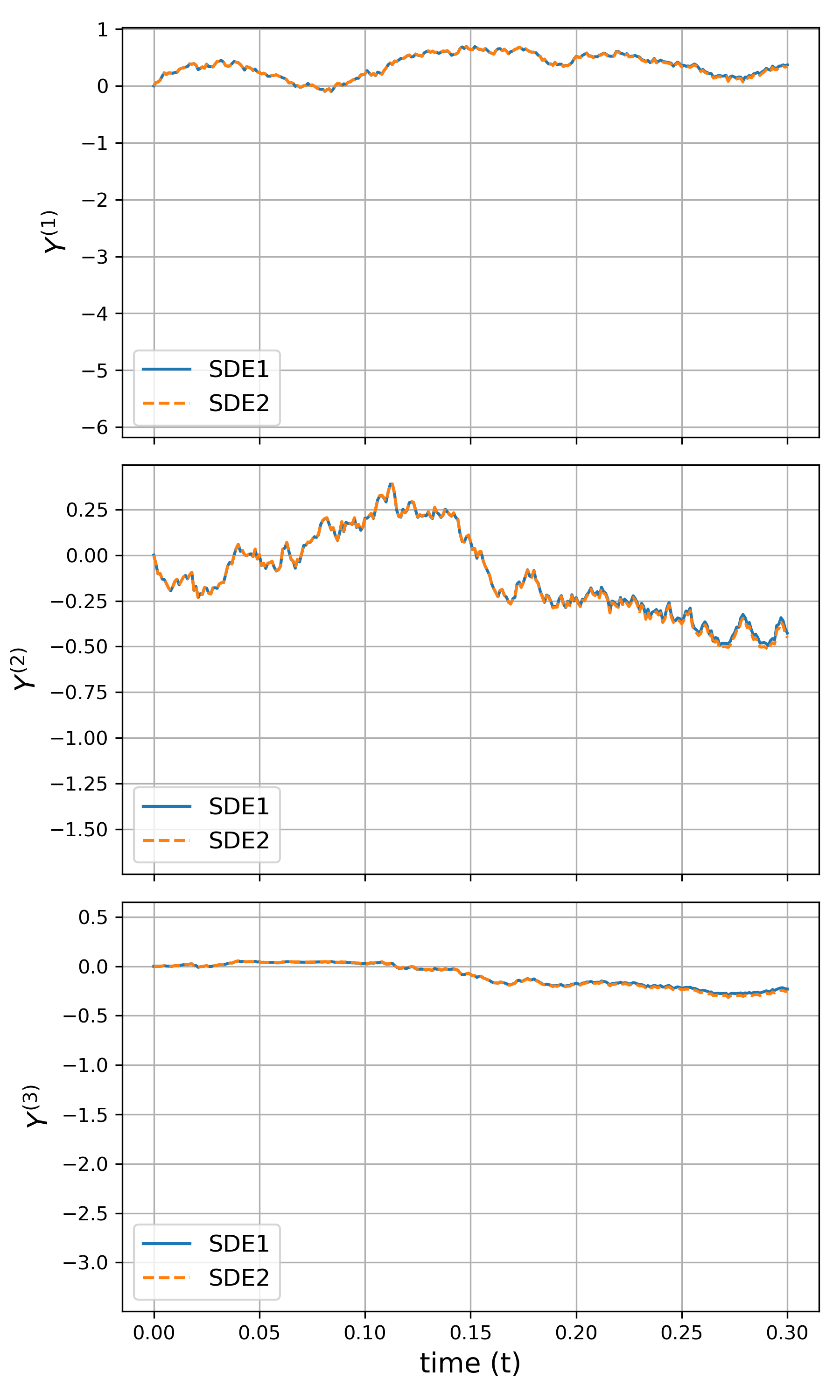}
    \caption{Trajectories observed until $T=0.3$}
    \label{fig:non-identifiability-01}
  \end{subfigure}
  ~
  \begin{subfigure}[b]{0.48\textwidth}
    \centering
    \includegraphics[width=0.7\textwidth]{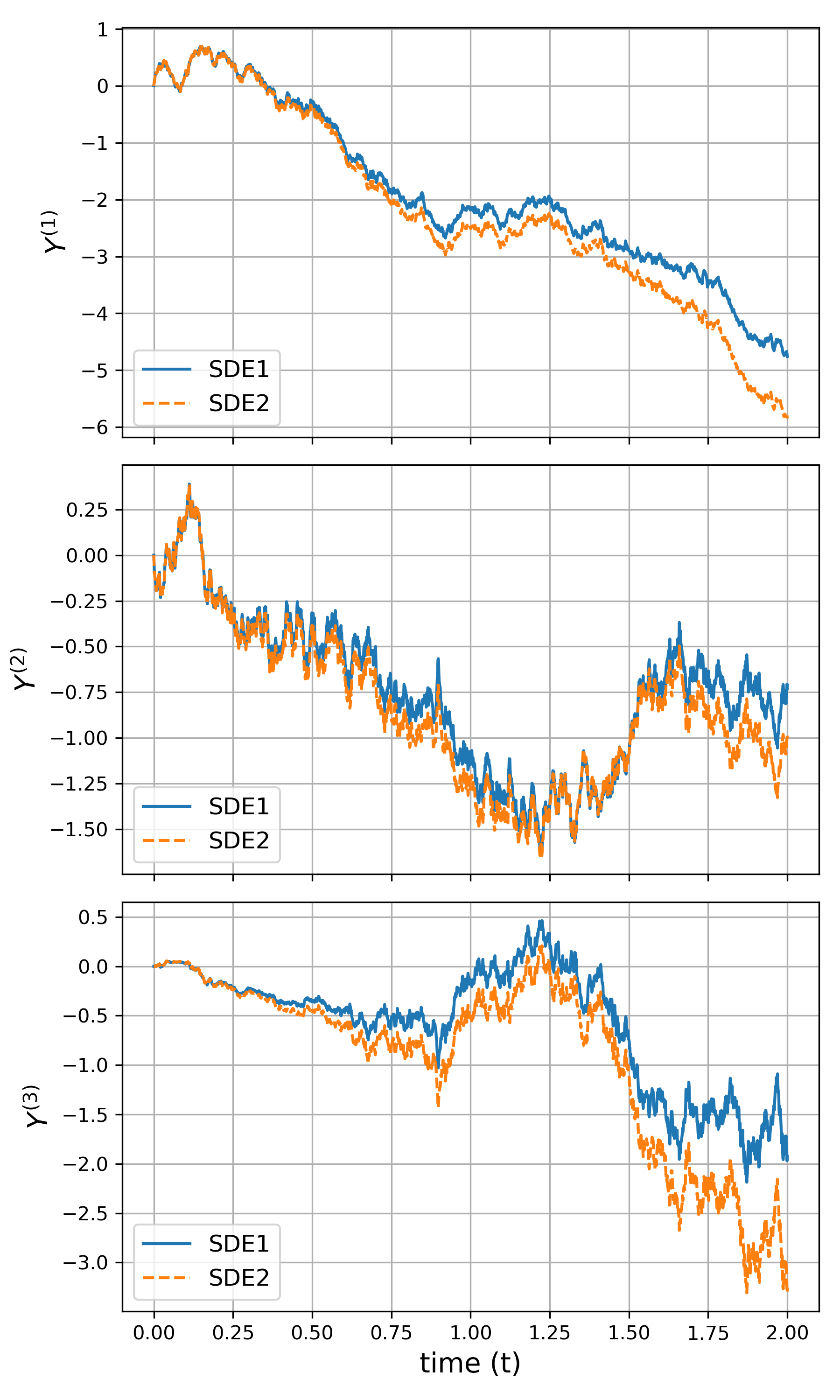}
    \caption{Trajectories observed until $T=2.0$}
    \label{fig:non-identifiability-02}
  \end{subfigure}

  \caption{The underlying paths in the sampled solutions to the path-dependent SDEs in~\eqref{eq: non-identifiablility example 1} and \eqref{eq: non-identifiablility example 2} are plotted for each component of the paths over the interval $[0,T]$ for (a) $T=0.3$ and (b) $T=2.0$. The $L^2$ distance between the two trajectories evaluated at $T$ is (a) $0.045$ and (b) $1.743$. }
  \label{fig:non-identifiability}
\end{figure}
\end{example}

\appendix

\clearpage

\section{Notation and Conventions} \label{apx:notation}
{\footnotesize
\begin{longtable}
    {ccc}
     \toprule
    Symbol & Description & Page\\ \midrule
    & Fixed Parameters & \\ \midrule
    $m$ & dimension of solution (excluding time parameter) &  \\
    $n$ & dimension of driving noise (excluding time parameter) & \\
    $q$ & truncation level of signature for signature vector field & \pageref{pg:p,gamma,q}\\ 
    $p$ & $p$-variation of a path & \pageref{pg:p,gamma,q}\\
    $r$ & depth of Picard iteration & \\
    $N$ & number of sampled trajectories & \\
    $T$ & length of time interval $[0,T]$ & \\
    $\gamma$ & constant for $\Lip(\gamma)$ functions & \pageref{pg:p,gamma,q}\\
    $M$ & radius of restriction ball & \\ 
    $\xi$ & parameter used to fix uniform bounds on noise & \pageref{pg:xi}\\
    $d$ & dimension of the unknown component of the parameter space & \pageref{pg:d}\\
    $\cN, \cM$ & noise bounded by $\cN(\xi)$ results in solutions bounded by $\cM(\xi)$ & \pageref{prop: N(theta) and M(theta)}\\
    \midrule 
    & Sets and Spaces &\\ \midrule
    $[n], [n]_0$ & finite sets $[n] \coloneqq \{1, \ldots, n\}$ and $[n]_0 \coloneqq \{0, \ldots, n\}$ & \\
    $(\samplespace, \mathbf{\mathcal{F}}, \mathcal{P})$ & probability space & \pageref{pg:prob_space}\\
    $\cW(\cA, \ell)$ & words in a set $\cA$ of length at most $\ell$ & \pageref{pg:cW_def}\\
    $U, V$ & Banach spaces for driving signal ($U$) and solution ($V$) & \pageref{pg:U,V}\\
    $T(V), T\ps{V}$ & tensor algebra and its completion & \pageref{pg:tensor_alg}\\
    $H\ps{V}$ & Hilbert space completion of tensor algebra & \pageref{pg:hilbert_alg}\\
    $\ssol$ & shorthand for $\ssol \coloneqq T^{(\leq q)}(V)$ & \pageref{pg:ssol}\\
    $B_\varepsilon (\theta)$  &  open ball of radius $\varepsilon$ centered at $\theta$ with respect to $\|.\|_\infty$ & \\
    $\GRP_p(V)$ & geometric $p$-rough paths on $V$ & \pageref{def:grp}\\
    $\GRP_p^{\tp}(\R \times V)$ & time parametrized geometric $p$-rough paths & \pageref{rem:time_parametrization} \\ 
    \midrule 
    & Notation for RDE/SDEs & \\
    \midrule
    $S$ & path signature for a (rough) path &  \pageref{pg:sig_smooth}, \pageref{pg:sig_rough}\\
    $\p{0}_U$ & the trivial rough path $\p{0}_U \coloneqq (1, 0, \ldots 0)$ in $\GRP_p(U)$ & \pageref{def: trivial rough path}\\
    $\bone$ & multiplicative identity in $\ssol$ & \pageref{eq:multiplicative identity}\\
    $\pi_U$ & projection map $\pi_U : T\ps{U \oplus V} \to T\ps{U}$ for tensor algebras & \pageref{pg:projection}\\
    $\tens, \tens_M$ & tensor product and its modification $\tens, \tens_M: \ssol \to \LL(\sol, \ssol)$  & \pageref{pg:tens},\pageref{def:modified_tensor_mult}\\
    $K_M$ & $M$-ball in truncated tensor algebra $K_M \subset \ssol$ & \pageref{def:modified_tensor_mult} \\
    $G_M$ & subset $G_M \subset \GRP_p(U)$ of driving noise such that the $M$-solution is bounded by $M$ & \pageref{thm: universal limit theorem 2} \\
    $I_{F,M}$ & lifted $M$-solution map $I_{F,M} : \GRP_p(U) \to \GRP_p(\ssol)$ for RDE with vector field $F$ & \pageref{def:I_FM}\\
    $J_{F,M}$ &  $M$-solution map $J_{F,M} : G_M \to \GRP_p(\sol)$ for RDE with vector field $F$ & \pageref{pg:J_FM}\\
    $A_\theta, B_\theta, F_\theta$ & parametrized vector field $F_\theta : \ssol \to \LL(\drive, \sol)$ with drift $A_\theta$ and diffusion $B_\theta$ & \pageref{pg:parametrized_vf} \\
    $\p{F}_{\theta, M}$ & lifted vector field $\p{F}_{\theta, M} : \ssol\to \LL(\drive, \ssol)$ & \pageref{eq:bounded_lifted_vf}  \\ 
    \midrule 
    & Notation for Parameter Estimation &\\
    \midrule
    $\p{\theta_0} =(\known\theta, \theta_0)$ & true parameter for the vector field, decomposed into known and unknown components & \pageref{pg:decomposition of the true parameter} \\
    $\known{\Theta}, \unknown{\Theta}$ & subspaces consisting of known and unknown components of the parameter space & \pageref{pg:known_unknown}\\
    $\Theta_\xi$ & subset of parameters $\Theta_\xi \subset \unknown{\Theta}$ allowable with respect to $\xi$ & \pageref{pg:allowable}\\
    $E_\xi$ & subset of samples $E_\xi \subset \samplespace$ allowable with respect to $\xi$ & \pageref{pg:allowable}\\
    $\chi_{E_\xi}$ & indicator function on the set $E_\xi$ & \\
    $P_r(\theta), P(\theta)$ & expectation of Picard iteration/solution as functions in $\theta$ & \pageref{eq:picard_expsig_polys},\pageref{eq:P_def}\\
    $Q(r, \ell)$ & maximum degree of polynomial for depth $r$ Picard iteration & \pageref{thm: polynomial equations}\\
    $\J_f(\theta)$ & Jacobian of $f$ at $\theta$ & \\
    \bottomrule
    
\end{longtable}
}

The following are some conventions for (rough) paths.
\begin{itemize}
    \item Paths are denoted using unbold capital letters $Y: [0,T] \to \sol$.
    \item Elements/paths valued in the (truncated/completed) tensor algebra are denoted using bold capital letters; for instance a rough path $\p{Y} \in \GRP_p(V)$.
    \item Paths valued in the (truncated/completed) tensor algebra of the tensor algebra are denoted using calligraphic capital letters; for instance a rough path of signatures $\pp{Y} \in \GRP_p(\ssol)$.
\end{itemize}

\section{Proofs of Lemmas and Propositions} \label{appendix: proofs}
\begin{proof}[Proof of \Cref{lem:product_lipgamma}]
    For all $k = 0,\ldots, \gamma_0$, to obtain $h^{(k)}(u) \in \LL(U^{\otimes k}, \LL(V_1,V_3))$ for $u\in U$, we apply the general Leibniz rule \cite[p. 318]{olver1993applications} and get
     \begin{align} \label{eq: lip(gamma) 1}
        h^{(k)}(u)(e_{I}) = \sum_{i=0}^k \sum_{
        \substack{J \subset [k] \\ |J| = i}}  g^{(i)}(u)(e_{I_J}) \cdot f^{(k-i)}(u)(e_{I_{[k]\setminus J}}) \in \LL(V_1, V_3)
    \end{align}
    for basis elements $e_{I}$ of $U^{\otimes k}$, as defined in~\eqref{eq: orthonormal basis}. Therefore, for some constant $K_1>0$,
    \begin{align} \label{eq: lip(gamma) 2}
        \|h^{(k)}(u)\|\leq  K_1 (k+1) 2^k  \|g\|_{\Lip(\gamma)} \|f\|_{\Lip(\gamma)}.
    \end{align}
     Then, using \eqref{eq: lip(gamma) 1} again, for $u_1,u_2\in U$, we have
    \begin{align} 
        \left\|\left(h^{(\gamma_0)}(u_1) - h^{(\gamma_0)}(u_2)\right)(e_I)\right\| &= \Bigg\|\sum_{i=0}^{\gamma_0} \sum_{\substack{J \subset [\gamma_0 ], \  |J|=i}} \bigg(g^{(i)}(u_1) (e_{I_J})\cdot f^{(\gamma_0-i)}(u_1)(e_{I_{[\gamma_0 ]\setminus J}}) \\ &
        \qquad -g^{(i)}(u_2)(e_{I_J}) \cdot f^{(\gamma_0-i)}(u_2) (e_{I_{[\gamma_0]\setminus J}}) \bigg)\Bigg\| \\
        &\le 2^{\gamma_0}\sum_{i=0}^{\gamma_0} \left\|g^{(i)}(u_1)-g^{(i)}(u_2)\right\|\left\|f^{(\gamma_0-i)}(u_1)\right\|  \\ & \qquad +2^{\gamma_0}\sum_{i=0}^{\gamma_0}\left\|f^{(\gamma_0-i)}(u_1)-f^{(\gamma_0-i)}(u_2)\right\|\left\|g^{(i)}(u_2)\right\|   \\ 
        &\le 2\cdot 2^{\gamma_0} \|f\|_{\Lip(\gamma)} \|g\|_{\Lip(\gamma)}\left(\gamma_02^{1 - (\gamma - \gamma_0)} +1\right)\|u_1 - u_2\|^{\gamma - \gamma_0},
    \end{align}
    where in the last line we use the fact that for all $i = 0,\ldots, \gamma_0-1$,
    \begin{align}
        \|f^{(i)}(u_1) - f^{(i)}(u_2)\| &= \|f^{(i)}(u_1) - f^{(i)}(u_2)\|^{1-(\gamma - \gamma_0)} \cdot \|f^{(i)}(u_1) - f^{(i)}(u_2)\|^{\gamma - \gamma_0} \\
        & \leq \left(2\|f\|_{\Lip(\gamma)}\right)^{1-(\gamma - \gamma_0)} \left(\|f\|_{\Lip(\gamma)} \|u_1 - u_2\|\right)^{\gamma - \gamma_0}. 
    \end{align}
   Therefore, for some constant $K_2>0$, we have 
    \begin{align} \label{eq: lip(gamma) 3}
        \left\|h^{(\gamma_0)}(u_1) - h^{(\gamma_0)}(u_2)\right\| \leq K_2 \|f\|_{\Lip(\gamma)} \|g\|_{\Lip(\gamma)}\|u_1 - u_2\|^{\gamma - \gamma_0}.
    \end{align}
    Equations~\eqref{eq: lip(gamma) 2} and \eqref{eq: lip(gamma) 3} conclude the proof. 
\end{proof}

\begin{proof}[Proof of \Cref{prop: signatures of the lift}]
First assume $X$ is 
    a bounded variation path, i.e. $p=1$. Let $\pp{Z}\in \GRP_1(\drive \oplus \ssol)$ be as in \Cref{def:M_solution}. By definition of the coupled $M$-solution, for all $(s,t) \in \Delta_T$, we have
    \begin{align} \label{eq: level 1 signature}
        \pp{Z}_{s,t}^1 = \int_s^t h_M(\pp{Z}_{0,u}^1) \ d\pp{Z}_{0,u}^1.
    \end{align}
    Therefore, 
    \begin{align}
        \pp{Y}_{s,t}^{(\emptyset)} =  \pp{Z}_{s,t}^{(\emptyset)} =  \int_{s}^t 0 = 0.
    \end{align}
    Now we prove the first equality in \eqref{eq: signatures of the lift} by induction on $k$. If $k=1$, this equality trivially holds. Assume it is also true for all $k\in [\ell]$ for some $1\leq \ell < q$. We will prove it is then true when $k=\ell+1$. 
    For $\pp{Y} = \pi_{\ssol}(\pp{Z}) \in \GRP_1(\ssol)$, we have
    \begin{align}
        \pp{Y}_{s,t}^{\left((i_1,\ldots,i_{\ell+1})\right)}
        &= \pp{Z}_{s,t}^{\left((i_1,\ldots,i_{\ell+1})\right)} &&\text{(Since $\pi_{\ssol}(\pp{Z}) = \pp{Y}$)}\\
        &= \int_s^t \pp{Z}_{0,u}^{\left((i_1,\ldots,i_{\ell})\right)} \ d\pp{Z}_{0,u}^{\left((i_{\ell+1})\right)} && \text{(\eqref{eq: level 1 signature}, and $\|\pi_{\tilde{V}}(\pp{Z})_{0,u}^1 + \bone \| < M \ \ \forall u \in [s,t]$)}\\
        &= \int_s^t \pp{Y}_{0,u}^{\left((i_1,\ldots,i_{\ell})\right)} \ d\pp{Y}_{0,u}^{\left((i_{\ell+1})\right)} && \text{(Since $\pi_{\ssol}(\pp{Z}) = \pp{Y}$)}\\
        &= \int_s^t \pp{Y}_{0,u}^{((i_1),\ldots,(i_{\ell}))} \ d\pp{Y}_{0,u}^{((i_{\ell+1}))} &&\text{(By the induction hypothesis)}\\
        &= \pp{Y}_{s,t}^{((i_1),\ldots,(i_{\ell+1}))} && \text{(Since $\pp{Y}$ is a bounded variation path)}.
    \end{align}
    This concludes the proof of the first equality in \eqref{eq: signatures of the lift} in the bounded 1-variation case.
    
    Now let $\p{X}\in \GRP_p(\drive)$. Then there exists a sequence $\{\p{X}(n)\}_{n\in\mathbb{N}} \subset \GRP_p(\drive)$ of extensions of bounded variation paths such that $d_p(\p{X}(n), \p{X}) \to 0$ as $n\to \infty$. Therefore, we get that $d_p(\pp{Y}(n), \pp{Y}) \to 0$ as $n\to \infty$, where $\pp{Y}(n) \coloneqq I_{F,M}(\p{X}(n))$. Since for all $t\in [0,T]$, we have $\|\p{Y}_{0,t}\| =\|\pp{Y}^1_{0,t}+ \bone\| <M$,  there exists $N>0$ such that for all $n\geq N$, we have $ \|\pp{Y}(n)_{0,t}^1 + \bone\| < M$ for all $t\in [0,T]$. Hence, for all $n\geq N$,
    \begin{align}
    \pp{Y}(n)_{s,t}^{(\emptyset)} = 0, \qquad
    \pp{Y}(n)_{s,t}^{\left((i_1),\ldots,(i_k) \right)} = \pp{Y}(n)_{s,t}^{\left((i_1,\ldots,i_k)\right)}.
    \end{align}
    Taking the limit as $n\to \infty$ proves
    \begin{align} \label{eq: the solution is a projection}
    \p{Y}_{s,t}^{\emptyset} = 1 + \pp{Y}_{s,t}^{(\emptyset)} = 1, \qquad
    \p{Y}_{s,t}^{(i_1,\ldots,i_k)}= \pp{Y}_{s,t}^{\left((i_1,\ldots,i_k)\right)} = \pp{Y}_{s,t}^{\left((i_1),\ldots,(i_k) \right)} 
    \end{align}
    for arbitrary $\p{X}\in \GRP_p(V)$, where we use the continuity of the extension in~\cite[Theorem 3.10]{lyons_differential_2007} when $\lfloor p \rfloor < k \leq q$. Equation \eqref{eq: the solution is a projection} proves that $\p{Y} = \pi_V(\pp{Y})$ as desired.
\end{proof}

\begin{proof}[Proof of \Cref{prop: N(theta) and M(theta)}]
Our first step is to define a control $\omega$ for each $\p{X} \in \GRP_p(\drive)$ which we will later bound. 
Recall that for $s <t$, we define $\Delta_{s,t} \coloneqq \{(u_1,u_2) : s\leq u_1 < u_2 \leq t\}$. 
We claim that
\begin{align}
    \omega : \Delta_{T} \to \mathbb{R}\quad \text{defined by} \quad \omega(s,t) = C\  d_p(\p{X}|_{\Delta_{s,t}},\mathbf{0})^p
\end{align}
is a control for the $p$-variation of $\p{X}\in \GRP_p(\drive)$. Indeed, let $0\leq s\leq u \leq t \leq T$. Then for any $i \in [\lfloor p \rfloor]$, any partition $s = a_0 < a_1 <\cdots < a_\ell = u$, and any partition $u = a_{\ell} < a_{\ell+1} < \cdots < a_{\ell+\ell'}=t$, we have 
\begin{align}
    \sum_{k=1}^\ell  \left\| \p{X}_{a_{k-1},a_k}^i \right\|^\frac{p}{i}+ \sum_{k=1}^{\ell'}  \left\| \p{X}_{a_{\ell+k-1},a_{\ell+k}}^i \right\|^\frac{p}{i} \leq \max_{1\leq i \leq \lfloor p \rfloor} \sup_{\substack{s = b_0 < b_1 < \cdots < b_{\ell''} = t,\\ \ell'' \in \mathbb{N}}} \sum_{k=1}^{\ell''}  \left\| \p{X}_{b_{k-1},b_k}^i \right\|^\frac{p}{i}  = d_p(\p{X}|_{\Delta_{s,t}},\mathbf{0})^p.
\end{align}
Therefore, $\omega$ is super-additive, i.e.
\begin{align}
    d_p(\p{X}|_{\Delta_{s,u}},\mathbf{0})^p + d_p(\p{X}|_{\Delta_{u,t}},\mathbf{0})^p \leq d_p(\p{X}|_{\Delta_{s,t}},\mathbf{0})^p.
\end{align}
Moreover, for all $(s,t) \in \Delta_T$, we have
\begin{align}
    \left\| \p{X}_{s,t}^i \right\| \leq \left(d_p(\p{X}|_{\Delta_{s,t}},\mathbf{0})^p\right)^{\frac{i}{p}} \leq \frac{\omega(s,t)^{\frac{i}{p}}}{\beta \left( \frac{i}{p}\right) !}.
\end{align}
This concludes the proof that $\omega$ is a control for the $p$-variation of $\p{X}$. \medskip

Next, our aim is to determine the conditions under which the bounds on the Picard iterations in part (4) of~\Cref{thm: universal limit theorem 2} hold for the entire interval $[0,T]$.
In order to do so, we must adapt intermediate steps of the proof of the Universal Limit Theorem in~\cite[Theorem 5.3]{lyons_differential_2007}. In particular, in~\cite[Section 5.5]{lyons_differential_2007}, we must consider for $i = 0,1,2$ three paths $\p{Z}_i \in \GRP_p(B_i)$ on Banach spaces $B_i$ and associated vector fields $H_{\theta, M}^i: B_i \to \LL(B_i, B_i)$. In our setting, the vector fields are determined by $h_{\theta, M}$ from~\eqref{eq:h_m_def}, where the $\theta$ denotes the additional dependence on $\theta$ of our parametrized vector fields in~\eqref{eq:bounded_lifted_vf}. In order to obtain $T_\rho$ in part (4) of~\Cref{thm: universal limit theorem 2}, one must consider the $p$-variation of
\begin{align}
    \int H^i_{\theta, M}(\p{Z}_i) d\p{Z}_i.
\end{align}
By \cite[Theorem 4.12]{lyons_differential_2007}, there exists a function $K: \mathbb{R}_{\geq 0} \times \mathbb{R}_{\geq 0} \times \mathbb{R}_{\geq 0} \to \mathbb{R}_{\geq 0}$ such that for $i=0,1,2$ and all $\p{Z}_i\in \GRP_p(B_i)$ with $p$-variation controlled by some control $\widehat\omega$ and $\widehat\omega(0,T)\leq 1$, the $p$-variation of $\int H_i^M(\p{Z}_i) d\p{Z}_i$ is controlled by 
\begin{align}
K\left(\left\|H^0_{\theta, M}\right\|_{\Lip},\left\|H^1_{\theta, M}\right\|_{\Lip},\left\|H^2_{\theta, M}\right\|_{\Lip} \right)\widehat\omega.
\end{align}
By the proof of \cite[Theorem 4.12]{lyons_differential_2007}, $K$ can be considered to be continuous and non-decreasing in all of the variables. 
Note that there exists a function $I: \mathbb{R}_{\geq 0} \times \mathbb{R}_{\geq 0} \to \mathbb{R}_{\geq 0}$ such that $I$ is continuous and non-decreasing in both variables, and   
\begin{align}
    \|h_{\theta, M}\|_{\Lip} \leq I(M,\|\theta\|).
\end{align}
Then, since the norms of the vector fields $H^i_{\theta, M}$ are bounded by continuous and non-decreasing functions of $\|h_{\theta, M}\|_{\Lip}$, we can define a continuous and non-decreasing (in both variables) function $J: \mathbb{R}_{\geq 0} \times \mathbb{R}_{\geq 0} \to \mathbb{R}_{\geq 0}$ such that
\begin{align}
    \max \left\{\left\|H_{\theta,M}^0\right\|_{\Lip}, \left\|H_{\theta,M}^1\right\|_{\Lip}, \left\|H_{\theta,M}^2\right\|_{\Lip}\right\} \leq J(M,\|\theta\|).
\end{align}
Therefore, for $i = 0,1,2$, the $p$-variation of $\int H_{M,\theta}^i(\p{Z}) d\p{Z}$ is also controlled by 
\begin{align}
    \widehat{K} (M, \|\theta\|) \widehat\omega \quad \text{where} \quad \widehat{K}(M, \|\theta\|) \coloneqq K \left( J(M, \|\theta\|), J(M, \|\theta\|),J(M, \|\theta\|) \right).
\end{align}
Now by the proof of the Universal Limit Theorem in~\cite[Page 89]{lyons_differential_2007}, if
\begin{align} \label{eq:full_interval_condition}
    \omega(0,T) \leq \frac{1}{\widehat{K}(M,\|\theta\|)^{\lfloor p \rfloor}},
\end{align}
then part (4) of~\Cref{thm:universal_limit} holds for $T_\rho = T$. In our setting, where we consider the level 1 component as in part (4) of~\Cref{thm: universal limit theorem 2}, we have in particular for all $(s,t)\in \Delta_T$ and all $r\in \mathbb{Z}_{\geq 0}$,
\begin{align} \label{eq:full_interval_picard1}
    \left\| \p{Y}_{\theta,M}(r)_{s,t} - \p{Y}_{\theta,M}(r+1)_{s,t} \right\| \leq 2 \rho^{-r} \frac{\omega(0,T)^{\frac{1}{p}}}{\beta \left(\frac{1}{p}\right)!} \leq 2 \rho^{-r} \frac{C^{1/p}  d_p(\p{X},\mathbf{0})}{\beta \left(\frac{1}{p}\right)!}.
\end{align}
Then, this implies that
\begin{align} \label{eq:full_interval_picard2}
    \left\|\p{Y}_{\theta,M}(r)_{s,t}\right\| \leq \left\|\p{Y}_{\theta,M}(0)_{s,t}\right\| + 2 \frac{C^{1/p}  d_p(\p{X},\mathbf{0})}{\beta \left(\frac{1}{p}\right)!} \sum_{i=0}^{r-1} \rho^{-i} < 1 + 2 \frac{C^{1/p}  d_p(\p{X},\mathbf{0})}{(1-\rho^{-1})\beta \left(\frac{1}{p}\right)!}.
\end{align}
Now by restricting to sufficiently bounded $\p{X}$, we wish to obtain uniform bounds on the values $\left\|\p{Y}_{\theta,M}(r)_{s,t}\right\|$ which are valid for all $(s,t)\in \Delta_T$. To do so, we define the function 
\begin{align}
    \cF : (N, \|\theta\|) \mapsto C N^p \widehat{K}\left(1 +\frac{ 4C^\frac{1}{p}}{\left(1-\rho^{-1}\right) \beta \left(\frac{1}{p}\right)!} N, \|\theta\|\right)^{\lfloor p \rfloor}.
\end{align}
For all $\theta \in \Theta$, since $\cF(0, \|\theta\|) = 0$, there exists $N>0$ with  $\cF(N,\|\theta\|)\leq 1$ by continuity. Fix $N_0>0$ and set
\begin{align}
    \cN(\mu) \coloneqq \sup \{N \in (0,N_0] : \cF(N, \mu ) \leq 1\} \andd
    \cM(\mu) \coloneqq 1 + \frac{4C^\frac{1}{p}}{\left(1-\rho^{-1}\right) \beta \left(\frac{1}{p}\right)!} \cN(\mu).
\end{align}
Now, consider $\p{X}$ such that $d_p(\p{X}, \mathbf{0})\leq \cN(\|\theta\|)$. Then 
\begin{align} \label{eq: T_rho = T}
    \omega(0,T) \ \widehat{K} (\cM(\|\theta\|), \|\theta\|)^{\lfloor p \rfloor} \leq \cF(\cN(\|\theta\|), \|\theta\|) \leq 1.
\end{align}
In particular, the condition in~\eqref{eq:full_interval_condition} holds. 
Then in this setting, by~\eqref{eq:full_interval_picard2}, for all $(s,t)\in \Delta_T$,
\begin{align}
    \left\|\p{Y}_{\theta, \cM(\|\theta\|)}(r)_{s,t}\right\| \leq 1 + 2 \frac{C^{1/p}  d_p(\p{X},\mathbf{0})}{(1-\rho^{-1})\beta \left(\frac{1}{p}\right)!} < \cM(\|\theta\|).
\end{align}
Furthermore, by part (3) of \Cref{thm: universal limit theorem 2}, this implies that
\begin{align}
    \left\|\left(\p{Y}_{\theta, \cM(\|\theta\|)}\right)_{s,t}\right\| \leq 1 + 2 \frac{C^{1/p}  d_p(\p{X},\mathbf{0})}{(1-\rho^{-1})\beta \left(\frac{1}{p}\right)!}  < \cM(\|\theta\|). 
\end{align}

To finish the proof, it only remains to show that $\cN$ and $\cM$ are non-increasing. 
Let $\mu_1,\mu_2\in 
\R_+$ with $\mu_1 \leq \mu_2$. Then $\cF(\cN(\mu_2),\mu_1) \leq \cF(\cN(\mu_2), \mu_2) \leq 1$, which means that $\cN(\mu_1) \geq \cN(\mu_2)$, and thus, $\cM(\mu_1)\geq \cM(\mu_2)$ as desired.
\end{proof}

\begin{proof}[Proof of \Cref{cor: uniform convergence}]
By \Cref{prop: N(theta) and M(theta)}, for all $\theta\in 
    \unknown\Theta$ and $\sample \in \samplespace$, if $d_p(\p{X}(\sample),\mathbf{0})\leq \cN(\|\p\theta\|)$, we have 
    \begin{align}
    \left\| \p{Y}_{\p{\theta}}(r)(\sample)^{I}_{0,T} - \p{Y}_{\p{\theta}}(r+1)(\sample)^{I}_{0,T} \right\| \leq 2 \rho^{-r} \frac{C^{\frac{1}{p}} d_p(\p{X}(\sample), \mathbf{0})}{\beta\left(\frac{1}{p}\right)!} \ 
            \text{ for all $r\in \mathbb{Z}_{\geq 0}$},
    \end{align}
    where constants $\rho$ and $C$ are as in \eqref{eq: C and rho}. So, by \Cref{thm: universal limit theorem 2},
    \begin{align}
    \left\| \p{Y}_{\p{\theta}}(r)(\sample)^{I}_{0,T} - \p{Y}_{\p{\theta}}(\sample)^{I}_{0,T} \right\| \leq  \frac{2 \rho^{-r}}{(1-\rho^{-1})\beta\left(\frac{1}{p}\right)!} C^{\frac{1}{p}} d_p(\p{X}(\sample), \mathbf{0}) \ 
            \text{ for all $r\in \mathbb{Z}_{\geq 0}$}.
    \end{align}
    For all $\theta \in \Theta_\xi$ and $\sample \in E_{\xi}$, we have $d_p(\p{X}(\sample), \bzero) < \xi \leq \cN(\|\p{\theta}\|)$. Hence, for all $\theta\in \Theta_\xi$ and $\sample \in \samplespace$, 
    \begin{align} \label{eq:pathwise_uniform_convergence_bound}
    \left\| \p{Y}_{\p{\theta}}(r)(\sample)^{I}_{0,T} \cdot \chi_{E_\xi}(\sample) - \p{Y}_{\p{\theta}}(\sample)^{I}_{0,T} \cdot \chi_{E_\xi}(\sample)\right\| \leq  \frac{2 \rho^{-r}}{(1-\rho^{-1})\beta\left(\frac{1}{p}\right)!} C^{\frac{1}{p}} \xi \ 
            \text{ for all $r\in \mathbb{Z}_{\geq 0}$}.
    \end{align}    
    Now taking expectations, we get that for all $\theta \in \Theta_\xi$ and $r\in \mathbb{Z}_{\geq 0}$,
    \begin{align}\label{eq: rate of uniform convergence}
        \left| P_{r}^I(\theta) - P^I(\theta) \right| \leq \mathbb{E} \left\| \p{Y}_{\p{\theta}}(r)^{I}_{0,T} \cdot \chi_{E_\xi} - \left(\p{Y}_{\p{\theta}}\right)^{I}_{0,T} \cdot \chi_{E_\xi}\right\| \leq  \frac{2 \rho^{-r}}{(1-\rho^{-1})\beta\left(\frac{1}{p}\right)!} C^{\frac{1}{p}} \xi.
    \end{align}
    Thus $\{P^I_r\}_{r\in \mathbb{Z}_{\geq 0}}$ converges uniformly  on $\Theta_\xi$ to $P^I$ as $r \to \infty$. But by \Cref{thm: polynomial equations}, the functions $P^I_{r}$ are polynomials, and therefore, continuous on $\Theta_\xi$. So $P^{I}$ is continuous as well.
\end{proof}

\begin{proof}[Proof of \Cref{lem:solution_given_miranda_domain}]
Without loss of generality, we can assume for all $k\in [d]$, 
    \begin{align}
        P^{I_k}(\theta_1) - P^{I_k}(\theta_0) >0 \text{ and } P^{I_k}(\theta_2) - P^{I_k}(\theta_0) <0 \text{ for all } \theta_1 \in D_k^+ \text{ and } \theta_2 \in D_k^-.
    \end{align}
    Note that all sets in $\mathcal{D}$ are compact by definition. For all $k\in [d]$, let 
    \begin{align}\label{eq: def of epsilon_k}
        \eta_k \coloneqq \min \left\{\min_{\theta_1 \in D_k^+} \frac{P^{I_k}(\theta_1) - P^{I_k}(\theta_0)}{2}, \min_{\theta_2\in D_k^-} \frac{P^{I_k}(\theta_0) - P^{I_k}(\theta_2)}{2} \right\} > 0.
    \end{align}
    By \Cref{cor: uniform convergence}, there exists $r_0 \in \mathbb{N}$ such that for all $r\geq r_0$ and all $k\in [d]$, 
    \[
        \sup_{\theta\in \Theta_\xi} \left|P_{r}^{I_k}(\theta) - P^{I_k}(\theta) \right| < \eta_k.
    \]
So, for all $\theta_1 \in D_k^+$ and all $\theta_2\in D_k^-$, we have 
\begin{align}
    &P^{I_k}_{r}(\theta_1)  - P^{I_k}(\theta_0) > P^{I_k} (\theta_1) - P^{I_k}(\theta_0) - \eta_k \geq \min_{\theta_1 \in D_k^+} \left(P^{I_k}(\theta_1) - P^{I_k}(\theta_0)\right) - \eta_k \geq \eta_k,\\
    &P^{I_k}(\theta_0)  - P^{I_k}_{r}(\theta_2) > P^{I_k} (\theta_0) - P^{I_k}(\theta_2) - \eta_k   \geq \min_{\theta_2 \in D_k^-} \left(P^{I_k}(\theta_0) - P^{I_k}(\theta_2)\right) - \eta_k \geq \eta_k.
\end{align}
On the other hand, by the Strong Law of Large Numbers, almost surely, there exists $N_0\in \mathbb{N}$ such that for all $N\geq N_0$ and all $k\in [d]$, we have
\begin{align}
    \left|\frac{1}{N}\sum_{i=1}^N \p{Y}_{\p{\theta_0}}(\sample_i)_{0,T}^{I_k} \cdot \chi_{E_\xi}(\sample_i) - P^{I_k}(\theta_0) \right| < \eta_k.
\end{align}
So, for all $r\geq r_0$, all $N\geq N_0$, all $k\in [d]$, all $\theta_1 \in D_k^+$, and all $\theta_2 \in D_k^-$, 
\begin{align}
    P_{r}^{I_k}(\theta_2) < P^{I_k}(\theta_0) - \eta_k < \frac{1}{N}\sum_{i=1}^N \p{Y}_{\p{\theta_0}}(\sample_i)_{0,T}^{I_k} \cdot \chi_{E_\xi}(\sample_i) < P^{I_k}(\theta_0) + \eta_k < P_{r}^{I_k}(\theta_1).
\end{align}
This means the continuous mapping $P_{r,N}: D \to \mathbb{R}^d$ defined by
\begin{align}
    P_{r,N}(\theta) \coloneqq \left(P^{I_1}_{r}(\theta) - \frac{1}{N}\sum_{i=1}^N \p{Y}_{\p{\theta_0}}(\sample_i)_{0,T}^{I_1} \cdot \chi_{E_\xi}(\sample_i), \ldots, P^{I_d}_{r}(\theta) - \frac{1}{N}\sum_{i=1}^N \p{Y}_{\p{\theta_0}}(\sample_i)_{0,T}^{I_d} \cdot \chi_{E_\xi}(\sample_i) \right)
\end{align}
satisfies the Miranda conditions on $(D, \mathcal{D})$. Therefore, by~\cite[Theorem 2.7]{miranda2002}, there exists $\theta_{r,N} \in D $ such that $P_{r,N}(\theta_{r,N}) = 0$, i.e. the system \eqref{eq: expected signature matching method} has a solution in $D$.
\end{proof}

\begin{proof}[Proof of \Cref{lem:existence_of_miranda}]
    This proof closely follows the idea of the proof of \cite[Theorem 3.1]{miranda2002}. Define 
    \begin{align}
        g : \mathbb{R}^d \to \mathbb{R}^d, \quad \theta \mapsto \left(\J_{P}(\theta_0)\right)^{-1} \theta + \theta_0.
    \end{align}
    Let $\tTheta_\xi \coloneqq g^{-1}(\Theta_\xi)$, and define
    \begin{align}
        h = (h^1, \ldots, h^d): \tTheta_\xi \xrightarrow{g} \Theta_\xi \xrightarrow{P} \R^d.
    \end{align}
    Moreover, define 
    \begin{align}
        R = (R^1, \ldots, R^d): \tTheta_\xi \to \mathbb{R}^d, \quad R(\theta) = h(\theta) - P(\theta_0) - \theta
    \end{align}
    Note that $g$ is differentiable, $g(0) = \theta_0$, and $P$ is differentiable at $\theta_0$. Thus, by the chain rule, $h$ is differentiable at $0$ and we get
    \begin{align}
        \J_{h}(0) = \J_{P}(g(0)) \J_g(0) = \J_{P}(\theta_0) \left(\J_{P}(\theta_0)\right)^{-1} = {\rm Id}_d,
    \end{align}
    where ${\rm Id}_d$ denotes the $d\times d$ identity matrix. This means that we have
    \begin{align}
     \frac{R(\theta)}{\|\theta\|_\infty} = \frac{h(\theta)-h(0) - \J_h(0) (\theta - 0)}{\|\theta\|_\infty} \andd \lim_{\theta \to 0} \frac{R(\theta)}{\|\theta\|_\infty} = 0
    \end{align}
    since $h$ is differentiable at $0$. Thus, there exists $\varepsilon_0>0$ such that $B_{\varepsilon_0}(\theta_0)\subset \Theta_\xi$, and for all $\theta = (\theta^1, \ldots, \theta^d)\in \tTheta_\xi$ with $\|\theta\|_\infty \leq \frac{\varepsilon_0}{2\left\| \left(\J_{P}(\theta_0)\right)^{-1} \right\|_\infty}$, we have $-\frac12 < \frac{R_k(\theta)}{\|\theta\|_\infty} < \frac12$ for all $k\in [d]$. This implies that
    \begin{align}
        -\frac12 \|\theta\|_\infty + \theta^k < h^k(\theta) - P^{I_k}(\theta_0) < \frac12 \|\theta\|_\infty + \theta^k \ \text{ for all } k\in [d].
    \end{align}
    Now let $\varepsilon \in (0, \varepsilon_0]$. Set 
    \begin{align} \label{eq:delta}
    \delta \coloneqq \frac{\varepsilon}{2\left\| \left(\J_{P}(\theta_0)\right)^{-1} \right\|_\infty}.  
    \end{align}
    Moreover, set $D \coloneqq g(\cube(\delta))$, $D_k^+ \coloneqq g(\cube_k(\delta)^+)$, and $D_k^- \coloneqq g(\cube_k(\delta)^-)$ for all $k\in [d]$. Then $D$ is a Miranda domain and $\mathcal{D} \coloneqq \{D_1^\pm, \ldots, D_d^\pm\}$ is a Miranda partition of $\partial D$. Note that for all $\theta\in \cube(\delta)$, we have
    \begin{align}
        \left\|g(\theta) - \theta_0 \right\|_\infty  = \left\| \left(\J_{P}(\theta_0)\right)^{-1} \theta\right\|_\infty \leq \left\| \left(\J_{P}(\theta_0)\right)^{-1} \right\|_\infty \delta < \varepsilon.
    \end{align}
    So, $D\subset B_\varepsilon (\theta_0)$. On the other hand, for all $k\in [d]$ and all $\theta_1\in D_k^+$, we have $\theta_1 = g(\theta_1')$ for some $\theta_1'\in \cube_k(\delta)^+$. So, 
    \begin{align} \label{eq: refined miranda 1}
        P^{I_k}(\theta_1) - P^{I_k}(\theta_0) = h^k(\theta_1') - P^{I_k}(\theta_0) > -\frac12 \|\theta_1'\|_\infty + (\theta_1')^k = -\frac{1}{2}\delta + \delta = \frac{1}{2} \delta. 
    \end{align}
    Similarly, for all $\theta_2\in D_k^-$, we have $\theta_2 = g(\theta_2')$ for some $\theta_2'\in I_k(\delta)^-$. So, 
    \begin{align} \label{eq: refined miranda 2}
        P^{I_k}(\theta_0) - P^{I_k}(\theta_2) = P^{I_k} (\theta_0) - h^k(\theta_2') > - \frac12 \|\theta_2'\|_\infty - (\theta_2')^k = -\frac12 \delta + \delta = \frac12 \delta.
    \end{align}
    This concludes the proof.
\end{proof}

\begin{proof}[Proof of \Cref{prop:P_diff_ae}]
Let $\hat\theta \in \Theta_{\xi}$. Consider the open neighborhood $U \coloneqq \{\theta \in \Theta_\xi: 2\|\hat\theta\| > \|\theta\| > \frac12 \|\hat\theta\| \}$ of $\hat\theta$.
     Recall that $P(\theta) = (P^{I_1}(\theta), \ldots, P^{I_d}(\theta))$, where $P^I(\theta) = \E[(\p{Y}_{\p{\theta}})_{0,T}^I \cdot \chi_{E_\xi}]$. Our aim is to show that $P^I$ is path-wise Lipschitz on $U$ with respect to $\theta$ by applying the locally Lipschitz property of RDE solutions in~\cite[Theorem 10.26]{friz_multidimensional_2010}. In particular, consider a driving signal $\bX$ in $E_\xi$, which is controlled by $\omega$. By the proof of \Cref{prop: N(theta) and M(theta)}, we can assume 
     \begin{align}
        \omega(0,T) = C d_p(\p{X}, \p{0})^p < C \xi^p. 
     \end{align} 
     Let $\theta_1,\theta_2\in U$. Then for $i=1,2$, we have $\cM(\|\theta_i\|) \leq \cM(\frac12 \|\hat\theta\|) = : M$, since by \Cref{prop: N(theta) and M(theta)}, the function $\cM$ is non-increasing. 
     Therefore, $\p{Y}_{\p{\theta_1}}$ and  $\p{Y}_{\p{\theta_2}}$ are the underlying paths in the solutions to the path-independent RDEs
     \begin{align}
         d\p{Y}_{\p{\theta_i}, t} = \p{F}_{\p{\theta_i},M} (\p{Y}_{
         \p{\theta_i},t
         }) d\p{X}_t \quad \p{Y}_{\p{\theta_i}, 0} = \bone,
     \end{align}
     for $i=1,2$ respectively, where
     \begin{align}
         \p{F}_{\p{{\theta_i}}, M}(\p{s}) = \tens_{M}(\p{s}) \circ F_{\p{\theta_i}, M}(\p{s}) \quad \text{for all $\p{s}\in \ssol$}.
     \end{align}
     The vector field $\p{F}_{\p{\theta_i},M}$ is $\Lip(\gamma)$ for any $\gamma>0$. So, without loss of generality, assume $\gamma>\max\{p,3\}$. Suppose $v ({\theta_1}, {\theta_2})> \|\p{F}_{\p{\theta_1}, M}\|_{\Lip(\gamma)}, \|\p{F}_{\p{\theta_2}, M}\|_{\Lip(\gamma)}$. Then, by~\cite[Theorem 10.26]{friz_multidimensional_2010}, we~have
    \begin{align} 
        \left\|(\p{Y}_{\p{\theta_1}})_{0,T} - (\p{Y}_\p{{\theta_2}})_{0,T} \right\| &\leq C' \|\p{F}_{\p{\theta_1}, M} - \p{F}_{\p{\theta_2}, M}\|_{\Lip(\gamma-1)} \omega(0,T)^{1/p} \exp(C' v ({\theta_1}, {\theta_2})^p \omega(0,T)) \\
        &\leq C' \|\p{F}_{\p{\theta_1}, M} - \p{F}_{\p{\theta_2}, M}\|_{\Lip(\gamma-1)} C^\frac{1}{p} \xi \exp(C' v ({\theta_1}, {\theta_2})^p C \xi^p),
    \label{eq:rde_locally_lipschitz}
    \end{align}
    where $C'> 0$ is a constant which depends only on $p$ and $\gamma$. Thus, it remains to show that $\|\p{F}_{\p{\theta}, M}\|_{\Lip(\gamma)}$ can be uniformly bounded over $\theta\in U$ and $\|\p{F}_{\p{\theta_1}, M} - \p{F}_{\p{\theta_2}, M}\|_{\Lip(\gamma-1)}$ can be bounded by a multiple of $\|\theta_1 - \theta_2\|$ on $U$. \medskip

    For all $\p{s}\in \ssol$, we have
    \begin{align} 
        \p{F}_{\p{\theta_1},M} (\p{s})- \p{F}_{\p{\theta_2},M} (\p{s}) = \tens_M (\p{s}) \circ \left(F_{\p{\theta_1},M} (\p{s}) - F_{\p{\theta_2},M}(\p{s}) \right).
    \end{align}
    Now by \Cref{lem:product_lipgamma},  $\|\p{F}_{\p{\theta_1},M} -\p{F}_{\p{\theta_2},M} \|_{\Lip(\gamma-1)} \leq K\ \|\tens_M\|_{\Lip(\gamma-1)} \|F_{\p{\theta_1},M} - F_{\p{\theta_2},M} \|_{\Lip(\gamma-1)}$ for some constant $K>0$ depending only on $\gamma$, $m$, and $q$. Therefore, it suffices to show that $\|F_{\p{\theta_1},M} - F_{\p{\theta_2},M} \|_{\Lip(\gamma-1)}$ can be bounded by some multiple of $\|\theta_1-\theta_2\|$. Since the extension operator in \Cref{thm:modify_lipgamma} is linear~\cite[p. 176]{singular_integrals}, we can consider  $F_{\p{\theta_1},M} - F_{\p{\theta_2},M}$ as the extension of the restriction $(F_{\p{\theta_1}} - F_{\p{\theta_2}})|_{K_M}$ to $K_M \coloneqq \{\p{s}\in \ssol: \|s\|\leq M\}$. So, by \Cref{thm:modify_lipgamma}, we only need to bound $\|(F_{\p{\theta_1}} - F_{\p{\theta_2}})|_{K_M}\|_{\Lip(\gamma-1)}$.\medskip

   Note that $F_{\theta_1, \theta_2} \coloneqq F_{\p{\theta_1}} - F_{\p{\theta_2}} \in \LL(\ssol, \LL(\drive, \ssol))$ is a linear function, and for all $\p{s},\p{s_1},\p{s_2}\in K_M$,
      \begin{align}
        \left\|F_{\theta_1,\theta_2}(\bs)\right\| \leq M \|\theta_1 - \theta_2\| \andd \left\| F_{\theta_1,\theta_2}(\bs_1) - F_{\theta_1,\theta_2}(\bs_2)\right\| \leq \|\theta_1 - \theta_2\| \|\bs_1 - \bs_2\|.
        \end{align}
    Moreover, the 2nd and higher derivatives of $F_{\theta_1,\theta_2}$ are identically 0. Therefore, as desired,
    \begin{align}
        \|(F_{\p{\theta_1}} - F_{\p{\theta_2}})|_{K_M}\|_{\Lip(\gamma-1)}\leq \max\{1,M\} \cdot \|\theta_1-\theta_2\|.
    \end{align}

    By similar arguments as above, we have a uniform bound for $\|\p{F}_{\p{\theta}, M}\|_{\Lip(\gamma)}$ over $U$. Therefore, by applying these bounds to~\eqref{eq:rde_locally_lipschitz}, $P$ must be locally Lipschitz, and by Rademacher's theorem, $P$ must be differentiable almost everywhere. 
     \end{proof}

     \begin{proof}[Proof of \Cref{prop: non-identifiability}]
    Let $\p{Y}$ be the $\cM(\|\theta_1\|)$-solution, i.e. the solution (in the sense of \Cref{thm: universal limit theorem 2}), to~\eqref{eq: identifiability 1}. Then for all $t\in [0,T]$ we have 
\begin{align}
    \p{Y}_{t}^{(m)} = \int_{0}^t \sum_{I\in \cW([m-1]_0,q')} \lambda^I \p{Y}_{s}^{I^-} d\p{Y}_{s}^{I^f}  = \langle \lambda, \p{Y}_{t}\rangle = H(\p{Y}_t),
\end{align} 
and therefore, $\p{Y}$ will also be the $\cM(\|\theta_1\|)$-solution, and thus, the solution to~\eqref{eq: identifiability 2}. But given the bound on $\p{X}$ and by \Cref{prop: N(theta) and M(theta)}, the $\cM(\|\theta_2\|)$-solution and the solution to~\eqref{eq: identifiability 2} are the same. So, $\p{Y}$ is also the $\cM(\|\theta_2\|)$-solution to~\eqref{eq: identifiability 2}.
\end{proof}

\clearpage
\section{Figures for \Cref{sec:experiments}} \label{appendix: figures for experiments}
This section contains the figures related to \Cref{sec:experiments}.
\begin{figure}[H]
  \centering

  \begin{subfigure}{0.96\textwidth}
  \centering
  \begin{tikzpicture}[every node/.style={align=center}]
    \definecolor{color1}{RGB}{75,16,48}
    \definecolor{color2}{RGB}{234, 66, 54}
    \definecolor{color3}{RGB}{252,188,0}
    \definecolor{color4}{RGB}{0,255,0}
    \definecolor{color5}{RGB}{187,187,187}
    \definecolor{color6}{RGB}{68,190,197}

    \tikzset{
    dot/.style={circle, draw=black, minimum size=10pt, inner sep=0pt}
      }

    \node[dot, fill=color1, label=right:{$\param^1$}] at (0,0) {};
    \node[dot, fill=color2, label=right:{$\param^2$}] at (1.5,0) {};
    \node[dot, fill=color3, label=right:{$\param^3$}] at (3,0) {};
    \end{tikzpicture}  
  \end{subfigure}

  \begin{subfigure}{0.96\textwidth}
    \centering
    \includegraphics[width=0.95\textwidth]{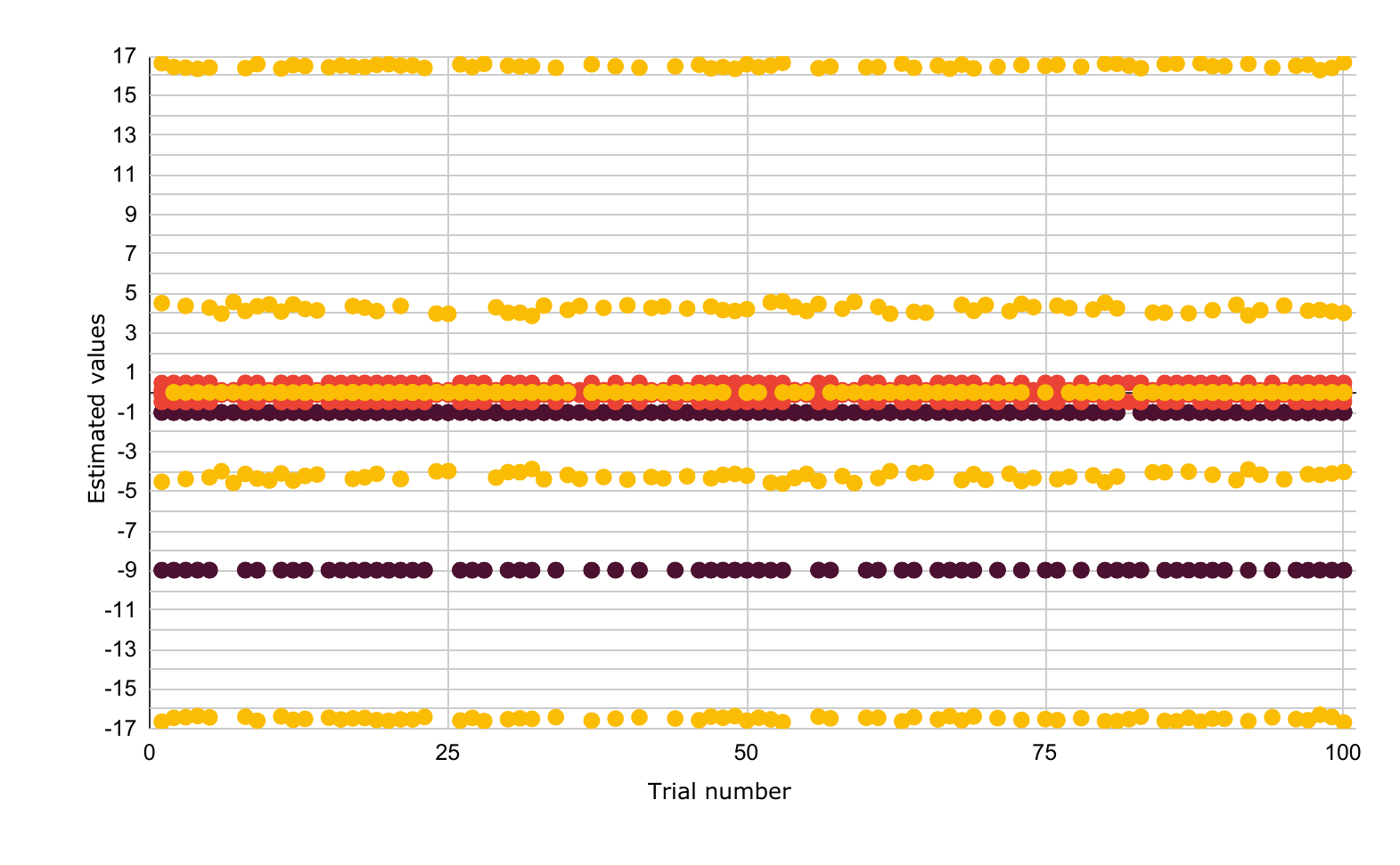}
    \vspace{-5mm}
    \caption{}
  \end{subfigure}
  ~
  
  \begin{subfigure}{0.96\textwidth}
    \centering
    \includegraphics[width=0.95\textwidth]{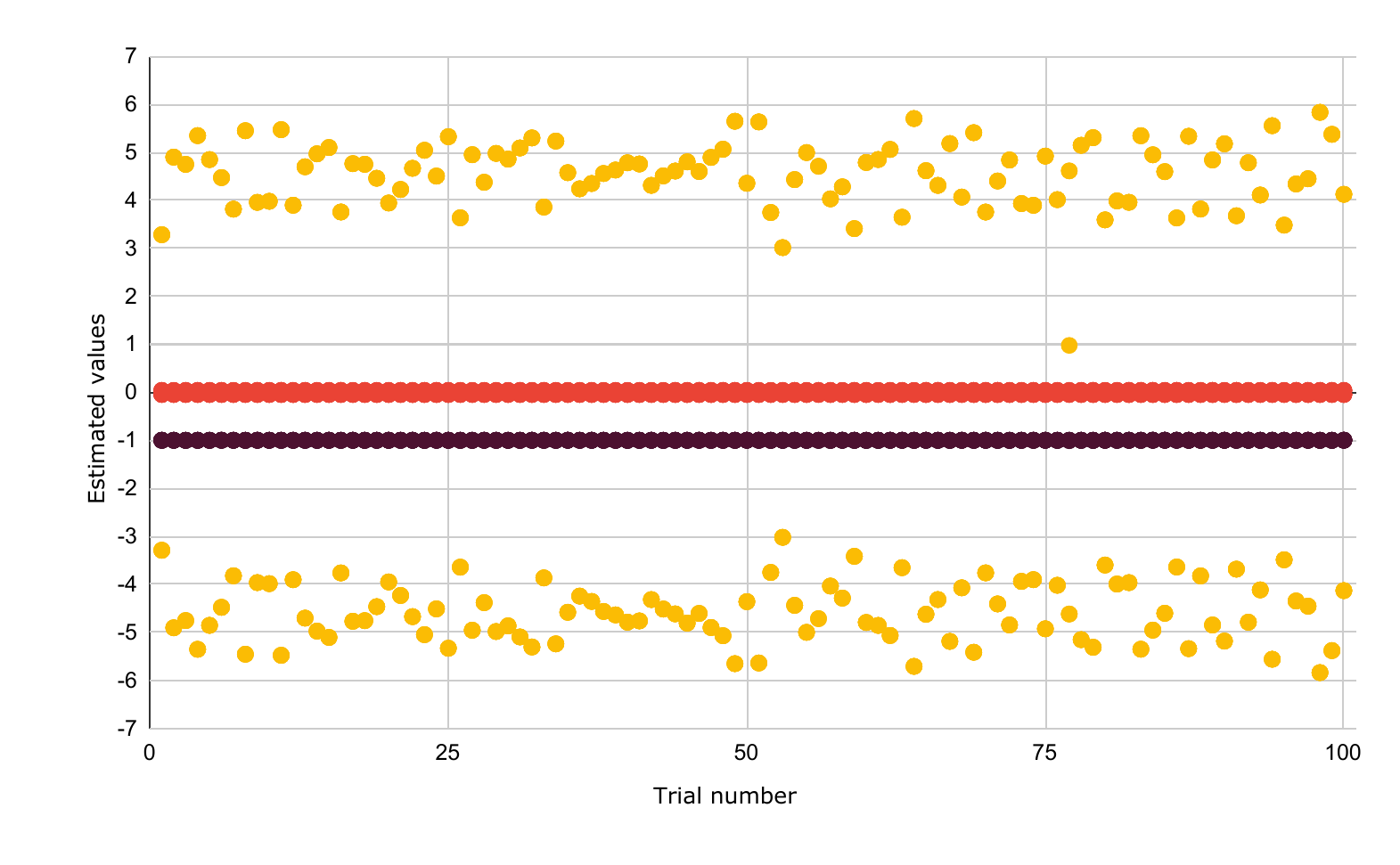}
    \vspace{-5mm}
    \caption{}
  \end{subfigure}
  ~

    \caption{The components of all the real solutions to the polynomial system~\eqref{eq: expected signature mathcing method - modified} over 100 trials in \Cref{experiment: 1 dimensional} corresponding to the sets of words (a) $\cW_1$ and (b) $\cW_2$, defined in~\eqref{eq: set of words 1}.}
    \label{fig: exp 1-all}
\end{figure}

\begin{figure}[H]
  \centering

  \begin{subfigure}{0.96\textwidth}
  \centering
  \begin{tikzpicture}[every node/.style={align=center}]
    \definecolor{color1}{RGB}{75,16,48}
    \definecolor{color2}{RGB}{234, 66, 54}
    \definecolor{color3}{RGB}{252,188,0}
    \definecolor{color4}{RGB}{0,255,0}
    \definecolor{color5}{RGB}{187,187,187}
    \definecolor{color6}{RGB}{68,190,197}

    \tikzset{
    dot/.style={circle, draw=black, minimum size=10pt, inner sep=0pt}
      }

    \node[dot, fill=color1, label=right:{$\param^1$}] at (0,0) {};
    \node[dot, fill=color2, label=right:{$\param^2$}] at (1.5,0) {};
    \node[dot, fill=color3, label=right:{$\param^3$}] at (3,0) {};
    \node[dot, fill=color4, label=right:{$\param^4$}] at (4.5,0) {};
    \node[dot, fill=color5, label=right:{$\param^5$}] at (6,0) {};
    \end{tikzpicture}  
  \end{subfigure}

  \begin{subfigure}{0.96\textwidth}
    \centering
    \includegraphics[width=0.95\textwidth]{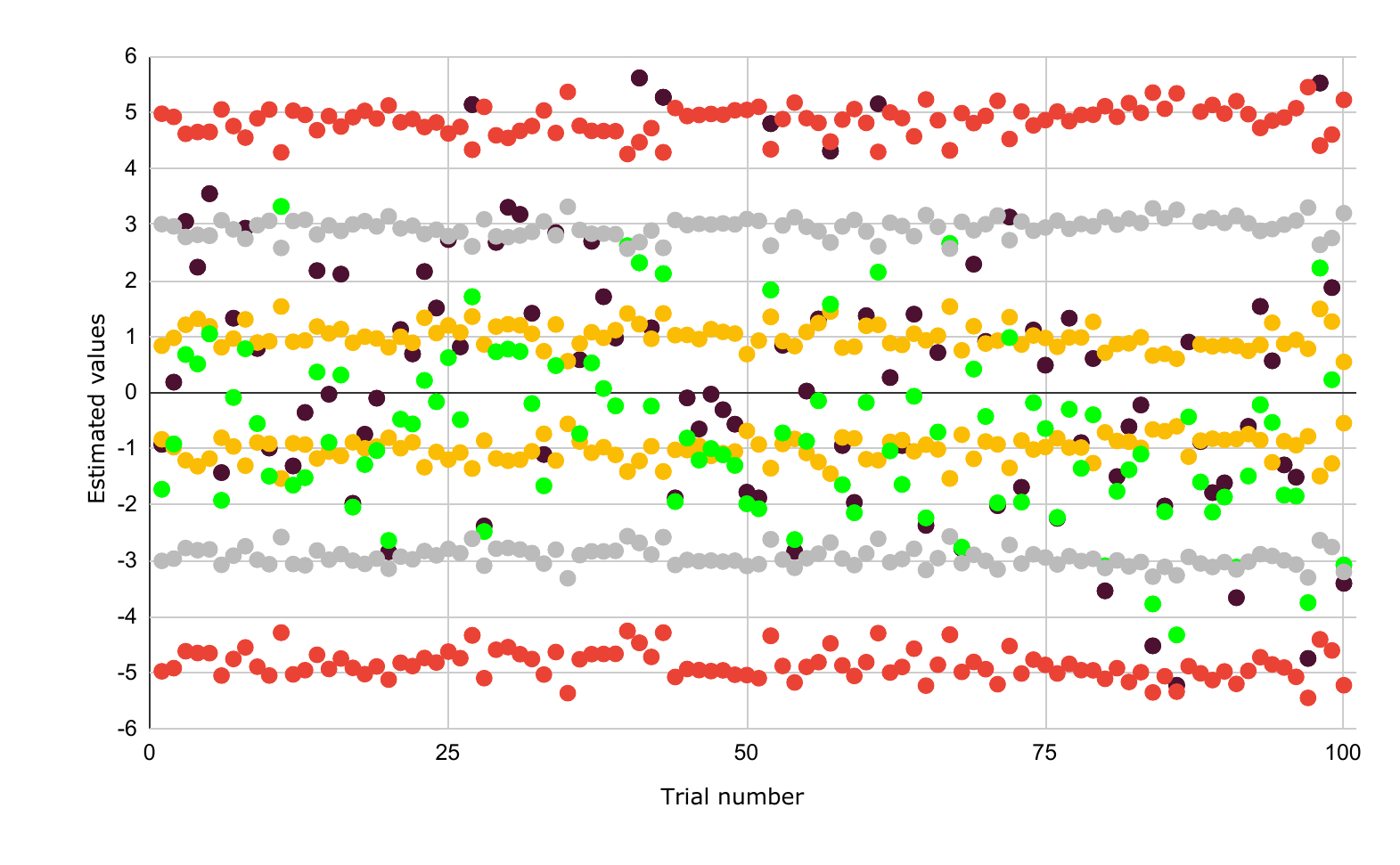}
    \vspace{-5mm}
    \caption{}
  \end{subfigure}
  ~
  
  \begin{subfigure}{0.96\textwidth}
    \centering
    \includegraphics[width=0.95\textwidth]{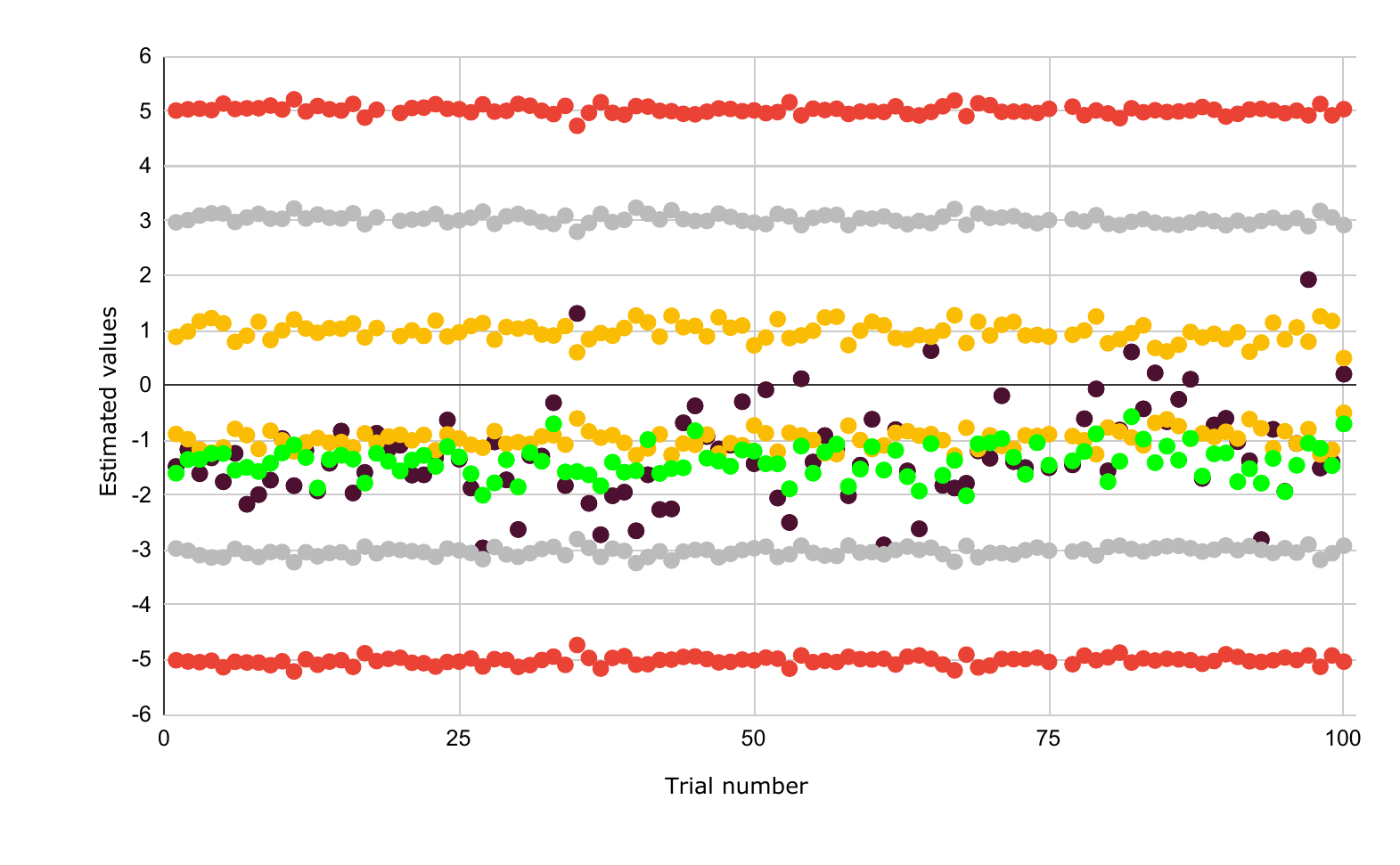}
    \vspace{-5mm}
    \caption{}
  \end{subfigure}
  ~

    \caption{The components of all the real solutions to the polynomial system~\eqref{eq: expected signature mathcing method - modified} over 100 trials in \Cref{experiment: 2 dimensional} corresponding to the sets of words (a) $\cW_1$ and (b) $\cW_2$, defined in~\eqref{eq: set of words 2}.}
    \label{fig: exp 2-all}
\end{figure}

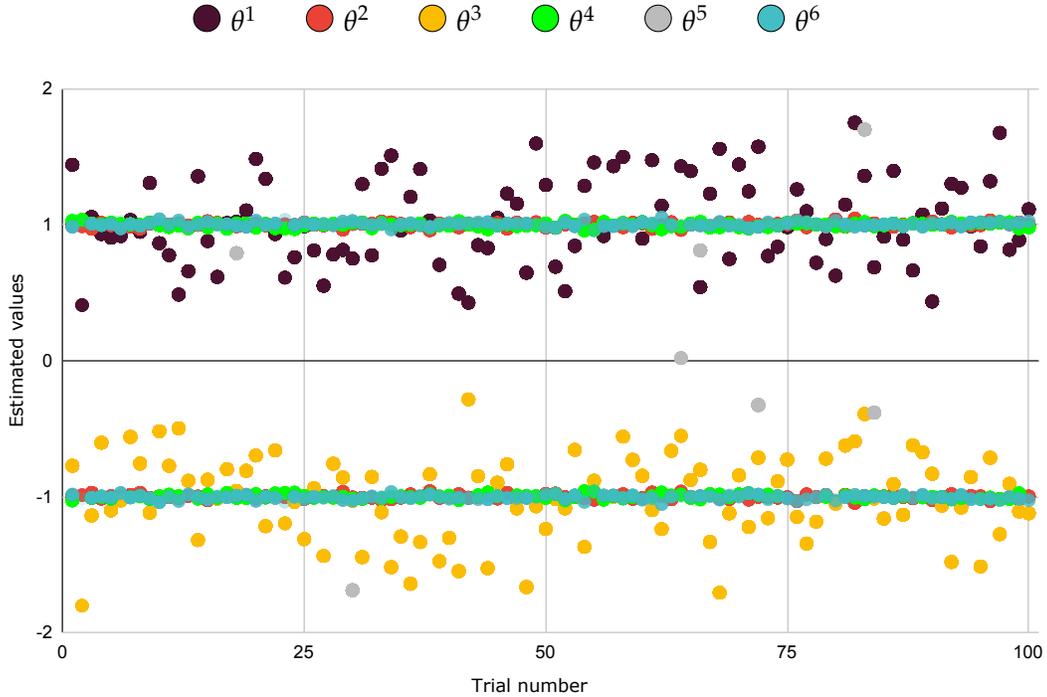
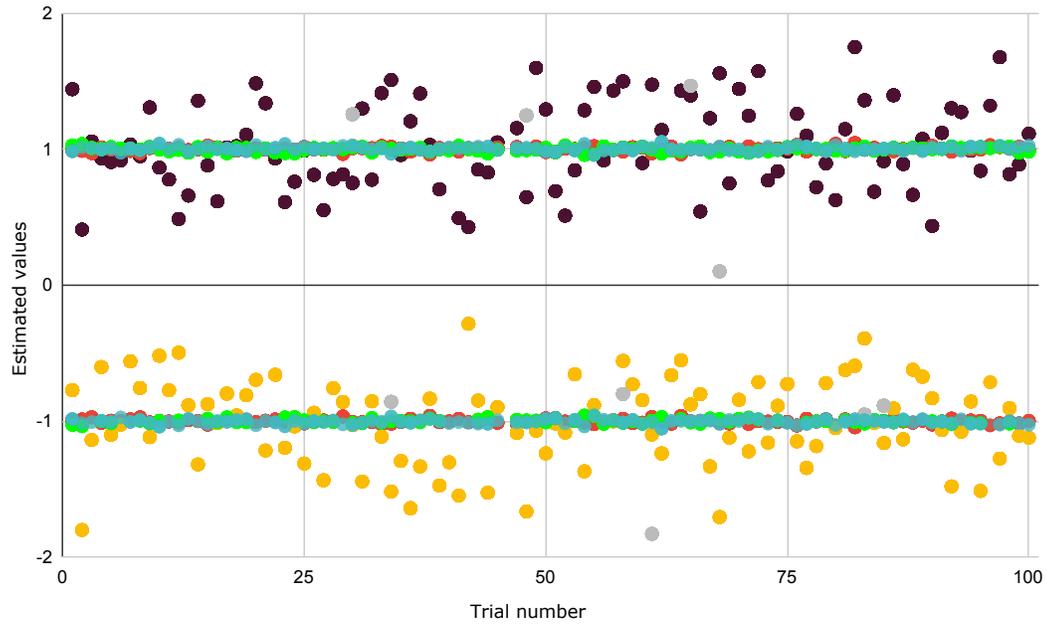
\begin{figure}[H]
  \centering

  \begin{subfigure}{0.96\textwidth}
  \centering
  \begin{tikzpicture}[every node/.style={align=center}]
    \definecolor{color1}{RGB}{75,16,48}
    \definecolor{color2}{RGB}{234, 66, 54}
    \definecolor{color3}{RGB}{252,188,0}
    \definecolor{color4}{RGB}{0,255,0}
    \definecolor{color5}{RGB}{187,187,187}
    \definecolor{color6}{RGB}{68,190,197}

    \tikzset{
    dot/.style={circle, draw=black, minimum size=10pt, inner sep=0pt}
      }

    \node[dot, fill=color1, label=right:{$\param^1$}] at (0,0) {};
    \node[dot, fill=color2, label=right:{$\param^2$}] at (1.5,0) {};
    \node[dot, fill=color3, label=right:{$\param^3$}] at (3,0) {};
    \node[dot, fill=color4, label=right:{$\param^4$}] at (4.5,0) {};
    \node[dot, fill=color5, label=right:{$\param^5$}] at (6,0) {};
    \node[dot, fill=color6, label=right:{$\param^6$}] at (7.5,0) {};
    \end{tikzpicture}  
  \end{subfigure}

  \begin{subfigure}{0.96\textwidth}
    \centering
    \includegraphics[width=0.95\textwidth]{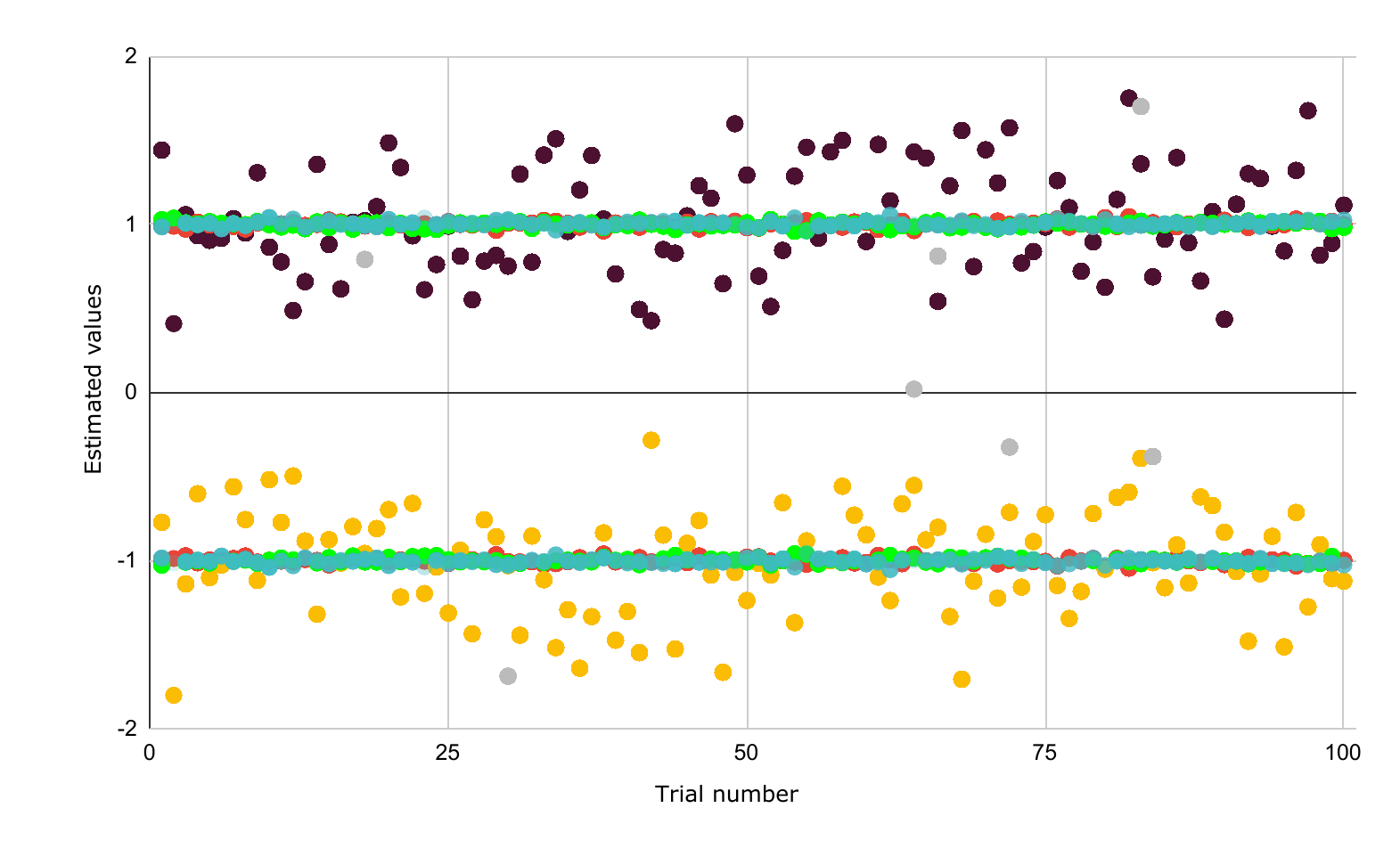}
    \vspace{-5mm}
    \caption{}
  \end{subfigure}
  ~
  
  \begin{subfigure}{0.96\textwidth}
    \centering
    \includegraphics[width=0.95\textwidth]{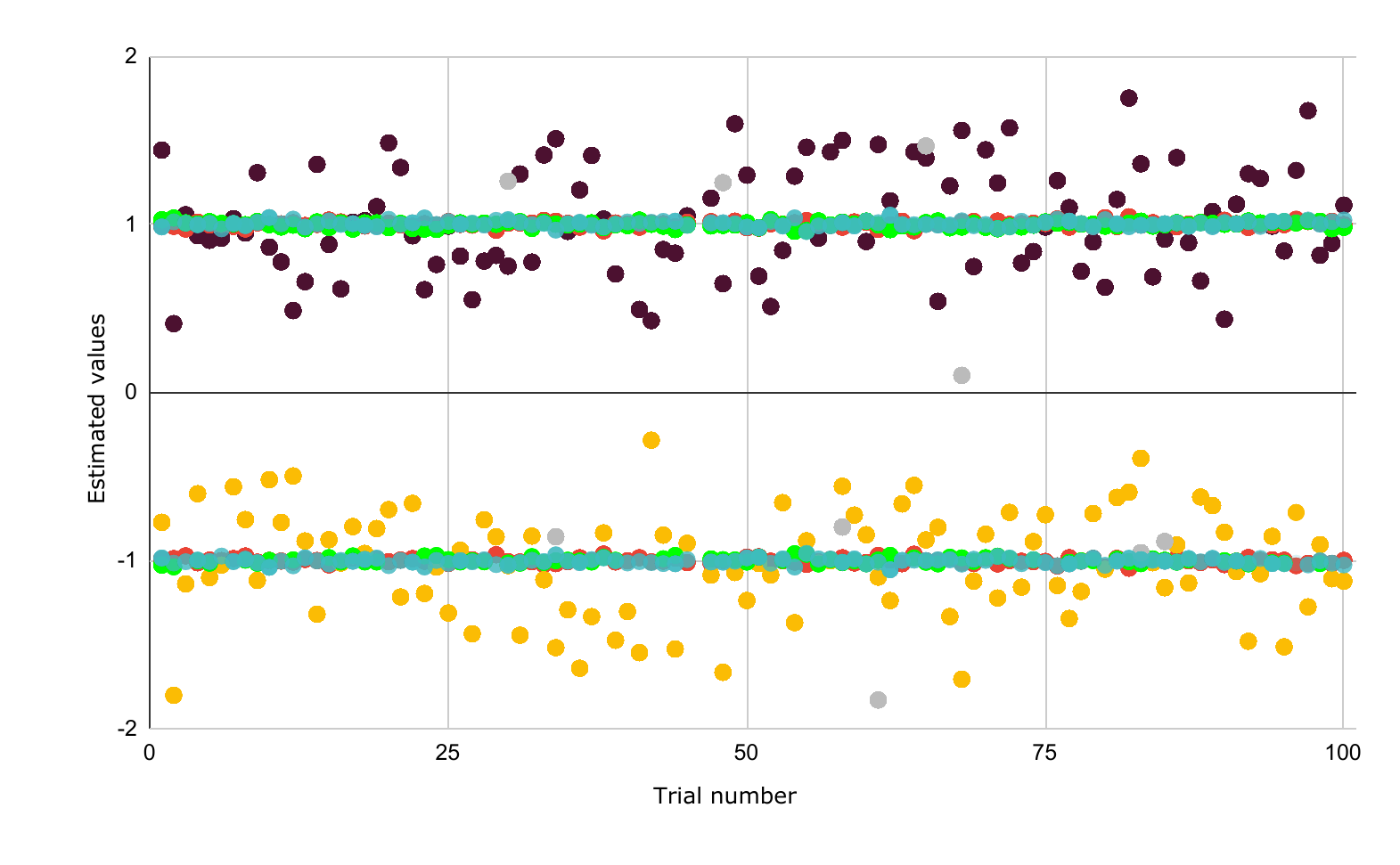}
    \vspace{-5mm}
    \caption{}
  \end{subfigure}
  ~

    \caption{The components of all the real solutions to the polynomial system~\eqref{eq: expected signature mathcing method - modified} over 100 trials in \Cref{experiment: 3 dimensional} corresponding to the sets of words (a) $\cW_5$ and (b) $\cW_6$, defined in~\eqref{eq: set of words 3}. Only a few of the estimated values for $\param^5$ fall within the interval $[-2,2]$, and thus, the rest are not visible in these plots. In each trial, many of the estimated values for $\param^2$, $\param^4$, and $\param^6$ nearly coincide.} 
    \label{fig: exp 3-all}
\end{figure}

\bibliographystyle{plain}
\bibliography{biblio}

\end{document}